\documentclass[a4paper,UKenglish,cleveref, autoref]{lipics-v2021}
\newcommand{\ERCagreement}{{\begin{minipage}{.56\textwidth}This paper is part of a project that has received funding from the European Research Council (ERC) under the European Union's Horizon 2020 research and innovation programme (grant agreement No 810115 -- {\sc Dynasnet}). \end{minipage}\hfill\begin{minipage}{.33\textwidth}\includegraphics[width=\textwidth]{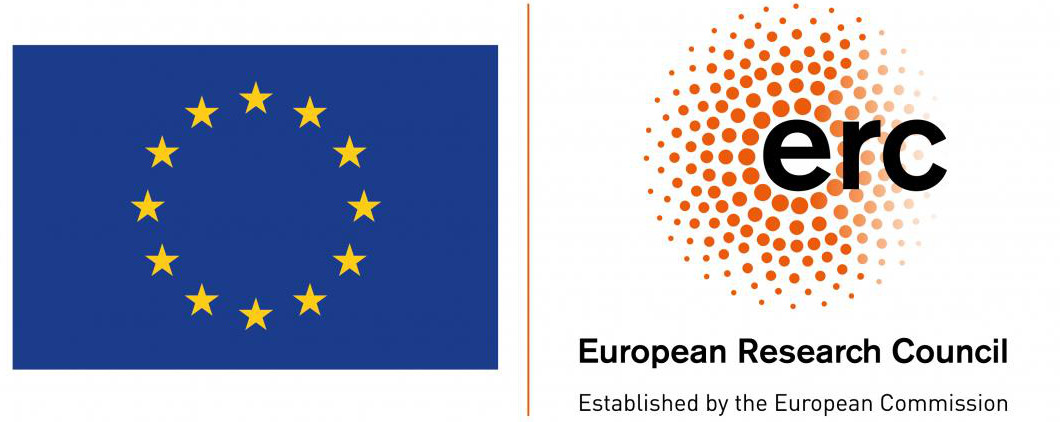}\end{minipage}\hfill}}

\title{Twin-width IV: ordered graphs and matrices}
\titlerunning{Twin-width IV: ordered graphs and matrices}%optional, please use title is longer than one line

\author{\'{E}douard Bonnet}{Univ Lyon, CNRS, ENS de Lyon, Université Claude Bernard Lyon 1, LIP UMR5668, France \and \url{http://perso.ens-lyon.fr/edouard.bonnet/}}{edouard.bonnet@ens-lyon.fr}{https://orcid.org/0000-0002-1653-5822}{}
\author{Ugo Giocanti}{Univ Lyon, CNRS, ENS de Lyon, Universit\'{e} Claude Bernard Lyon 1, LIP UMR5668, France}{ugo.giocanti@ens-lyon.fr}{}{}
\author{Patrice Ossona de Mendez}{Centre d'Analyse et de Mathématique Sociales CNRS UMR 8557, France \and and Computer Science Institute of Charles University (IUUK), Praha, Czech Republic \and \url{http://cams.ehess.fr/patrice-ossona-de-mendez/} }{pom@ehess.fr}{https://orcid.org/0000-0003-0724-3729}{}
\author{Pierre Simon}{UC Berkeley \and \url{http://www.normalesup.org/~simon/index.html}}{simon@math.berkeley.edu}{https://orcid.org/0000-0002-2923-8202}{}
\author{St\'{e}phan Thomass\'{e}}{Univ Lyon, CNRS, ENS de Lyon, Universit\'{e} Claude Bernard Lyon 1, LIP UMR5668, France \and \url{http://perso.ens-lyon.fr/stephan.thomasse/}}{stephan.thomasse@ens-lyon.fr}{}{}
\author{Szymon Toru\'nczyk}{University of Warsaw \and \url{https://www.mimuw.edu.pl/~szymtor/}}{szymtor@mimuw.edu.pl}{https://orcid.org/0000-0002-1130-9033}{}

\authorrunning{\'E. Bonnet, U. Giocanti, P. {Ossona de Mendez}, P. Simon, S. Thomassé, S. Toru\'nczyk}

\Copyright{Édouard Bonnet, Ugo Giocanti, Patrice {Ossona de Mendez}, Pierre Simon, Stéphan Thomassé, Szymon Toru\'nczyk}%TODO mandatory, please use full first names. LIPIcs license is "CC-BY";  http://creativecommons.org/licenses/by/3.0/

%TODO mandatory: Please choose ACM 2012 classifications from https://dl.acm.org/ccs/ccs_flat.cfm 


%TODO mandatory; please add comma-separated list of keywords

\category{}%optional, e.g. invited paper

\relatedversion{}%optional, e.g. full version hosted on arXiv, HAL, or other respository/website
\supplement{}%optional, e.g. related research data, source code, ... hosted on a repository like zenodo, figshare, GitHub, ...

\funding{\ERCagreement}

\acknowledgements{}%optional

\nolinenumbers %uncomment to disable line numbering

\hideLIPIcs  %uncomment to remove references to LIPIcs series (logo, DOI, ...), e.g. when preparing a pre-final version to be uploaded to arXiv or another public repository

\EventEditors{John Q. Open and Joan R. Access}
\EventNoEds{2}
\EventLongTitle{42nd Conference on Very Important Topics (CVIT 2016)}
\EventShortTitle{CVIT 2016}
\EventAcronym{CVIT}
\EventYear{2016}
\EventDate{December 24--27, 2016}
\EventLocation{Little Whinging, United Kingdom}
\EventLogo{}
\SeriesVolume{42}
\ArticleNo{23}

\usepackage[utf8]{inputenc}  

\usepackage[T1]{fontenc}
\usepackage{lmodern}

\usepackage[colorinlistoftodos,bordercolor=orange,backgroundcolor=orange!20,linecolor=orange,textsize=normalsize]{todonotes}

\usepackage{amsmath}  
\usepackage{amssymb}     
\usepackage{bbm}
\usepackage{tkz-graph}
\usepackage{complexity}
\usepackage{algorithm}
\usepackage{algorithmic}
\usepackage[all]{xy}
\usepackage{comment}

\usepackage{bm}
\usepackage{hyperref}
\usepackage{mathrsfs}
\usepackage{mathtools} 
\usepackage{paralist}
\usepackage{fixmath}
\usepackage{interval}
\usepackage{dsfont}

\Crefname{figure}{Figure}{Figures}

\crefformat{equation}{\textup{#2(#1)#3}}
\crefrangeformat{equation}{\textup{#3(#1)#4--#5(#2)#6}}
\crefmultiformat{equation}{\textup{#2(#1)#3}}{ and \textup{#2(#1)#3}}
{, \textup{#2(#1)#3}}{, and \textup{#2(#1)#3}}
\crefrangemultiformat{equation}{\textup{#3(#1)#4--#5(#2)#6}}%
{ and \textup{#3(#1)#4--#5(#2)#6}}{, \textup{#3(#1)#4--#5(#2)#6}}{, and \textup{#3(#1)#4--#5(#2)#6}}

\Crefformat{equation}{#2Equation~\textup{(#1)}#3}
\Crefrangeformat{equation}{Equations~\textup{#3(#1)#4--#5(#2)#6}}
\Crefmultiformat{equation}{Equations~\textup{#2(#1)#3}}{ and \textup{#2(#1)#3}}
{, \textup{#2(#1)#3}}{, and \textup{#2(#1)#3}}
\Crefrangemultiformat{equation}{Equations~\textup{#3(#1)#4--#5(#2)#6}}%
{ and \textup{#3(#1)#4--#5(#2)#6}}{, \textup{#3(#1)#4--#5(#2)#6}}{, and \textup{#3(#1)#4--#5(#2)#6}}

\crefformat{enumi}{#2(\textit{#1})#3}

\usepackage{xspace}
\usepackage{tikz}
\usepackage{multicol}
\usepackage{enumitem}

\usetikzlibrary{arrows}
\usetikzlibrary{patterns}
\usetikzlibrary{calc}
\usetikzlibrary{shapes}
\usetikzlibrary{positioning}
\usetikzlibrary{math}
\usetikzlibrary{external}
\usepackage{ifthen}

\usepackage[scr=boondox,scrscaled=1.05]{mathalfa}

\makeatletter
\newtheorem*{rep@theorem}{\rep@title}
\newcommand{\newreptheorem}[2]{%
\newenvironment{rep#1}[1]{%
 \def\rep@title{#2 \ref{##1}}%
 \begin{rep@theorem}}%
 {\end{rep@theorem}}}
\makeatother

\newreptheorem{theorem}{Theorem}
\newreptheorem{lemma}{Lemma}

\renewcommand{\neq}{{\not=}}
\renewcommand{\geq}{\geqslant}
\renewcommand{\leq}{\leqslant}

\renewcommand{\le}{\leq}
\renewcommand{\ge}{\geq}
\renewcommand{\phi}{\varphi}
\newcommand{\interp}{\mathsf} %interpretations
\newcommand{\CC}{\mathscr C}
\newcommand{\DD}{\mathscr D}
\newcommand{\sym}{{\rm{s}}}
\newcommand{\six}{\set{{=, \neq, \le_R, \ge_R,} {\le_C, \ge_C}}}
\newcommand{\sax}{\set{{=, \neq, \le_l, \ge_l,} {\le_r, \ge_r}}}

\newcommand{\set}[1]{\ensuremath{\{#1\}}} %\set{a,b,c}
\newcommand{\setof}[2]{\set{#1\mid#2}} %\setof{x}{x>5}
\newcommand{\from}{\colon} %function $f\from X\to Y$
\newcommand{\str}[1]{\mathbf{#1}} %structure

\DeclareMathOperator{\sign}{\mathsf{sign}}
\DeclareMathOperator{\ot}{ot}

\DeclareMathOperator{\mt}{mt}

\DeclareMathOperator{\reduct}{\text{\sf Reduct}}
\newcommand{\Nn}{\mathbb{N}}
\newcommand*\sg[1]{\left \{ #1 \right \}}

\newcommand{\bipri}[2]{\mathrm{b}^{(#2)}_2(#1)}
\newcommand{\bipramm}[2]{\mathrm{b}_{#2}(#1)}
\newcommand{\ramm}[2]{\mathrm{R}_{#2}(#1)}
\newcommand{\gridramm}[2]{\mathrm{g}_{#2}(#1)}

\theoremstyle{definition}

\theoremstyle{remark}

\newcommand{\LFP}{\text{\sf{LFP}}}

\DeclareMathOperator{\tww}{tww}

\DeclareMathOperator{\gr}{gr}

\newcommand{\utr}{\mathbf{U}_k}
\newcommand{\ltr}{\mathbf{L}_k}
\newcommand{\id}{\mathbf{I}_k}
\newcommand{\ao}{\mathbf{1}_k}
\newcommand{\az}{\mathbf{0}_k}
\newcommand{\rami}{\mathcal N_k}
\newcommand{\row}{\text{rows}}
\newcommand{\col}{\text{cols}}
\newcommand{\kop}{\mbox{$k$-overlapping} partition\xspace}
\newcommand{\kops}{\mbox{$k$-overlapping} partitions\xspace}
\newcommand{\opar}[1]{\mbox{$#1$-overlapping} partition\xspace}
\newcommand{\opars}[1]{\mbox{$#1$-overlapping} partitions\xspace}
\newcommand{\Oof}{O}

\begin{document}

\maketitle

\begin{abstract}
  We establish a list of characterizations of bounded twin-width for hereditary, totally ordered binary structures.
  This has several consequences.
  First, it allows us to show that a (hereditary) class of matrices over a finite alphabet either contains at least $n!$ matrices of size $n \times n$, or at most $c^n$ for some constant $c$.
  This generalizes the celebrated Stanley-Wilf conjecture/Marcus-Tardos theorem from permutation classes to any matrix class over a finite alphabet, answers our small conjecture [SODA '21] in the case of ordered graphs, and with more work, settles a question first asked by Balogh, Bollob\'as, and Morris [Eur. J. Comb. '06] on the growth of hereditary classes of ordered graphs.
  Second, it gives a fixed-parameter approximation algorithm for twin-width on ordered graphs.
  Third, it yields a full classification of fixed-parameter tractable first-order model checking on hereditary classes of ordered binary structures. 
  Fourth, it provides a model-theoretic characterization of classes with bounded twin-width.
\end{abstract}

%\tableofcontents

\subparagraph{Acknowledgments}

We thank Colin Geniet, Eunjung Kim, Jarik Nešetřil, Sebastian Siebertz, Michał Pilipczuk, and Rémi Watrigant for fruitful discussions.

\section{Introduction}
A common goal in combinatorics, structural graph theory, parameterized complexity, and finite model theory 
is to understand \emph{tractable} classes of graphs, or other structures. Depending on the context, tractability may mean e.g.
 few structures of any given size, 
 small chromatic number, efficient algorithms, for example for the clique problem; or structural properties of sets definable by logical formulas.

Twin-width is a recently introduced graph width parameter~\cite{twin-width1,twin-width2,twin-width3} 
which, despite its generality, ensures many of those properties.
Intuitively, a graph  has twin-width $d$ if it can be constructed by merging larger and larger parts so that at any moment, every part has a non-trivial interaction
with at most $d$ other parts
(two parts have a trivial interaction if either no edges, or all edges span across the two parts).

Many well-studied graph classes have bounded twin-width: planar graphs, and more generally, any class of graphs excluding a fixed minor; cographs, and more generally, any class of bounded clique-width.
Twin-width can be generalized to structures equipped with several binary relations. Then posets of bounded width, tree orders, and permutations omitting a fixed permutation also have bounded twin-width. 

Classes of bounded twin-width enjoy many   remarkable properties of combinatorial, algorithmic, and logical nature.
For instance, classes of bounded twin-width are small (contain $n!\cdot 2^{\Oof(n)}$ graphs with vertex set $\set{1,\ldots,n}$), are $\chi$-bounded (the chromatic number is bounded in terms of the clique number)~\cite{twin-width2}, and have the NIP property from model theory (every  first-order formula $\phi(x,y)$ defines a binary relation of bounded VC-dimension).
Furthermore, model checking first-order logic (\FO) is fixed-parameter tractable (\FPT) on classes of bounded twin-width, assuming a \emph{witness} of $G$ having bounded twin-width is provided.
This means that there is an algorithm which, given an \FO~sentence $\phi$, a graph $G$ together with a witness that its twin-width is at most $d$, decides whether $\phi$ holds in $G$ in time $f(\phi,d)\cdot |G|^c$ for some computable function $f$ and universal constant $c$.

For each of the classes $\CC$ mentioned above there is actually an algorithm which, given a graph $G\in \CC$, computes the required witness of low twin-width, in polynomial time~\cite{twin-width1}.
Hence \FO~model checking is \FPT~on these classes (without requiring the witness), generalizing many previous results, while it is $\AW[*]$-hard (thus, unlikely to be \FPT) on the class of all graphs~\cite{Downey96}.
It is however unknown whether such witness can be computed efficiently for every class $\CC$ of bounded twin-width.

It transpires from the mere definitions that every graph can be equipped with some total order, resulting in an ordered graph of the same twin-width~\cite{twin-width1}.
Such orders, while elusive to efficiently find, are crucial to most of the combinatorial and algorithmic applications.
This suggests that \emph{ordered graphs} of bounded twin-width are a more fundamental object than unordered graphs of bounded twin-width.

\subparagraph{Main result.}
% \label{sec:newmain}
Our main result,~\cref{thm:main} below, gives multiple characterizations of hereditary (i.e., closed under induced substructures) classes of ordered graphs of bounded twin-width, connecting notions from various areas of mathematics and theoretical computer science -- enumerative combinatorics, model theory, parameterized complexity, matrix theory, and graph theory -- and solving several open problems on the way.
Furthermore, we show that if a class $\CC$ of ordered graphs has bounded twin-width, then for each $G\in\CC$, a witness that $G$ has twin-width bounded by a constant can be computed in polynomial time.
Consequently, \FO~model checking is \FPT~on $\CC$. We also prove that the converse holds, under common complexity-theoretic assumptions.

 \tikzexternaldisable
\begin{figure}\centering
  \begin{tikzpicture}[vertex/.style={draw,fill,circle,inner sep=0.03cm}]
    \def\n{6}
    \def\s{1.2}
    \def\z{0.6}
    \foreach \i in {1,...,\n}{
       \node at (-0.5,-\i * \z) {\i} ;
    }

    \foreach \symb/\symbp/\xsh/\sign in {{=}/{=}/0/=,{<}/{=}/3/{\leq},{>}/{=}/6/{\geq},{<}/{>}/9/{\neq}}{
      \begin{scope}[xshift=\xsh cm]
        \node at (\s / 2,-\n * \z-0.5) {$\sign$} ;
    \foreach \i in {1,...,\n}{
      \node[vertex] (l\i) at (0,-\i * \z) {} ;
      \node[vertex] (r\i) at (\s,-\i * \z) {} ;
    }
    \foreach \i in {1,...,\n}{
      \foreach \j in {1,...,\n}{
        \ifthenelse{\i \symb \j}{\draw (l\i) -- (r\j);}{}
        \ifthenelse{\i \symbp \j}{\draw (l\i) -- (r\j);}{} 
      }
    }
      \end{scope}
    }
  \end{tikzpicture}
  \caption{The four bipartite graphs $G^n_=$, $G^n_\le$, $G^n_\ge$, and $G^n_{\neq}$, for $n=6$.}
  \label{fig:ladders1}
\end{figure}
 \tikzexternalenable

 \medskip
We now briefly discuss some notions which are involved in our characterization.
One of them involves certain forbidden patterns, which are obfuscated matchings between ordered sets, defined as follows.
Fix a number $n\ge 1$ and 
let $L$ and $R$ be two copies of $[n]=\set{1,\ldots,n}$.
Consider the four bipartite graphs $G_=^n,G_\le^n,G_\ge^n,G_{\neq}^n$ with vertices $L\cup R$  corresponding to the binary relations $=$, $\le$, $\ge$, $\neq$
on $L\times R$, as depicted in~\cref{fig:ladders1}. 
Fix parameters $P\in\set{=,\neq,\le,\ge}$,
$S\in\set{l,r}$ and $\lambda,\rho\in\set{0,1}$, and define $\mathscr M_{P_S,\lambda,\rho}$ as the hereditary closure of the class of all ordered graphs obtained from the graphs $G_P^n$, for $n\ge 1$, as follows:
\begin{itemize}
  \item if $S=l$, order $L\cup R$ by ordering $L\simeq [n]$ as usual, followed by the vertices of $R$ ordered arbitrarily,
  \item if $S=r$, order $L\cup R$ by ordering $L$ arbitrarily, followed by the vertices of $R\simeq [n]$ in reverse order,
  \item if $\lambda=1$, create a clique on $L$, otherwise $L$ remains an independent set,
  \item if $\rho=1$, create a clique on $R$, otherwise $R$ remains an independent set.
\end{itemize}
As $\mathscr M_{=_l,\lambda,\rho}$ and $\mathscr M_{=_r,\lambda,\rho}$ are the same class,  we denote it by $\mathscr M_{=,\lambda,\rho}$.
Similarly $\mathscr M_{\neq_l,\lambda,\rho}=\mathscr M_{\neq_r,\lambda,\rho}$ is denoted $\mathscr M_{\neq,\lambda,\rho}$.
Altogether there are 24 classes: $\mathscr M_{\sym,\lambda,\rho}$ for each $\sym\in\sax$ and $\lambda,\rho\in\set{0,1}$.
For example, $\mathscr M_{=,0,0}$ consists  of all induced subgraphs of ordered matchings with vertices $a_1<\ldots<a_n<b_1<\ldots<b_n$,
whose edges form a matching between the $a_i$'s and the $b_j$'s, while $\mathscr M_{\neq,1,1}$ is the class of their edge complements. See~\cref{fig:ladders} for another representation of those classes.

\begin{figure}\centering
  \includegraphics[scale=0.5,page=7]{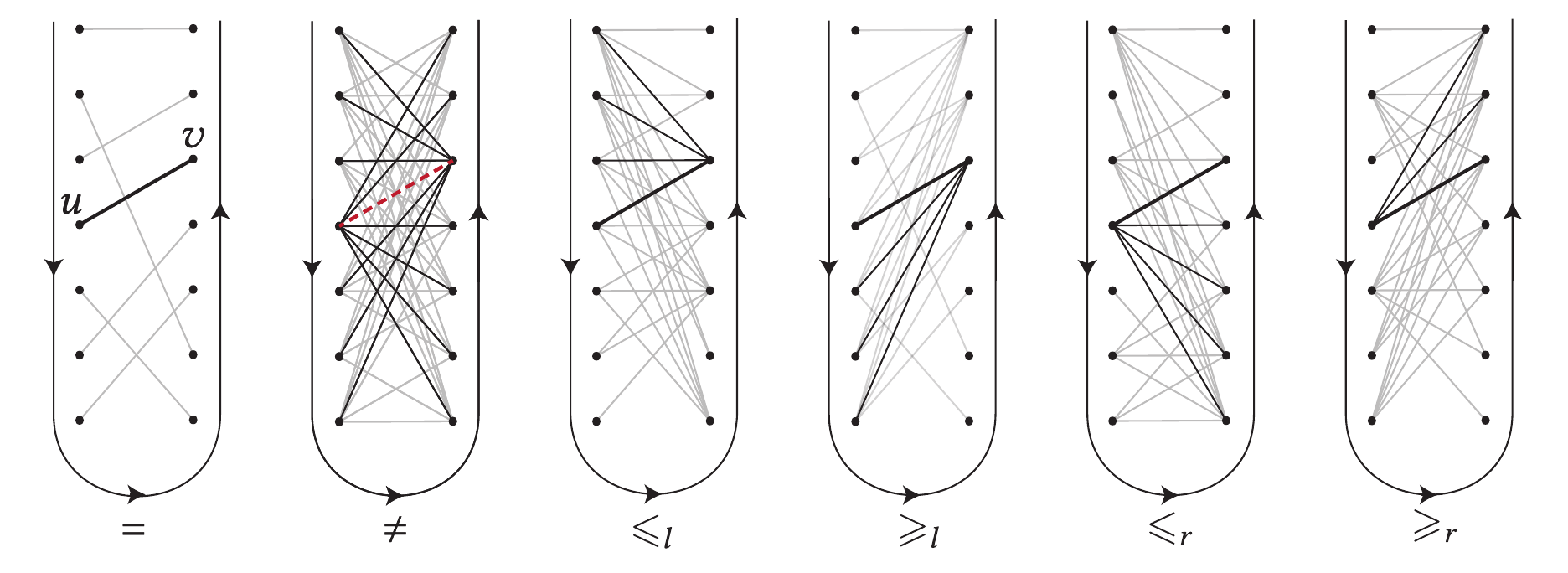}
  \caption{
    The six ordered graphs of $\mathscr M_{\sym,0,0}$ for $\sym\in\sax$ corresponding to the same ordered matching of $\mathscr M_{=,0,0}$ represented to the left.
    In each ordered graph the black edges are those implied by the bold edge $uv$ in the matching.
    To picture the other classes $\mathscr M_{\sym,\rho,\lambda}$ with $(\lambda,\rho) \neq (0,0)$, one just needs to turn the left part and/or the right part of each graph into cliques.}
  \label{fig:ladders}
\end{figure}

In addition to these 24 classes, we will need to define one more class, $\mathscr P$. 
An  \emph{ordered permutation graph} associated 
with a permutation $\pi$ of $[n]=\set{1,\ldots,n}$
is the ordered graph with vertices $[n]$ ordered naturally, such that two vertices $i<j$ are adjacent if and only if $\pi(i)>\pi(j)$.
Let $\mathscr P$ be the class of all ordered permutation graphs.

Another characterization is in terms of adjacency matrices of the considered ordered graphs.
% Say that a matrix $M$ over a finite alphabet has \emph{combinatorial rank} at least $k$ if it  has at least $k$ distinct rows or  at least $k$ distinct columns\footnote{If the alphabet is a finite field, then the combinatorial rank of a matrix is  is bounded in terms of the usual rank of $M$, and vice-versa.}.
A \emph{$d$-division} of a matrix is a partition of its entries into $d^2$ zones using $d-1$ vertical lines and $d-1$ horizontal lines.
A \emph{rank-$k$ $d$-division} is a $d$-division where each zone has at least $k$ non-identical rows or at least $k$ non-identical columns.
A matrix has \emph{grid rank} at least $k$ if it has a rank-$k$ $k$-division, simply called rank-$k$ division.
Note that this notion depends on the order of rows and columns in the matrix.
The grid rank of an ordered graph is the grid rank of its adjacency matrix along the order of the graph.

A further characterization involves (simple first-order) \emph{interpretations}, which is a notion originating in logic.
Interpretations are a means of producing new structures out of old ones, using formulas.
The new structure has the same domain as the old one (or a subset of it, defined by a formula $\delta(x)$) while each of its relations is defined by a formula $\phi(\overline{x})$ interpreted in the old structure.
For example, there is an interpretation which transforms a given graph $G$ into its edge complement (using the formula $\neg E(x,y)$), and an interpretation which transforms $G$ into its square (using the formula $\exists z.E(x,z) \land E(z,y)$).
Transductions are a similar notion, but additionally allow to arbitrarily color the old structure before applying the interpretation and then use the colors in the formulas. 
Say that $\CC$ \emph{interprets} the class of all graphs if there is an interpretation $\mathsf I$ such that every (finite) graph $G$ can be obtained as the result of $\mathsf I$ applied to some structure in $\CC$.
Replacing interpretations with transductions, we say that $\CC$ \emph{transduces} the class of all graphs.
It is known that no class of bounded twin-width transduces the class of all graphs~\cite{twin-width1}.

Finally, the \emph{growth} of a class of structures $\CC$
is the function counting, for a given $n$, the number of $n$-element structures in $\CC$, up to isomorphism.
It was previously shown that every class of ordered graphs of bounded twin-width has growth $2^{\Oof(n)}$~\cite{twin-width2}.

\medskip

 We now state our main result, characterizing hereditary classes of ordered graphs with bounded twin-width.
 It provides a dichotomy result for all such classes: Either they have bounded twin-width, and are therefore well-behaved, or otherwise, they are very untamable. 

\begin{theorem}\label{thm:main}
	Let $\CC$ be a hereditary class of ordered graphs.
	Then either $\CC$ satisfies conditions~(i)-(v), or $\CC$ satisfies conditions~(i')-(v') below:
	\begin{multicols}{2}%		
	\begin{enumerate}[label={(\roman*)},ref=\roman*]
			\item\label{git:tww} $\CC$ has bounded twin-width			
			 \item\label{git:matrix} $\CC$ has bounded~grid rank
			\item\label{git:small} $\CC$ has growth $2^{\Oof(n)}$
			
			 \item\label{git:transduce} $\CC$ does not transduce the class of all graphs
			 \item\label{git:easy} \FO~model checking  is \FPT~on $\CC$
			
	\end{enumerate}
	\begin{enumerate}[label={(\roman*')},ref=\roman*']
	\item\label{git:utww} 	$\CC$ has unbounded twin-width
	\item\label{git:patterns} 
	  $\CC$ contains $\mathscr P$ or one of the 24 classes $\mathscr M_{\sym,\lambda,\rho}$
	 \item\label{git:big} $\CC$ has growth at least $\sum_{k=0}^{\lfloor n/2 \rfloor} {n \choose 2k}k!\ge\lfloor \frac n 2\rfloor!$
	 
	  \item\label{git:interp} $\CC$ interprets the class of all graphs
	  \item\label{git:hard} \FO~model checking  is $\AW[*]$-hard on~$\CC$.
\end{enumerate}
		\end{multicols}
\end{theorem}

The above result connects notions from graph theory~\cref{git:tww}, enumerative combinatorics~\cref{git:small},\cref{git:big}, parameterized complexity~\cref{git:easy},\cref{git:hard}, model theory~\cref{git:transduce},\cref{git:interp}, matrix theory~\cref{git:matrix} and Ramsey theory~\cref{git:patterns}.
The lower bound in~\cref{git:big} is optimal.

\cref{thm:main} is proved in greater generality for arbitrary classes of ordered, binary structures.
We prove an analogous result for classes $\mathcal M$ of $0, 1$-matrices (and more generally matrix classes over finite alphabets). 
In this result (see \cref{thm:equiv}), the lower bound on the number of $n\times n$-matrices in $\mathcal M$ in \cref{git:big} is replaced by $n!$, and the 25 classes in~\cref{git:patterns} are reduced to six classes $\mathcal F_\sym$ of matrices, indexed by a single parameter $\sym \in \six$ (see~\cref{fig:six-unavoidable}).
Those are the class of all permutation matrices, the class obtained from permutation matrices by exchanging $0$'s with $1$'s, and four classes obtained from permutation matrices by propagating each 1 entry downward/upward/leftward/rightward, respectively.

 \tikzexternaldisable 
  \begin{figure}[h!]
    \centering
    \begin{tikzpicture}[scale = .26]

   \foreach \symb/\b/\xsh/\ysh in {{==}/0/-30/-10,{!=}/0/-20/-10, {<=}/0/-10/-10,{>=}/0/0/-10,{<=}/1/10/-10,{>=}/1/20/-10}{
   \begin{scope}[xshift=\xsh cm,yshift=\ysh cm]     
    \foreach \i in {0,...,8}{
      \foreach \j in {0,...,8}{
        \pgfmathsetmacro{\ip}{\i+1}
        \pgfmathsetmacro{\jp}{\j+1}
        \pgfmathsetmacro{\col}{ifthenelse(\b==1,ifthenelse(\i \symb mod(\j,3)*3+floor(\j/3),"black","white"),ifthenelse(\j \symb mod(\i,3)*3+floor(\i/3),"black","white"))}
        \fill[\col] (\i,\j) -- (\i,\jp) -- (\ip,\jp) -- (\ip,\j) -- cycle;
      }
    }
    \draw[purple] (0, 0) grid (9, 9);
   \end{scope}
   }
    \end{tikzpicture}
    \caption{The matrices in $\mathcal F_=, \mathcal F_{\neq}, \mathcal F_{\le_R}, \mathcal F_{\ge_R}, \mathcal F_{\le_C}, \mathcal F_{\ge_C}$ (from left to right) for the same permutation matrix (the one to the left).
      The 1 entries are represented in black, the 0 entries, in white.
      As is standard with permutation patterns, we always place the first row of the matrix at the bottom.
      }
    \label{fig:six-unavoidable}
  \end{figure}
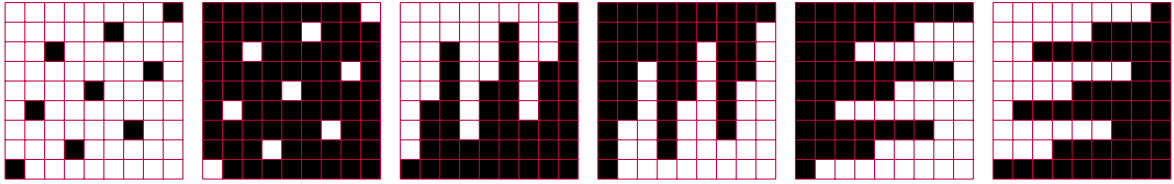
  \tikzexternalenable

As a consequence or by-product of~\cref{thm:main}, we settle a handful of questions in combinatorics and algorithmic graph theory.  
The main by-product is an approximation algorithm for twin-width in totally ordered binary structures.

\begin{theorem}\label{thm:approx-tww}
  There is a fixed-parameter algorithm that, given a totally ordered binary structure $G$ of twin-width~$k$, outputs a witness that $G$ has twin-width at most $2^{O(k^4)}$.
\end{theorem}

 As our second main result, we provide further characterizations of bounded twin-width classes in terms of model-theoretic notions, but which also transpire in algorithmic and structural graph theory.
We consider arbitrary \emph{monadically NIP} classes  of relational structures, which are not necessarily finite, ordered, or binary.
Those can be equivalently characterized as classes which do not transduce the class of all finite graphs.
They include all graph classes of bounded twin-width (with or without an order), but also all transductions of \emph{nowhere dense} classes \cite{sparsity}, such as classes of bounded maximum degree.

The following theorem generalizes some notions and implications appearing in \cref{thm:main}.
	
\begin{theorem}\label{thm:summary-mt0}
	For any class of structures $\CC$, consider the following statements:
	\begin{enumerate}[label={(\arabic*)},ref=\arabic*]
		\item\label{qit:trans} $\CC$ does not transduce the class of all graphs,
		\item\label{qit:mNIP} $\CC$ is monadically NIP,
		\item\label{qit:grids} $\CC$ does not {define large grids} (see~\cref{def:grids}),
		\item\label{qit:1-dim} $\CC$ is $1$-dimensional (see~\cref{def:1-dim}),
		\item\label{qit:restrained} $\CC$ is a {restrained} class (see~\cref{def:restrained}).
	\end{enumerate}
    Then the implications $\cref{qit:trans}\leftrightarrow\cref{qit:mNIP}\leftrightarrow\cref{qit:grids}\leftrightarrow\cref{qit:1-dim}\rightarrow\cref{qit:restrained}$ hold.
	For hereditary classes of binary, ordered structures, the above conditions are all equivalent to
		 $\CC$ having bounded twin-width.
\end{theorem}
Defining large grids generalizes the property of containing one of the classes $\mathscr M_{\sym,\lambda,\rho}$ or $\mathscr P$
to arbitrary  structures, while the notion of a restrained class  
implies, for classes of ordered graphs, bounded grid rank.
In particular, those notions do not require the structures to be ordered, finite, or binary.
The notion of 1-dimensionality has a somewhat geometric flavor.
It is defined in terms of a variant of forking independence -- a central concept in stability theory, generalizing independence in vector spaces or algebraic independence.
The equivalence $\cref{qit:trans}\leftrightarrow\cref{qit:mNIP}$ is due to Baldwin and Shelah~\cite{BS1985monadic}, the implications $\cref{qit:mNIP}\rightarrow\cref{qit:grids}\rightarrow\cref{qit:1-dim}$ are due to Shelah~\cite{shelah:hanfnumbers} and the implication $\cref{qit:1-dim}\rightarrow\cref{qit:mNIP}$ is a very recent result of Braunfeld and Laskowski~\cite{braunfeld2021characterizations}.
Our contribution is the implication~$\cref{qit:1-dim}\rightarrow\cref{qit:restrained}$,  and the overall equivalence in the case of ordered binary structures. Indeed, the implication~$\cref{qit:restrained}\rightarrow\cref{qit:trans}$ follows from the implication~\cref{git:matrix}$\rightarrow$\cref{git:transduce} in \cref{thm:main}.

We now detail the consequences of our results.

\subparagraph{Stanley-Wilf conjecture, Marcus-Tardos theorem.}
In the late 80's, Stanley and Wilf independently conjectured that every proper permutation class has single-exponential growth.
To be more specific, every set of permutations closed under taking subpermutations (definition of permutation \emph{class}) and not equal to the set of all permutations (definition of \emph{proper}) contains at most $c^n$ permutations over $[n]$, for some constant $c$ depending solely on the class.
This conjecture was confirmed in 2004 by Marcus and Tardos~\cite{MarcusT04} (see \cref{subsec:enum-comb} for more details).
What Marcus and Tardos actually showed is that there is a function $f$ such that every square $0,1$-matrix with at least $f(k)$ 1's in average per row (or column) admits a $k$-division where every zone contains a 1.
This result is a milestone in combinatorics, and is at the core of the twin-width theory.
We will use it again in the current paper.

Let us call \emph{permutation matrix class} any (hereditary) class consisting of all the submatrices of a set of permutation matrices. 
In the language of matrices, the Stanley-Wilf conjecture/Marcus-Tardos theorem says that the growth of permutation matrix classes is (at least) $n!$ (for the class of all permutation matrices) or at most $2^{O(n)}$.
\cref{thm:main} \cref{git:small},\cref{git:big} extends that result to any matrix class over a finite alphabet, and the dividing line is the boundedness of twin-width.

We also obtain the following classification of all inclusion-minimal classes of superexponential growth.
Call a (submatrix-closed) class of matrices a \emph{Stanley-Wilf} class if it has superexponential growth, but every its proper subclass has at most exponential growth, that is, growth $2^{O(n)}$. Then the class of permutation matrices is a Stanley-Wilf class, as shown by Marcus and Tardos.
By a similar argument, each of the classes $\mathcal F_\sym$ for $\sym\in\six$ is a Stanley-Wilf class. Moreover, these six classes are precisely \emph{all} the Stanley-Wilf classes of $0,1$-matrices,
and every matrix class of superexponential growth contains one of those classes.
This is a consequence of our result for matrices, and the fact that the six classes are mutually incomparable (see~\cref{cor:stanley-wilf1}).

In the same way we may define Stanley-Wilf classes of ordered graphs, as those hereditary classes of superexponential growth whose every proper, hereditary subclass has at most exponential growth. Then the 25 classes $\mathscr M_{\sym,\lambda,\rho}$ and $\mathscr P$ are precisely all the Stanley-Wilf classes of ordered graphs
(see~\cref{cor:stanley-wilf2}).

\subparagraph{Speed gap in hereditary classes of ordered graphs.}
% \label{subsec: speedgap}
A couple of years after Marcus and Tardos proved the Stanley-Wilf conjecture, Balogh, Bollob\'as, and Morris~\cite{Balogh06,balogh2006hereditary} analyzed the growth of ordered structures, and more specifically, ordered graphs, in an attempt to generalize Marcus and Tardos's ideas to new settings.
They conjectured \cite[Conjecture 2]{Balogh06} that a hereditary class of (totally) ordered graphs has, up to isomorphism, either at most $2^{\Oof(n)}$ graphs with $n$ elements, or at least $n^{n/2+o(n)}$ such graphs, and proved it for weakly sparse graph classes, that is, without arbitrarily large bicliques (as subgraphs).
In a concurrent work, Klazar \cite{Klazar08} repeated that question, and more recently, Gunby and P\'alvölgyi~\cite{Gunby19} observe that establishing the first superexponential jump in the growth of hereditary ordered graph classes is still an open question.

\cref{thm:main} \cref{git:small},\cref{git:big} settles that question, and yields the optimal bound $\sum_{k=0}^{\lfloor n/2 \rfloor} {n \choose 2k}k!$, as conjectured by Balogh et al~\cite{Balogh06}.

\subparagraph{Small conjecture.}
% \label{subsec:small-conj}

Classes of bounded twin-width are small~\cite{twin-width2}, that is, they contain at most $n! c^n$ distinct labeled $n$-vertex structures, for some constant $c$.
(Actually they further contain at most $c^n$ pairwise non-isomorphic structures~\cite{twin-width&permutations}.)
The converse was conjectured for hereditary classes~\cite{twin-width2}.
In the context of classes of totally ordered structures, it is simpler to drop the labeling and to count up to isomorphism.
Indeed, every ordered structure has no non-trivial automorphism.
Then a class is \emph{small} if, up to isomorphism, it contains  $2^{\Oof(n)}$ distinct $n$-vertex structures.
With that in mind,~\cref{thm:main} \cref{git:tww},\cref{git:small} resolves the conjecture in the particular case of ordered graphs.

\subparagraph{Ramsey theory.}
Our proofs are based on multiple Ramsey-theoretic arguments, but also our main result,~\cref{thm:main}, has a bearing on Ramsey theory.
For example, we can conclude 
the following:
\emph{For every ordered matching $H$ there is some cubic graph $G$
such that for every total order $\le$ on $V(G)$, the resulting ordered graph $G_\le$ contains $H$ as an induced subgraph.}
Indeed, otherwise there is some ordered matching $H$ such that every cubic graph $G$ can be equipped with an order in a way which avoids $H$ as an induced subgraph.
This way, we obtain a class $\CC$ of ordered subcubic graphs which does not contain the class $\mathscr M_{=,0,0}$, as it already fails to contain $H$.
Clearly, $\CC$ does not contain any of the remaining 24 classes $\mathscr M_{\sym,\lambda,\rho}$ and $\mathscr P$, as those have unbounded degree.
By~\cref{thm:main} \cref{git:patterns},\cref{git:tww}, $\CC$ has bounded twin-width.
This implies that the class of all (unordered) subcubic graphs also has bounded twin-width, which we know is false (see~\cite{twin-width2}).
More directly based on~\cref{thm:main}, a contradiction can be reached by observing that $\CC$ does \emph{not} have growth $2^{O(n)}$. 

\subparagraph{Transductions and interpretations.}
The study of transductions in theoretical computer science originates from the study of word-like and tree-like structures, such as graphs of bounded treewidth~\cite{10.1007/3-540-19488-6_105} or graphs of bounded clique-width~\cite{COURCELLE199453}.
Classes of bounded twin-width are closed under transductions~\cite{twin-width1}; in particular, no class of bounded twin-width transduces (nor interprets) the class of all graphs.

\cref{thm:main} \cref{git:tww},\cref{git:transduce} characterizes hereditary classes of ordered graphs of bounded twin-width as precisely those which do not transduce the class of all graphs.
This is not unlike a result~\cite{COURCELLE200791} characterizing classes of bounded clique-width as precisely those which do not transduce the class of all graphs via some transduction of counting monadic second-order logic (CMSO, an extension of first-order logic).

\subparagraph{Grid theorems.}
Grid theorems are dichotomy results in structural graph theory which state that either a structure has a small \emph{width}, or otherwise, a grid-like obstruction can be found in the structure.
For example, this applies to the treewidth parameter and planar grids occurring as minors~\cite{ROBERTSON198692}.
It also applies to clique-width and grids being definable in CMSO~\cite{COURCELLE200791}.
 
\cref{thm:summary-mt0} proves an appropriate grid theorem for classes of ordered graphs of bounded twin-width, and more generally, for all classes which are not restrained.
Indeed, such classes define large grids,
which, intuitively, allows to define the `same row' and `same column' relations of arbitrarily large grids
using first-order formulas in the graphs from the class.
From this (also, from~\cref{thm:main} \cref{git:utww},\cref{git:interp}) it follows that if a hereditary class has unbounded twin-width then it interprets the class of all graphs.

There are other known grid theorems, including the Marcus-Tardos theorem itself.
%is a result of this kind, stating that if a square $0,1$ matrix has sufficiently many $1$'s then it must contain a large grid-like pattern formed by the $1$'s.

\subparagraph{Monadic NIP.}
Model theory typically classifies infinite structures according to the combinatorial complexity of families of definable sets. This is usually done through the introduction of tameness properties. The most important such notion is that of stability. A structure is stable if no formula $\phi(x, y)$ encodes arbitrary large half-graphs (the graphs $G_\le^n$ in \cref{fig:ladders1}),
which roughly means that there is no definable order on large subsets of the structure.
 A related, weaker, notion is that of NIP: a structure is \emph{NIP} (or \emph{dependent}) if no formula $\phi(x,y)$ encodes arbitrary bipartite graphs. This captures the tameness properties of families of sets arising from geometric settings (for instance sets of points defined by polynomials of bounded degree).
 
  The notion of monadically NIP (or monadically dependent)  is a stronger requirement which says that the structure is NIP even if every subset of the domain can be used as a unary predicate.
 This is closely related to notions studied in finite model theory and structural graph theory.
A class $\CC$ of structures is monadically NIP if and only if it does not transduce the class of all graphs.
By the recent result of Braunfeld and Laskowski~\cite{braunfeld2021characterizations}, this is equivalent to not defining large grids.

Thus~\cref{thm:main} \cref{git:tww},\cref{git:transduce} proves that a class  of ordered graphs is monadically NIP if, and only if it has bounded twin-width.
Furthermore, our results provide a model-theoretic characterization of 
classes of \emph{unordered} graphs of bounded twin-width:
a class of graphs has bounded twin-width if and only if it it can be obtained from some monadically NIP class of \emph{ordered} graphs by forgetting the order.

% Examples of monadically NIP graph classes include all {nowhere dense} classes.
% A class of graphs  $\CC$ is \emph{nowhere dense}
% if for all $r\in\mathbb N$ there is some $t\in\mathbb N$ such that the $r$-subdivision of the $t$-clique is not a subgraph of any graph in~$\CC$.
% Those include the class of graphs with maximum degree bounded by a constant (those classes have unbounded twin-width~\cite{tww2}),
% as well as every proper minor-closed graph class (here the twin-width is bounded).
% A subgraph-closed class of graphs  is nowhere dense if and only if it is monadically NIP~\cite{adler2014interpreting}.
% Monadically NIP classes are closed under transductions, so any transduction of a nowhere dense graph class is also monadically NIP,
% but not necessarily nowhere dense. 

\subparagraph{Fixed-parameter tractable first-order model checking.}\label{subsec:fomc}
Testing if a given \FO~sentence $\phi$ holds in a given structure $G$ takes time $\Oof(|G|^{|\phi|})$ using a naive algorithm, and it is conjectured that the exponential dependency on $|\phi|$ cannot be avoided.
More precisely, it is conjectured that \FO~model checking is not \emph{fixed-parameter tractable} (\FPT) on the class of all graphs, i.e., does \emph{not} admit an algorithm with running time $f(\phi)\cdot |G|^c$, for some computable function $f\from\mathbb N\to\mathbb N$ and constant $c$.
That statement would indeed hold if $\FPT\neq\AW[*]$, as it is conjectured in parameterized complexity theory~\cite{10.5555/1121738}.

% There are several known classes $\CC$ of structures  for which $\FO$ model checking  is $\FPT$.
% To the best of our knowledge, all known tractable hereditary  classes
% are monadically NIP\footnote{Tractable classes that are not hereditary include for example the class of all finite Abelian groups~\cite{bova2015first}}.

There is an ongoing program aiming to classify all the hereditary graph classes on which \FO~model checking is \FPT. 
Currently such an algorithm is known for nowhere dense classes~\cite{Grohe17}, for interpretations of bounded-degree classes~\cite{Gajarsky16,gajarsky2018recovering},  for map graphs \cite{eickmeyer2017fo}, for some families of intersection and visibility graphs~\cite{hlinveny2019fo}, for transductions of bounded expansion classes when a suitable witness is given~\cite{Gajarsky18}, and finally, for classes with bounded twin-width for which a witness of bounded twin-width can be efficiently computed~\cite{twin-width1}.
Those include proper minor-closed classes, classes of bounded clique-width, posets of bounded width and permutations excluding a fixed permutation pattern.

All known tractable hereditary\footnote{Tractable classes that are not hereditary include for example the class of all finite Abelian groups~\cite{bova2015first}} classes are monadically NIP.
This observation is the basis of the following conjecture:\footnote{This conjecture has been circulating in the community for some time, see e.g. the open problem session at the workshop on Algorithms, Logic and Structure in Warwick in 2016. See also \cite[Conjecture 8.2]{DBLP:journals/corr/abs-1805-01823}.}
\begin{conjecture}\label{conj:NIP}
  Let $\CC$ be a hereditary class of structures.
  Then  \FO~model checking is \FPT~on $\CC$ if and only if $\CC$ is monadically NIP.
\end{conjecture}
Both implications in the conjecture are open.
This conjecture is now confirmed in two prominent cases,  assuming $\FPT\neq\AW[*]$:
\begin{itemize}
	 \item subgraph-closed graph classes~\cite{Grohe17}, where monadically NIP classes are precisely nowhere dense classes,
	 \item hereditary classes of ordered graphs, where monadically NIP classes are precisely classes with bounded twin-width.
\end{itemize}
The second item is by our main result, \cref{thm:main} \cref{git:tww},\cref{git:transduce},\cref{git:easy},\cref{git:hard}.
Both the $\AW[*]$-hardness result in the case of unbounded twin-width, and the \FPT~algorithm\footnote{Recall that~\cite{twin-width1} only provides an $\FPT$ algorithm when a witness of bounded twin-width is given.
  Here we lift this requirement for classes of ordered graphs, thanks to~\cref{thm:approx-tww}.} in the case of bounded twin-width, are new.
Furthermore \cref{thm:main} \cref{git:easy},\cref{git:hard} indicates that for ordered graphs, twin-width is the dividing line of the tractability of \FO~model checking, as for (unordered) graphs, bounded treewidth is with MSO$_2$ (where quantifying over edge subsets is allowed) \cite{DBLP:conf/lics/KreutzerT10}, and rank-width is with MSO$_1$ (where quantifying over edge subsets is disallowed). 
%\sz{Removed this as this is vague}

\cref{conj:NIP} in particular predicts tractability of every class of \emph{unordered} graphs of bounded twin-width.
Note that from~\cite{twin-width1} this is only known if the graph is provided with a witness of bounded twin-width.
Our \cref{thm:approx-tww} gives the desired missing link for ordered graphs, that is, an~\FPT~algorithm which either concludes that the twin-width is at least $k$, or provides a witness of the twin-width being bounded by some computable function of $k$.
This is interesting on its own and gives some hope for the unordered case. 

Related to this, we believe that~\cref{thm:summary-mt0}  may be of independent interest, and possibly of broader applicability than 
just in the context of ordered, binary structures. For example, by~\cref{thm:summary-mt0}, all graph classes of bounded twin-width (without an order) and all transductions of nowhere dense classes are restrained, generalizing the fact that classes of ordered graphs of bounded twin-width have bounded grid rank.
We remark that although our proof of~\cref{thm:main} is purely combinatorial, an alternative proof can be derived from~\cref{thm:summary-mt0} (see our unpublished report~\cite{SimonTorunczyk}). This demonstrates that model-theoretic methods can be used in the context of algorithmic and structural graph theory, and that those two areas are intimately related.

\subparagraph{Expressive power of least fixed-point logic.}
It is well-known that least fixed-point logic {\LFP} captures the complexity class {\P} on ordered structures.
The \emph{ordered conjecture} \cite{kolaitis1992fixpoint} asserts that {\LFP} is more expressive than first-order logic on every infinite class $\mathscr C$ of finite ordered structures, in the sense that there is a sentence $\phi\in\LFP$  which is not equivalent to any sentence $\phi'\in\FO$ on $\CC$.

The conjecture holds for every class $\mathscr C$ of ordered structures for which \FO~model checking is \FPT~(see \cite[Cor. III.2]{chen2012ordered}).
In particular, it follows from \cref{thm:main} \cref{git:transduce},\cref{git:easy} that the ordered conjecture holds for every subclass of a hereditary dependent class of finite ordered binary structures. 

\section{Roadmap}

For the most part, the proof will be carried out in the language of matrices over a finite alphabet.
Matrices are considered ordered, in the sense that they are equipped with a total order on the rows and on the columns.
A \emph{class} of matrices is, by definition, assumed to be closed under taking submatrices, that is, removing rows and/or columns.

For the sake of simplicity, the description below concerns matrices with entries $0$ or $1$, called \emph{$0,1$-matrices}. 
A $0,1$-matrix can be seen as a relational structure whose domain consists of its rows and columns, equipped with two unary predicates marking the rows and the columns, respectively, a total order which places the rows before the columns, and a binary, symmetric relation which relates a row with a column if the entry at their intersection is equal to 1.

Recall that the six matrix classes of unbounded twin-width which arise are: the class $\mathcal F_=$ of all permutation matrices, the class $\mathcal F_\neq$ obtained from permutation matrices  by exchanging $0$'s with $1$'s, and four classes $\mathcal F_{\le R},\mathcal F_{\ge R},\mathcal F_{\le C},\mathcal F_{\ge C}$ obtained from permutation matrices by propagating each value $1$ downward/upward/leftward/rightward, respectively (see~\cref{fig:six-unavoidable}).
The \emph{growth} of a class of matrices is the function counting the number of distinct (square) $n \times n$-matrices in the class, for a given $n\ge 1$. 

Our main result concerning (hereditary) classes of matrices over a finite alphabet is as follows.

\begin{theorem}\label{thm:equiv}
Given a class $\mathcal M$ of $0, 1$-matrices, the following are equivalent.
\begin{enumerate}[label=({\roman*}),ref=\roman*]
\item \label{it:bd-tww} $\mathcal M$ has bounded twin-width.
\item \label{it:bd-gr} $\mathcal M$ has bounded grid rank.
\item \label{it:ramsey} $\mathcal M$ does not contain any of the six classes $\mathcal F_=,\mathcal F_\neq, \mathcal F_{\le R},\mathcal F_{\ge R},\mathcal F_{\le C},\mathcal F_{\ge C}$.
\item \label{it:nip} $\mathcal M$ does not interpret the class of all graphs.
\item \label{it:m-nip} $\mathcal M$ does not transduce the class of all graphs.
\item \label{it:fact-speed} $\mathcal M$ does not have growth at least $n!$.
\item \label{it:exp-speed} $\mathcal M$ has growth at most $2^{\Oof(n)}$.
\item \label{it:lin-mc} \FO~model checking is \FPT~on $\mathcal M$. (The implication from \cref{it:lin-mc} holds if $\FPT\neq \AW[*]$.\/)  
 \item \label{it:bd-rd} there is some $r \in\mathbb N$ such that no matrix $M \in \mathcal M$ admits a $r$-rich division.
\end{enumerate}
\end{theorem}

The last condition, \cref{it:bd-rd}, is a technical one, whose definition is deferred to~\cref{sec:prelim}.
This will be a key intermediate step in proving that \cref{it:bd-gr} implies \cref{it:bd-tww}, as well as in getting an approximation algorithm for the twin-width of a matrix.
\cref{thm:equiv} reads the same for matrices over a finite alphabet $A$, except that \cref{it:ramsey} is replaced by: No selection of $\mathcal M$ contains any of the six classes $\mathcal F_=,\mathcal F_\neq, \mathcal F_{\le R},\mathcal F_{\ge R},\mathcal F_{\le C},\mathcal F_{\ge C}$, where for $a\in A$, the  \emph{$a$-selection} of a matrix class $\mathcal M$ is the class obtained from the matrices of $\mathcal M$ by replacing the letter $a\in A$ with $1$ and the remaining letters with $0$.
In~\cref{fig:proof-scheme}, a class satisfying~\cref{it:ramsey} is called \emph{pattern-avoiding} (the definition will be recalled in~\cref{subsec:patterns}).

\tikzexternaldisable  
\begin{figure}[h!]
  \centering
  \resizebox{360pt}{!}{
  \begin{tikzpicture}
    \def\s{1.3}
    \def\gr{black!30!green}
    \def\red{black!10!red}
    %\footnotesize{}
    \foreach \i/\j/\l/\t in {0/0/tww/$(i)$ bounded twin-width, 0/1.5/rd/$(ix)$ no rich division, 4.2/2.8/gr/$(ii)$ bounded grid rank, 2.5/1/sm/$(vii)$ small, 4.2/1.7/fa/$(vi)$ subfactorial growth, 4.3/-1.5/mnip/$(v)$ monadically NIP, 6.8/-0.5/nip/$(iv)$ NIP, 7.3/1.2/st/$(iii)$ pattern-avoiding, 4.35/-0.5/tr/$(viii)$ tractable}{
      \node[preaction={fill,fill opacity=0.1},draw,rectangle,rounded corners] (\l) at (\i * \s,\j * \s) {\t} ; 
    }

    \draw[\gr,double, ->] (tww) to [bend right = 25] node[black,below] {\cite[Sec. 8]{twin-width1}} (mnip) ;
    \draw[\gr,double, ->] (mnip) to [bend right = 25] node[black,below] {def} (nip) ; 
    %\draw[\gr,double, ->] (tr) to [bend left = 20] node[black,above] {if \FPT $\neq$ \AW$[*]$} (nip) ; 
    \draw[orange,double, ->] (tww) to [bend left = 15] node[black,above] {\cite[Sec. 7]{twin-width1}}  node[black,below] {\cref{thm:approx-tww}}(tr) ; 
    \draw[\red,double, ->] (tr) to [bend left = 15] node[black,above] {\cref{sec:NIP}} node[black,below] {if $\FPT \neq \AW[*]$} (st) ;
    \draw[\red,double, ->] (st) to [bend right = 15] node[black,right] {\cref{thm:Ram6}} (gr) ;
    \draw[\red,double, ->] (nip) to [bend right = 10] node[black,right] {\cref{sec:NIP}} (st) ; 
    \draw[\red,double, ->] (fa) to [bend right = 10] node[black,left] {\cref{thm:fact-implies-bgr}} (gr) ;
    \draw[\gr,double, ->] (sm) to [bend right = 5] node[black,below] {def} (fa) ; 
    \draw[\gr,double, ->] (tww) to [bend left = 15] node[black,above] {\cite[Sec. 3]{twin-width2}} (sm) ; 
    \draw[\red,double, <->] (rd) to [bend right = 15] node[black,left] (app) {\cref{sec:approx-tww}} (tww) ;  
    \draw[\red,double, ->] (gr.west) to [bend right = 15] node[black,above] {\cref{sec:rich-to-gr}} (rd) ;
    \normalsize{}
  \end{tikzpicture}
  }
  \caption{A bird's eye view of the paper.
    In green (arrows without a reference to a part of the paper), the implications that were already known for general binary structures.
    In red (other arrows except \cref{it:bd-tww} $\Rightarrow$ \cref{it:lin-mc}), the new implications for matrices on finite alphabets, or ordered binary structures.
    The effective implication \cref{it:bd-tww} $\Rightarrow$ \cref{it:bd-rd} is useful for~\cref{thm:approx-tww}.
    See~\cref{fig:zoom} for a more detailed proof diagram, distinguishing what is done in the language of matrices and what is done in the language of ordered graphs.
  }
  \label{fig:proof-scheme}
\end{figure}
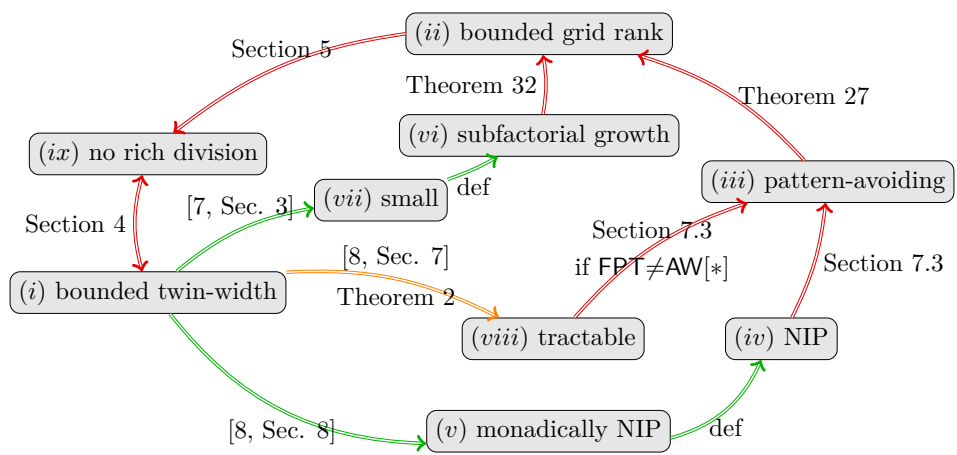
\tikzexternalenable
 
% We transpose these results for hereditary classes of ordered graphs.
% We also refine the model-theoretic (\cref{it2:dep,it2:no-s-int}) and growth (\cref{it2:sub-fact}) characterizations.

% \begin{theorem}\label{thm:her-ordered-graphs}
%   Let $\mathscr C$ be a hereditary class of ordered graphs.
%   The following are equivalent.
% 	\begin{enumerate}[label={\arabic*}\emph{)},ref=\arabic*]
% 		\item \label{it2:bd-tw} $\mathscr C$ has bounded twin-width.
% 		\item \label{it2:mon-dep} $\mathscr C$ is monadically dependent.
% 		\item \label{it2:dep} $\mathscr C$ is dependent.
% 		\item \label{it2:no-s-int} No simple interpretation in $\mathscr C$ is the class of all ordered graphs. 
% 		\item \label{it2:small} $\mathscr C$ is small.
% 		\item \label{it2:smallish} $\mathscr C$ contains $2^{O(n)}$ ordered $n$-vertex graphs.
% 		\item \label{it2:sub-fact} $\mathscr C$ contains less than $\sum_{k=0}^{\lfloor n/2\rfloor}\binom{n}{2k}\,k!$ ordered $n$-vertex graphs, for some $n$.
% 		\item \label{it2:ram-min} $\mathscr C$ does not include one of 256 hereditary ordered graph classes $\mathscr M_{\eta,\lambda,\rho}$ with unbounded twin-width. 
% 		\item \label{it2:ram-per} There exists a permutation $\sigma$ such that $\mathscr C$ does not include any of 256 ordered graphs defined from $\sigma$. 
%                 \item \label{it2:fo-mc} \FO~model checking is fixed-parameter tractable on $\mathscr C$.\\
%                   (This implies the other items only if $\FPT \neq \AW[*]$.)
% 	\end{enumerate}
% \end{theorem}

As mentioned in the introduction, we prove an analogous result (see \cref{thm:main}) for classes of ordered graphs, or more generally for classes of ordered binary structures.
In an informal nutshell, the high points of the paper read: For hereditary ordered binary structures, bounded twin-width, small, subfactorial growth, and tractability of \FO~model checking are all equivalent.
We conclude by giving a more detailed statement of the approximation algorithm. 

\begin{reptheorem}{thm:approx-tww}[more precise statement]
  There is a fixed-parameter algorithm, which, given an ordered binary structure $G$, encoded by a matrix $M$, and a parameter $k$, either outputs
  \begin{itemize}
    \item a $2^{O(k^4)}$-sequence of $G$, implying that $\tww(G) = 2^{O(k^4)}$, or
    \item a $2k(k+1)$-rich division of $M$, implying that $\tww(G) > k$.
  \end{itemize}
\end{reptheorem}

We now outline the proof of~\cref{thm:equiv,thm:main}.

\subparagraph{Proof outline.}
Bounded twin-width is already known to imply interesting properties:  \FO~model checking is {\FPT}  if a witness of small twin-width is part of the input~\cite{twin-width1}, monadic dependence~\cite{twin-width1}, smallness~\cite{twin-width2} (see the green and orange arrows in~\cref{fig:proof-scheme,fig:zoom}).
For a characterization of some sort in the particular case of ordered structures, the challenge is to find interesting properties implying bounded twin-width.
A~central characterization in the first paper of the series~\cite{twin-width1} goes as follows.

A~graph class $\mathscr C$ has bounded twin-width if and only if 
there is a constant $d$ such that every graph $G\in\CC$ can be ordered so that the adjacency matrix along that order has no $d$-division where each zone contains two non-identical rows and two non-identical columns.
The backward direction is effective: From such an ordering, we obtain a witness of bounded twin-width in polynomial time.

Now that we consider \emph{ordered} graphs it is tempting to try this order to get a witness of low twin-width.
Things are not that simple.
Consider the graph $G$ with vertices $1,\ldots,n$ ordered naturally, where two vertices are adjacent if and only if they have different parity.
The adjacency matrix $M$ of  $G$ is the checkerboard matrix to the left in \cref{fig:checkerboard}.
This matrix is fairly simple and indeed has bounded twin-width (this will be evident once we formally define twin-width).
%It admits a $(1,2)$-sequence.
%We can merge the first and third columns into $C_o$, the second and fourth columns into $C_e$, then $C_o$ and the fifth into $C_o$, $C_e$ and the sixth into $C_e$, and so on.
%This creates a sequence of \opars{1} since only two column parts, $C_o$ and $C_e$, ever get in conflict.
%The maximum error value remains 0 since all columns of odd (resp.~even) index are equal.
%Then we proceed in the same way on the row parts.
%Again it makes for a ``partial'' $(1,0)$-sequence.
%Finally we are left with two row parts and two column parts that we merge in any order.
%This yields an error value of~2, while preserving the fact that the partitions are 1-overlapping.
%
%So the twin-width of all the checkerboard matrices is bounded.
\begin{figure}\centering
        \includegraphics[scale=0.5,page=11]{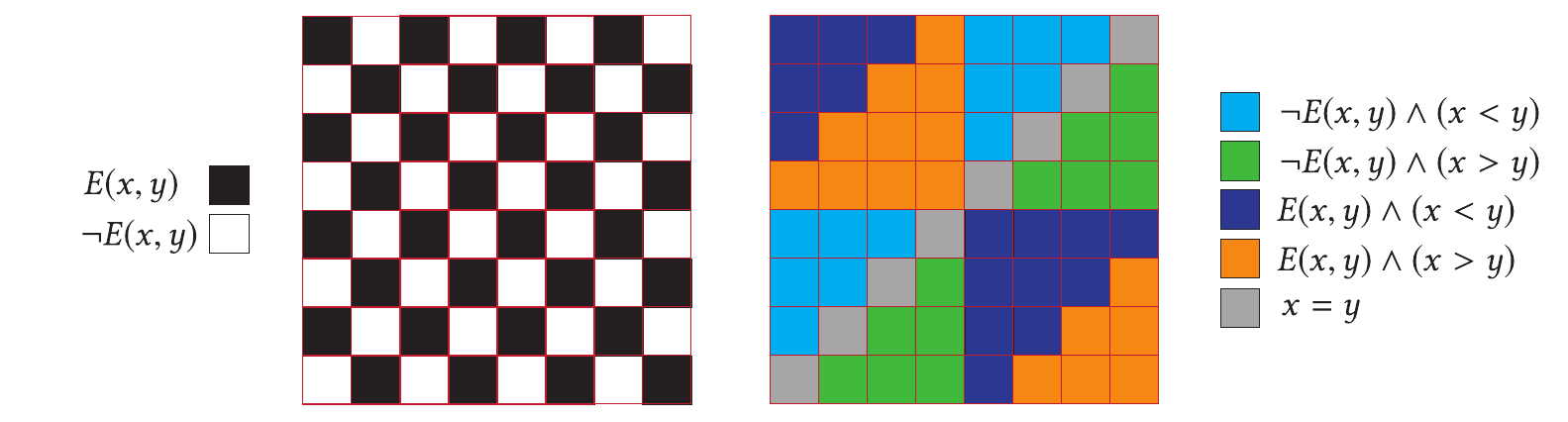}
        \caption{
          \emph{Left:} The adjacency matrix of the ordered graph 
          $G$ with vertices $1,\ldots,n$  and edges $ij$ such that $i+j$ is odd, along the usual order.  (The first row is at the bottom.)
          \emph{Right:} The adjacency matrix along another order,
          encoding the adjacency as well as the original order. Every $4$-division contains some constant zone.
        }
        \label{fig:checkerboard}
      \end{figure}
Yet for $d = \lfloor n/2 \rfloor$, the matrix $M$ has a $d$-division where each zone has two different rows and columns.
Now a \emph{good reordering} of $G$ would put all the odd-indexed vertices together, followed by all the even-indexed vertices.
Then the adjacency matrix $M'$ of $G$ along the new order (\cref{fig:checkerboard}, right), where the entries of $M'$ now encode the edges of $G$ as well as the original order, 
is such that every $4$-division contains a constant zone.

Can we find such reorderings automatically?
Eventually we can, but a~crucial opening step is precisely to nullify the importance of the reordering.
We show that matrices have bounded twin-width exactly when they have bounded grid rank, that is, they do not admit rank-$k$ divisions for arbitrary $k$.
This natural strengthening on the condition that cells should satisfy (from rank 2 to rank $k$) exempts us from the need to reorder.
Note that the checkerboard matrix does not have any rank-$k$ division already for $k=3$.

An important intermediate step is provided by the concept of rich divisions (see~\cref{subsec:rds} for a definition).
We first prove that a greedy strategy to find a potential witness of bounded twin-width can only be stopped by the presence of a large rich division; thus, unbounded twin-width implies the existence of arbitrarily large rich divisions.
This brings a theme developed in~\cite{twin-width1} to the ordered world.
In turn, leveraging the Marcus-Tardos theorem, we show that huge rich divisions contain large rank divisions (this is almost entirely summarized by~\cref{fig:rd-to-gr}).

By a series of Ramsey-like arguments, we find in large rank divisions more and more structured submatrices encoding universal permutations.
Eventually we find at least one of six encodings of all permutations.
More precisely, each class of unbounded grid rank contains one of the classes $\mathcal F_\sym$, for some $\sym\in\six$.
As each of the classes $\mathcal F_\sym$ has growth $n!$, this chain of implications shows that hereditary classes with unbounded grid rank have growth at least $n!$. 
Conversely, classes of matrices of bounded twin-width have growth $2^{O(n)}$ by~\cite{twin-width2}. 
That establishes the announced speed gap for matrix classes.
Moreover, as each of the classes $\mathcal F_\sym$ interprets the class of all graphs and has an $\AW[*]$-hard model checking, we obtain~\cref{thm:equiv}.

Finally we translate the permutation encodings in the language of ordered graphs.
This allows us to refine the growth gap specifically for ordered graphs.
We also prove that including a family $\mathcal F_\sym$, or its ordered-graph equivalent $\mathscr M_{\sym,\lambda,\rho}$, is an obstruction to being NIP.
This follows from the fact that the class of all permutation graphs is independent.
As we get an effectively constructible interpretation to the class of all structures (matrices or ordered graphs), we conclude that \FO~model checking is not {\FPT} on hereditary classes of unbounded twin-width.
This is the end of the road.
The remaining implications to establish the equivalences of~\cref{thm:equiv,thm:main} come from \cite[Sections 7 and 8]{twin-width1}, \cite[Section 3]{twin-width2}, and~\cref{thm:approx-tww} (see \cref{fig:zoom}).

\cref{thm:summary-mt0} is proved using model-theoretic methods. In particular, it relies on a suitable analogue of forking independence for monadically NIP classes.

\begin{figure}[ht]
	\begin{center}
		\includegraphics[width=\textwidth]{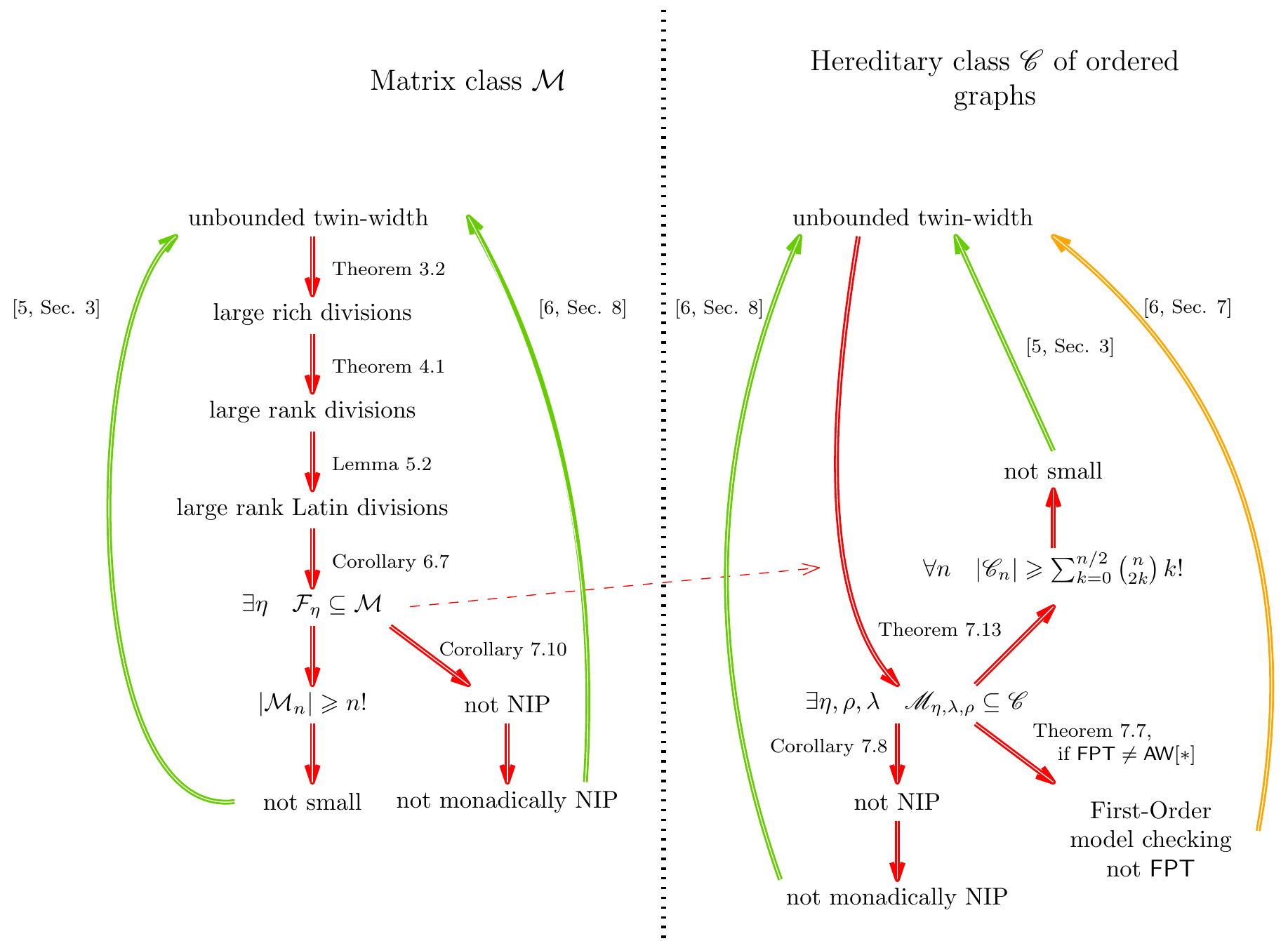}
	\end{center}
	\caption{A more detailed proof diagram. 
        }
        \label{fig:zoom}
\end{figure}

\subparagraph{Organization.}

The rest of the paper is organized as follows.
In~\cref{sec:prelim}, we formally define all the notions used in the rest of the paper.
In~\cref{sec:approx-tww}, we show that \cref{it:bd-tww} and \cref{it:bd-rd} are equivalent.
As a by-product, we obtain a fixed-parameter $f(\text{OPT})$-approximation algorithm for the twin-width of ordered matrices.
In~\cref{sec:rich-to-gr}, we prove the implication \cref{it:bd-gr} $\rightarrow$ \cref{it:bd-rd}.
In~\cref{sec:latin}, we introduce so-called \emph{rank Latin divisions} and show that large rank divisions contain large rank Latin divisions. 
In~\cref{sec:more-ramsey}, we further clean the rank Latin divisions in order to show that \cref{it:ramsey} $\rightarrow$ \cref{it:bd-gr} and \cref{it:fact-speed} $\rightarrow$ \cref{it:bd-gr}.
In~\cref{sec:matchings}, we show that \cref{it:lin-mc} $\rightarrow$ \cref{it:ramsey} and \cref{it:nip} $\rightarrow$ \cref{it:ramsey} transposed to the language of ordered graphs.
We also refine the lower bound on the growth of ordered graph classes with unbounded twin-width, to completely settle Balogh et al.'s conjecture~\cite{Balogh06}.
See~\cref{fig:zoom} for a visual outline.
Finally, in~\cref{sec:model theory} we prove~\cref{thm:summary-mt0}.

\section{Preliminaries}\label{sec:prelim}

Everything which is relevant to the rest of the paper will now be properly defined.
We may denote by $[i,j]$ the set of integers that are at least $i$ and at most $j$, and $[i]$ is a short-hand for $[1,i]$. 
We start with the combinatorial objects.

\subsection{Graphs, orders, matrices, permutations}\label{subsec:basic-notions}
By \emph{graph} we mean a simple, undirected graph $G$,
and denote its set of vertices $V(G)$ and set of edges $E(G)$. An edge with endpoints $u$ and $v$ is denoted $uv$ or $vu$.
A \emph{total order}  on a set $X$ is a binary relation $<$ which is transitive, irreflexive, such that for all $x,y\in X$ either $x<y$ or $y<x$ holds.
An \emph{ordered graph} is a graph together with a total order on its vertices. The \emph{edge complement} of a graph (resp. ordered graph) $G$ is the graph (resp. ordered graph) obtained from $G$ by replacing edges by non-edges, and vice-versa.

A \emph{matrix} $M$ over a finite alphabet $A$ is a function $M\from R \times C \to A$, where $R$ is a totally ordered set of rows and $C$ is a totally ordered set of columns.
The value $M(r,c)$, that we will often denote $M_{r,c}$, is the \emph{entry} of $M$ at position $(r,c)$, or in row $r$ and column $c$.
We may say that $M$ is an $R\times C$ matrix, or an $n \times m$ matrix, where $n=|R|$ and $m=|C|$.

A $0,1$-matrix is a matrix over the alphabet $\set{0,1}$. 
A $0,1$-matrix with rows $R$ and columns $C$ can be viewed as an ordered graph with vertices $R\uplus C$,
 total order $<$ obtained from the orders on $R$ and $C$ by making all the columns  larger than all the rows, and edges $rc$ such that $r\in R$, $c\in C$
 and $M(r,c)=1$.

We distinguish matrices only up to isomorphisms which preserve the order of the rows and columns.
A~\emph{submatrix} of a matrix $M$ is any matrix obtained from $M$ by deleting a (possibly empty) set of rows and columns. 
Analogously to permutation classes which are by default supposed closed under taking subpermutations (or patterns), we will define a \emph{class} of matrices as a set of matrices closed under taking submatrices.
The \emph{submatrix closure} of a matrix $M$ is the set of all submatrices of $M$ (including $M$ itself).
Thus our matrix classes include the submatrix closure of every matrix they contain.
On the contrary, classes of (ordered) graphs are only assumed to be closed under isomorphism.
A~\emph{hereditary} class of (ordered) graphs (resp.~binary structures) is one that is closed under taking induced subgraphs (resp.~induced substructures).

An $n$-\emph{permutation}, for $n\ge 1$, is a bijection $\pi\from [n]\to[n]$.
The set of all $n$-permutations is denoted $\mathfrak S_n$. 
Permutations turn out to be of central importance in the theory developed here,
and indeed, twin-width has its origins in the Stanley-Wilf conjecture which is precisely about permutations. As we will see, 
classes of ordered graphs or matrices with unbounded twin-width are exactly those which contain encodings of all permutations,
under a suitable encoding.

We will use several views on permutations
(see~\cref{fig:permutations}):
 as bijections between two ordered sets, as sets equipped with two total orders,  
 as ordered matchings, as 
ordered permutation graphs,
and as $0,1$-matrices. 

\begin{figure}\centering
  \includegraphics[scale=0.43,page=9]{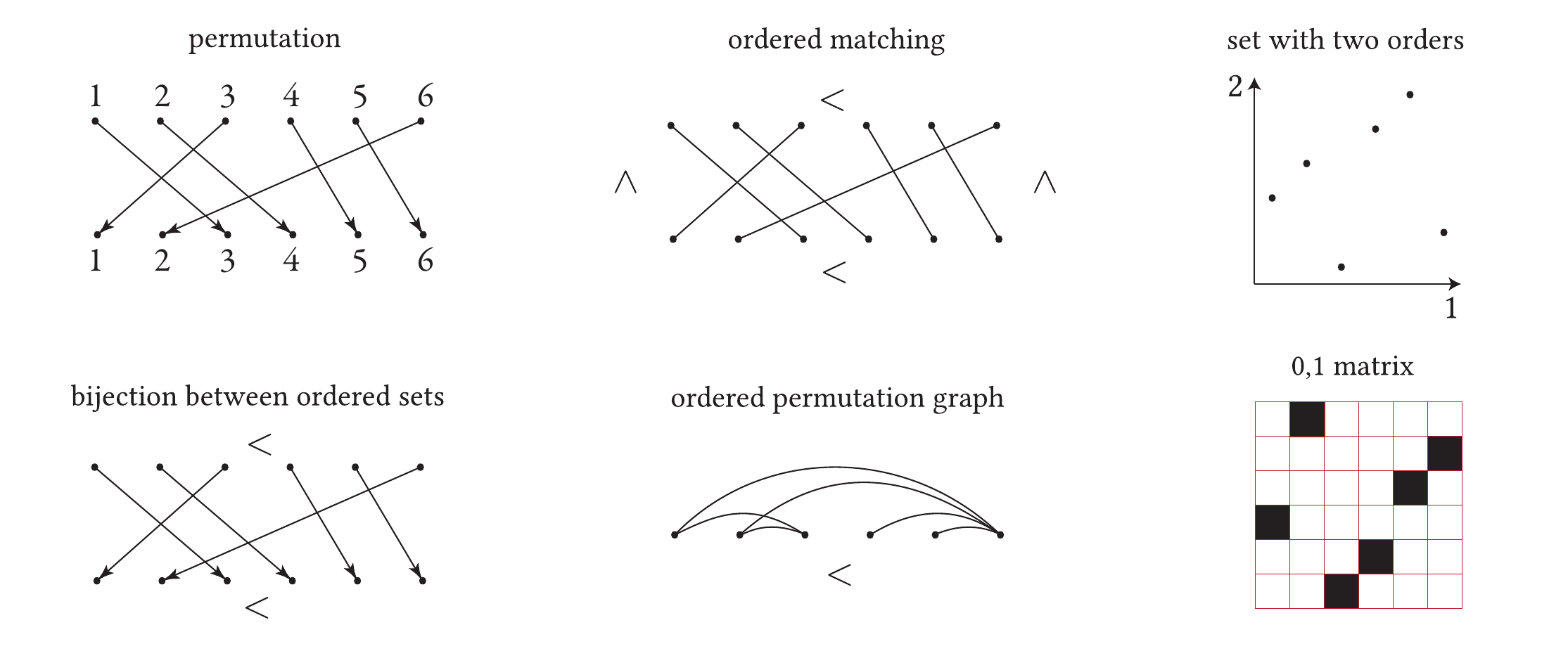}
  \caption{Six different views on the same permutation. We use the convention that the first row of a matrix is at the bottom.}
  \label{fig:permutations}
\end{figure}

\medskip

An $n$-permutation $\pi$ may be viewed as a bijection $\pi$ between two totally ordered sets, namely $X=([n],<)$ and $Y=([n],<)$.
Conversely, for every bijection $f\from X\to Y$ between two totally ordered sets of size $n$ there is a unique $n$-permutation $\pi$ such that $f=i_Y^{-1}\circ\pi\circ i_X$ holds for the unique order-preserving bijections $i_X\from X\to [n]$ and $i_Y\from Y\to [n]$.
Using this correspondence, we may define the notion of a \emph{subpermutation}.
A subpermutation of an $n$-permutation $\pi$
induced by a set $U\subseteq[n]$ is the unique $|U|$-permutation which corresponds to the restriction $\pi|_U$, treated as a bijection between the ordered sets $U\subseteq [n]$ and $\pi(U)\subseteq[n]$, via the correspondence described above.

Similarly, an $n$-permutation $\pi$ defines two orders on $[n]$, namely the usual order $<_1$, and the order $<_2$ such that $i<_2j$ if and only if $\pi(i)<\pi(j)$.
Conversely, every finite set equipped with 
two total orders is isomorphic to one obtained from a permutation as described above. Via this correspondence, subpermutations correspond exactly to induced substructures of sets equipped with two total orders.

An \emph{ordered matching} is an ordered graph with vertices $a_1<\ldots<a_n<b_1<\ldots<b_n$ such that each $a_i$ is adjacent with exactly one $b_j$, and vice-versa. Hence, there is a unique $n$-permutation $\pi$ such that $a_i$ is adjacent with $b_{\pi(i)}$, for $i \in [n]$.

An \emph{ordered permutation graph} associated with an $n$-permutation $\pi$ 
 is the ordered graph $G_\pi$ with vertices $[n]$ ordered naturally, such that $i<j$ are adjacent if and only if $\pi(i)>\pi(j)$.
 Note that the isomorphism type of $G_\pi$ determines the permutation $\pi$ uniquely. 
If $\sigma$ is a subpermutation of $\pi$ induced by $U\subseteq [n]$ then $G_\sigma$ is the subgraph of $G_\pi$ induced by $U$.
Observe that the edge complement of a permutation graph $G_\pi$ is also a permutation graph. Namely, if $G_\pi$ corresponds to two total orders $<_1,<_2$ on $[n]$, as explained above, then the edge complement of $G_\pi$ corresponds to the orders $<_1,>_2$ on $[n]$.

Finally, $n$-permutations correspond to $n\times n$ $0,1$-matrices with exactly one $1$ in each row and in each column.
We adopt the standard conventions that the 1 entries in the matrix of permutation $\sigma \in \mathfrak S_n$ are at positions $(i,\sigma(i))$ for $i \in [n]$, and that, in this context of patterns, the first row is placed at the bottom.
A permutation $\sigma$ is a subpermutation of $\pi$ if the matrix of $\sigma$ is a submatrix of the matrix of $\pi$.

In fact, we will see even more representations of permutations as matrices or ordered graphs, namely five further matrix classes and twenty-three further classes of ordered graphs.

 \subsection{Structures}\label{subsec:logic}
A relational \emph{signature} $\Sigma$ is a finite set of relation symbols~$R$, each  with a specified arity $r\in\mathbb N$. 
\mbox{A \emph{$\Sigma$-structure}~$\mathbf A$} is defined by a set~$A$ (the \emph{domain} of $\mathbf A$) together with a relation $R^\mathbf A\subseteq A^{r}$ for each relation symbol \mbox{$R \in \Sigma$} with arity $r$.
The syntax and semantics of first-order formulas over $\Sigma$, or \emph{$\Sigma$-formulas} for brevity, are defined as usual. 
% The first-order language $\fo(\Sigma)$ associated to \mbox{$\Sigma$-structures} defines, for each relation symbol $R$ with arity $r$ the predicate $R$ such that 
% $\mathbf A\models R(v_1,\dots,v_{r})$ if  $(v_1,\dots,v_{r})\in R^{\mathbf A}$.

A graph is viewed as a structure 
over the signature with one binary relation $E$
indicating the adjacency between vertices.
A total order is viewed as a structure over the signature with one binary relation $<$.
An ordered graph is viewed as a structure over the signature with two binary relations, $E$ and $<$.
More generally, an \emph{ordered binary structure} is
a structure $\str A$ over a signature $\Sigma$ consisting of unary and binary relation symbols which includes the symbol $<$,
and such that $<$ defines in $\str A$ a total order on the domain of $\str A$.
% We could also allow some unary relation symbols but any unary relation $U$ can be simulated by a binary relation $E_U$, with $E_U(x,y) \leftrightarrow (x=y \land U(x))$. 

A matrix $M$ over a finite alphabet $A$ 
with rows $R$ and columns $C$
is viewed as an ordered binary structure with domain $R \uplus C$,
 equipped with the following relations:
 \begin{itemize}
   \item unary relations $R$ and $C$, interpreted as the set of rows and set of columns, respectively,
   \item a binary relation $<$ which defines a total order on $R\uplus C$, extending the total orders  on the rows and columns of $M$ in such a way that the rows precede the columns,
   \item  one binary relation $E_a$, for each $a\in A$,
   where $E_a(r,c)$ holds if and only if $r$ is a row, $c$ is a column, and $a$ is the entry of $M$ at row $r$ and column $c$.
 \end{itemize}

\subsection{Twin-width}\label{subsec:graph-theory}

In the first paper of the series~\cite{twin-width1}, we define twin-width for general binary structures via \emph{unordered} matrices.
The twin-width of (ordered) matrices can be defined this way by encoding the total orders on the rows and on the columns with two binary relations.
However we will give an equivalent definition, tailored to ordered structures.
This slight shift is already a first step in understanding these structures better, with respect to twin-width.
We insist that matrices are always ordered objects, in the current paper.
Thus the twin-width of a matrix does \emph{not} coincide with the twin-width of unordered matrices, as defined in~\cite{twin-width1}.

Let $M$ be an $n \times m$ matrix with entries ranging over a fixed finite set.
We denote by $R:=\{r_1, \ldots, r_n\}$ its set of rows and by $C:=\{c_1, \ldots, c_m\}$ its set of columns.
Let $S$ be a non-empty subset of columns, $c_a$~be the column of~$S$ with minimum index~$a$, and $c_b$, the column of~$S$ with maximum index~$b$. 
The \emph{span} of $S$ is the set of columns $\{c_a, c_{a+1}, \ldots, c_{b-1}, c_b\}$. 
We say that a subset $S \subseteq C$ is in \emph{conflict} with another subset $S' \subseteq C$ if their spans intersect. 
A~partition $\mathcal P$ of $C$ is \emph{$k$-overlapping} if every part of $\mathcal P$ is in conflict with at most~$k$ other parts of~$\mathcal P$.
The definitions of \emph{span}, \emph{conflict}, and \emph{\kop} similarly apply to sets of rows.
With that terminology, a \emph{division} is a \opar{0}.

A partition $\mathcal P$ is a \emph{contraction} of a partition $\mathcal P'$ (defined on the same set) if it is obtained by merging two parts of ${\mathcal P}'$.
A \emph{contraction sequence} of $M$ is a sequence of partitions ${\mathcal P}_1,\dots, {\mathcal P}_{n+m-1}$ of the set $R \cup C$ such that ${\mathcal P}_1$ is the partition into $n+m$ singletons, ${\mathcal P}_{i+1}$ is a contraction of ${\mathcal P}_i$ for all $i \in [n+m-2]$, and $\mathcal P_{n+m-1}=\{R,C\}$.
In other words, we merge at every step two \emph{column parts} (made exclusively or columns) or two row parts (made exclusively or rows), and terminate when all rows and all columns both form a single part.
We denote by $\mathcal P^R_i$ the  partition of $R$ induced by $\mathcal P_i$ and by $\mathcal P^C_i$ the partition of $C$ induced by $\mathcal P_i$.
A \emph{contraction sequence} is \emph{$k$-overlapping} if all partitions $\mathcal P^R_i$ and $\mathcal P^C_i$ are \kops.
Note that a $0$-overlapping sequence is a sequence of divisions.

If $S^R$ is a subset of $R$, and $S^C$ is a subset of $C$, we denote by $S^R \cap S^C$ the submatrix at the intersection of the rows of $S^R$ and of the columns of $S^C$.
Given some column part $C_a$ of ${\mathcal P}^C_i$, the \emph{error value} of $C_a$ is the number of row parts $R_b$ of ${\mathcal P}^R_i$ for which the submatrix $C_a \cap R_b$ of $M$ is not constant.
The error value is defined similarly for rows, by switching the role of columns and rows.
The \emph{error value} of ${\mathcal P}_i$ is the maximum error value of some part in ${\mathcal P}^R_i$ or in ${\mathcal P}^C_i$. 
A contraction sequence is a \emph{$(k,e)$-sequence} if all partitions ${\mathcal P}^R_i$ and ${\mathcal P}^C_i$ are \kops with error value at most $e$.
Strictly speaking, to be consistent with the definitions in the first paper~\cite{twin-width1}, the \emph{twin-width} of a matrix $M$, denoted by $\tww(M)$, is the minimum $k+e$ such that $M$ has a $(k,e)$-sequence.
This matches, setting $d := k+e$, what we called a \emph{$d$-sequence} for the binary structure encoding $M$~\cite{twin-width1}.  
We will however not worry about the exact value of twin-width, but merely whether it is bounded or unbounded on a class of structures.
Thus for the sake of simplicity, we often consider the minimum integer $k$ such that $M$ has a $(k,k)$-sequence.
This integer is indeed sandwiched between $\tww(M)/2$ and $\tww(M)$. 

The twin-width of a matrix class $\mathcal M$, denoted by $\tww(\mathcal M)$, is simply defined as the supremum of $\{\tww(M)$ $|$ $M \in \mathcal M\}$.
We say that $\mathcal M$ has \emph{bounded twin-width} if $\tww(\mathcal M) < \infty$, or equivalently, if there is a finite integer $k$ such that every matrix $M \in \mathcal M$ has twin-width at most~$k$.
A class $\mathscr C$ of ordered graphs has \emph{bounded twin-width} if all the adjacency matrices of graphs $G \in \mathscr C$ along their vertex ordering, or equivalently their submatrix closure, form a set/class with bounded twin-width.

We can more generally define the twin-width of ordered binary structures via matrices. 
The \emph{matrix encoding} of an ordered binary structure $\mathbf A$ with relations $<, E_1, \ldots, E_p$ is the $|A| \times |A|$ matrix over the alphabet $\set{-1,0,1,2}^p$ whose entry at position $(x,y)$, is the vector $(b_1,\ldots,b_p)\in\set{-1,0,1,2}^p$ such that $b_i=1$ if $E_i(x,y) \land \neg E_i(y,x)$ holds, $b_i=-1$ if $\neg E_i(x,y) \land E_i(y,x)$ holds, $b_i=2$ if $E_i(x,y) \land E_i(y,x)$ holds, and $b_i=0$ otherwise.
Then the twin-width of an ordered binary structure $\mathbf A$ is simply the twin-width of the matrix encoding of $\mathbf A$.
We choose this particular encoding so that the vector $(b_1, b_2, \ldots, b_p)$ at position $(x,y)$ and the one $(b'_1, b'_2, \ldots, b'_p)$ at position $(y,x)$ satisfies $b_i=\pm b'_i$ for every $i \in [p]$.
We then say that the matrix is \emph{mixed-symmetric} as in each vector some coordinates are symmetric while others are skew-symmetric.
This technicality allows to turn a contraction sequence of a matrix encoding into a contraction sequence of its associated ordered binary structure.
For more details, see~\cite[Section 5, Theorem 14]{twin-width1}.

%Or, insisting on the equivalence between matrices and ordered bipartite graphs, $\mathcal M$ has bounded \emph{twin-width} if the class of ordered bipartite graphs corresponding to the matrices has bounded twin-width.
%An \emph{$(r,d)$-division} of $M$ is a division which partitions both ${\mathcal P}^R$ and ${\mathcal P}^C$ into $d$ parts, and such that for every $p\in {\mathcal P}^R$ and $p'\in {\mathcal P}^C$ the submatrix ${\mathcal P}\cap p'$ has rank at least $r$ (say in $GF(2)$). Intuitively, one can divide $M$ into many complex zones. 

\subsection{Rank division and rich division}\label{subsec:rds}

%We will now require that the matrix entries are elements of a finite \emph{field} $\mathbb F$.
We recall that a division $\mathcal D$ of a matrix $M$ is a pair $(\mathcal D^R,\mathcal D^C)$, where $\mathcal D^R$ (resp.~$\mathcal D^C$) is a partition of the rows (resp.~columns) of $M$ into (contiguous) intervals, or equivalently, a \opar{0}.
A~\emph{$d$-division} is a division satisfying $|D^R| = |D^C| = d$.
For every pair $R_i \in \mathcal D^R$, $C_j \in \mathcal D^C$, the submatrix $R_i \cap C_j$ may be called~\emph{zone} (or \emph{cell}) of $\mathcal D$ since it is, by definition, a contiguous submatrix of $M$.
We observe that a $d$-division defines $d^2$ zones.

A~\emph{rank-$k$ $d$-division} of $M$ is a $d$-division $\mathcal D$ such that for every $R_i \in {\mathcal D}^R$ and $C_j \in {\mathcal D}^C$ the zone $R_i \cap C_j$ has at least $k$ distinct rows or at least $k$ distinct columns.
A~\emph{rank-$k$ division} is simply a short-hand for a rank-$k$ $k$-division.
The~\emph{grid rank} of a matrix $M$ is the largest integer $k$ such that $M$ admits a rank-$k$ division.
% The~\emph{grid rank} of a matrix class $\mathcal M$, denoted by $\gr(\mathcal M)$, is defined as $\sup \{\gr(M)$ $|$ $M \in \mathcal M\}$.
A class $\mathcal M$ has \emph{bounded grid rank} 
%if $\gr(\mathcal M) < \infty$, or equivalently, 
if there is some integer $k$ such that every matrix $M\in\mathcal M$ has grid rank less than $k$, or equivalently, for every $k$-division $\mathcal D$ of $M$, there is a zone of $\mathcal D$ with less than $k$ distinct rows and less than $k$ distinct columns. 

Closely related to rank divisions, a \emph{$k$-rich division} is a division ${\mathcal D}$ of a matrix $M$ on rows and columns $R \cup C$ such that:
\begin{itemize}
\item for every part $R_a$ of ${\mathcal D}^R$ and for every subset $Y$ of at most $k$ parts in ${\mathcal D}^C$, the submatrix $R_a \cap (C \setminus \cup Y)$ has at least $k$ distinct row vectors, and symmetrically
\item for every part $C_b$ of ${\mathcal D}^C$ and for every subset $X$ of at most $k$ parts in ${\mathcal D}^R$, the submatrix $(R \setminus \cup X) \cap C_b$ has at least $k$ distinct column vectors. 
\end{itemize}
Informally, in a large rich division (that is, a $k$-rich division for some large value of $k$), the diversity in the column vectors within a column part cannot drop too much by removing a controlled number of row parts.   
And the same applies to the diversity in the row vectors.
        
\subsection{Enumerative combinatorics}\label{subsec:enum-comb}

In the context of unordered structures, a graph class $\mathcal C$ is~\emph{small} if there is a constant $c$, such that its number of $n$-vertex graphs bijectively labeled by $[n]$ is at most $n! c^n$.
When considering totally ordered structures, for which the identity is the unique automorphism, one can advantageously drop the labeling and the $n!$ factor.
Indeed, on these structures, counting up to isomorphism or up to equality is the same.
Thus a matrix class $\mathcal M$ is~\emph{small} if there exists a real number $c$ such that the total number of $n \times m$ matrices in $\mathcal M$ is at most $c^{\mathrm{max}(n,m)}$.

Marcus and Tardos~\cite{MarcusT04} showed the following central result, henceforth referred to as \emph{Marcus-Tardos theorem}, which by an argument due to Klazar~\cite{Klazar00} was known to imply the Stanley-Wilf conjecture, that permutation classes avoiding any fixed pattern are small.

\begin{theorem}\label{thm:marcustardos}
There exists a function $\mt: \mathbb N \rightarrow \mathbb N$ such that every $n \times m$ matrix $M$ with at least $\mt(k)\max(n,m)$ non-zero entries has a $k$-division in which every zone contains a non-zero entry.
\end{theorem}

We call $\mt(\cdot)$ the \emph{Marcus-Tardos bound}.
The current best bound is $\mt(k)=\frac{8}{3}(k+1)^22^{4k}=2^{O(k)}$~\cite{Cibulka16}.
Among other things, the Marcus-Tardos theorem is a crucial tool in the development of the theory around twin-width.
In the second paper of the series~\cite{twin-width2}, the Stanley-Wilf conjecture/Marcus-Tardos theorem was generalized to classes with bounded twin-width.
We showed that every graph class with bounded twin-width is small (while proper subclasses of permutation graphs have bounded twin-width~\cite{twin-width1}).
This can be readily extended to every bounded twin-width class of binary structures.
It was further conjectured that the converse holds for hereditary classes: Every hereditary small class of binary structures has bounded twin-width.
We will confirm this conjecture, in the current paper, for the special case of totally ordered binary structures. 

We denote by $\mathcal M_n$, the \emph{$n$-slice} of a matrix class $\mathcal M$, that is the set of all $n \times n$ matrices of $\mathcal M$ (up to isomorphism which preserves the order on the rows and columns).
The \emph{growth} (or~\emph{speed}) of a matrix class is the function $n \in \mathbb N \mapsto |\mathcal M_n|$.
A class $\mathcal M$ has \emph{subfactorial} growth if there is a finite integer beyond which the growth of $\mathcal M$ is strictly less than $n!$;
more formally, if there is $n_0$ such that for every $n \geqslant n_0$, $|\mathcal M_n| < n!$. 
Similarly, $\mathscr C$ being a class of ordered graphs, the \emph{$n$-slice} of $\mathscr C$, $\mathscr C_n$, is the set of $n$-vertex ordered graphs in $\mathscr C$, up to isomorphism.
And the \emph{growth} (or~\emph{speed}) of a class $\mathscr C$ of ordered graphs is the function $n \in \mathbb N \mapsto |\mathscr C_n|$. 

The following result follows from~\cite{twin-width2}.
\begin{theorem}\label{thm:small}
  Let $\CC$ be a class of ordered binary structures of bounded twin-width. Then $\CC$ has growth $c^n$ for some constant $c$.
\end{theorem}

\subsection{Computational complexity}\label{subsec:comp-compl}

First-order (\FO) matrix model checking asks, given a matrix $M$ (or a totally ordered binary structure~$\mathcal S$) and a first-order sentence $\varphi$ (i.e., a formula without any free variable), if $M \models \varphi$ holds.
The atomic formulas in $\varphi$ are of the kinds described in~\cref{subsec:logic}.

We then say that a matrix class $\mathcal M$ is \emph{tractable} if \FO~model checking is fixed-parameter tractable (\FPT) when parameterized by the sentence size and the input matrices are drawn from $\mathcal M$.
That is, $\mathcal M$ is tractable if there exists a constant~$c$ and a computable function $f$, such that $M \models \varphi$ can be decided in time $f(\ell)\,(m+n)^c$, for every $n \times m$ matrix $M \in \mathcal M$ and \FO~sentence $\varphi$ of quantifier depth $\ell$.
We may denote the size of $M$, $n+m$, by $|M|$, and the quantifier depth (i.e., the maximum number of nested quantifiers) of $\varphi$ by~$|\varphi|$.
Similarly a class $\mathscr C$ of binary structures is~\emph{tractable} if \FO~model checking is \FPT~on~$\mathscr C$.

\FO~model checking of general (unordered) graphs is $\AW[*]$-complete~\cite{Downey96}, and thus very unlikely to be \FPT.
Indeed $\FPT \neq \AW[*]$ is a much weaker assumption than the already widely-believed Exponential Time Hypothesis~\cite{Impagliazzo01}, and if false, would in particular imply the existence of a subexponential algorithm solving \textsc{3-SAT}.
In the first paper of the series \cite{twin-width1}, we show that \FO~model checking of general binary structures of bounded twin-width given with an $O(1)$-sequence can even be solved in linear FPT time $f(|\varphi|)\,|U|$, where $U$ is the universe of the structure.
In other words, bounded twin-width classes admitting a~$g(\text{OPT})$-approximation for the contraction sequences are tractable.
It is known for (unordered) graph classes that the converse does not hold.
For instance, the class of all subcubic graphs (i.e., graphs with degree at most~3) is tractable~\cite{Seese96} but has unbounded twin-width~\cite{twin-width2}.
\cref{thm:approx-tww} will show that, on every class of ordered graphs, a fixed-parameter approximation algorithm for the contraction sequence exists. 
Thus every bounded twin-width class of ordered graphs is tractable.
We will also see that the converse holds for hereditary classes of ordered graphs.

\subsection{Interpretations and transductions}\label{subsec:fmt}
Let $\Sigma,\Gamma$ be signatures.
A \emph{simple interpretation} $\mathsf I\from \Sigma\to\Gamma$ 
consists of the following $\Sigma$-formulas: a \emph{domain} formula $\nu(x)$,
and for each relation symbol $R\in\Gamma$ of arity $r$, a formula $\rho_R(x_1,\ldots, x_r)$. If $\mathbf A$ is a $\Sigma$-structure, the $\Gamma$-structure $\mathsf I(\mathbf A)$ has domain $\nu(\mathbf A)=\{v\in A:\mathbf A\models\nu(v)\}$ and the interpretation of a relation symbol $R\in\Sigma$ of arity $r$ is $\rho_R(\mathbf A)\cap \nu(\mathbf A)^{r}$, that is:
\[
R^{\mathsf I(\mathbf A)}=\{(v_1,\dots,v_{r})\in \nu(\mathbf A)^{r}:\mathbf A\models \rho_R(v_1,\dots,v_r)\}.
\]
If $\CC$ is a class of $\Sigma$-structures then denote $\mathsf I(\CC)=\setof{\mathsf I(\str A)}{\str A\in\CC}$.

An important property of (simple) interpretations is that they can be composed:
if $\mathsf I\from \Sigma\to\Gamma$ and $\mathsf J\from \Gamma\to\Delta$ are interpretations,
then there is an interpretation $\mathsf J \circ \mathsf I \from \Sigma\to\Delta$ (computable from $\mathsf I$ and~$\mathsf J$) such that $(\mathsf J\circ \mathsf I)(\str A)=\mathsf J(\mathsf I(\str A))$ for every $\Sigma$-structure $\str A$.
Similarly,
for every $\Sigma$-sentence $\varphi$ there is a sentence $\mathsf I^*(\varphi)$ computable from $\mathsf I$ and $\phi$ such that for every $\Sigma$-structure $\mathbf A$ and we have
\[
	\mathsf I(\mathbf A)\models\varphi\quad\iff\quad\mathbf A\models\mathsf I^*(\varphi).
\]

A class $\CC$ \emph{interprets} a class $\DD$ if there is an interpretation $\mathsf I$ such that $\mathsf I(\CC)\supseteq\DD$.
We say that $\CC$ \emph{efficiently interprets} $\DD$
if additionally there is an algorithm which, given $\str D\in\DD$, computes in time polynomial in the size of $\str D$ a structure $\str C\in\CC$ such that $\mathsf I(\str C)$ is isomorphic to $\str D$. (A structure is represented by the size of its domain written in unary, followed by the adjacency matrices representing each of its relations.)
By composition of interpretations, we conclude that 
if $\CC$ efficiently interprets $\DD$ and $\DD$ efficiently interprets $\mathscr E$, then $\CC$ efficiently interprets $\mathscr E$.

Efficient interpretations are a convenient way for obtaining \FPT~reductions,
as expressed by the following straightforward lemma.
\begin{lemma}
  Suppose that $\CC$ efficiently interprets a class $\DD$.
  Then there is an \FPT~reduction of \FO~model checking on $\DD$ to \FO~model checking on $\CC$:
   there is a computable function $f$, a constant~$c$, and an algorithm which given a structure $\str D\in\DD$ and an \FO~sentence $\phi$ computes 
  in time $f(\phi)\cdot |\str D|^c$ a~structure $\str C\in\CC$ and an \FO~sentence $\psi$ such that 
  \[\str D\models\phi\quad\Leftrightarrow\quad \str C\models \psi.\]
\end{lemma}
% \begin{proof}
%    By assumption, there is an interpretation $\mathsf I$ and an algorithm which, given a graph $G$, computes in time polynomial in $G$ a structure $\str C\in\CC$ such that $\mathsf I(\str C)$ is isomorphic to $G$.
%   Then
%   \[G\models \phi\quad\Leftrightarrow \quad \mathsf I(\str C)\models \phi\quad\Leftrightarrow\quad \str C\models \mathsf I^*(\phi).\]
%   This gives the desired reduction.
% \end{proof}
Since \FO~model checking on the class of all graphs is $\AW[*]$-hard~\cite{Downey96}, we get:
\begin{corollary}\label{cor:AW-hard}
  If $\CC$ efficiently interprets the class of all graphs then model checking on $\CC$ is $\AW[*]$-hard.
\end{corollary}

An important class of ordered graphs which efficiently interprets the class of all graphs is the class $\mathscr M$ of all ordered matchings.
This is expressed by the following folklore result, whose proof is included in~\cref{app:matchings} for completeness.
\begin{lemma}\label{lem:matchings are dependent}
  The class $\mathscr M$ of ordered matchings efficiently interprets the class of all graphs.
\end{lemma}

Let $\Sigma\subseteq \Sigma^+$ be relational signatures.
The \emph{$\Sigma$-reduct} of a $\Sigma^+$-structure $\mathbf A$ is the structure obtained from $\mathbf A$ by ``forgetting'' all the relations not in $\Sigma$. We denote this interpretation as $\reduct_{\Sigma}\from \Sigma^+\to\Sigma$, or simply $\reduct$, when $\Sigma$ is clear from context.

A class $\CC$ of $\Sigma$-structures \emph{transduces} a class $\DD$ if there is 
a class $\CC^+$ of $\Sigma^+$-structures,
where $\Sigma^+$ is the union of $\Sigma$ and some unary relation symbols such that $\reduct_{\Sigma}(\CC^+)=\CC$ and $\CC^+$ interprets $\DD$.

The following result follows from~\cite{twin-width1}.
\begin{theorem}\label{thm:tww implies mNIP}
  Let $\CC$ be a class of ordered, binary structures, and suppose that $\CC$ has bounded twin-width. Then $\CC$ does not transduce the class of all graphs.
\end{theorem}
This result more generally holds for (non necessarily ordered) binary structures.
We only state it in the ordered case, since the definition of twin-width we gave in~\cref{subsec:graph-theory} only fits ordered binary structures.  

\medskip

Fix a binary signature $\Sigma$ containing the symbol $<$.
An \emph{atomic type} $\tau(x_1,\ldots,x_n)$  over $\Sigma$ is a maximal conjunction of atomic formulas or negated atomic formulas with variables $x_1,\ldots,x_n$,
which is satisfiable in some ordered $\Sigma$-structure. (It is sufficient to verify this condition for structures with $n$ elements, since the formulas are quantifier-free.)
If $\bar a$ is an $n$-tuple of elements of an ordered $\Sigma$-structure $\str A$ then \emph{the} atomic type of $\bar a$ is the unique (up to equivalence) atomic type $\tau(x_1,\ldots,x_n)$ satisfied by $\bar a$ in $\str A$.
 For an atomic type $\tau(x,y)$ 
 and ordered $\Sigma$-structure $\str A$
 let $\interp I_\tau(\str A)$ be the ordered graph whose domain and order are the same as in $\str A$,
 and where two vertices $u<v$ are adjacent if and only if $\tau(u,v)$ holds in $\str A$.
Then $\interp I_\tau$ is an interpretation from $\Sigma$ to the signature of ordered graphs.

We formulate a standard lemma reducing the model checking problem for adjacency matrices of structures from a class $\CC$ to the model checking problem for $\CC$.
Let us view here the adjacency matrix $M(\str A)$ of an ordered $\Sigma$-structure $\str A$ as the matrix $|\str A|\times|\str A|$ matrix whose entry at position $(a,b)$, for $a,b\in\str A$, is the atomic type of the pair $(a,b)$.
Hence, $M(\str A)$ is a matrix over the alphabet $A_\Sigma$ consisting of all atomic types $\tau(x,y)$ with two variables.
See~\cref{app:reduction} for a proof of the lemma.

\begin{lemma}\label{lem:reducing matrices to structures}
Let $\CC$ be a class of ordered binary structures and let $\mathcal M=\setof{M(\str A)}{\str A\in\CC}$ be the class of adjacency matrices of structures in $\CC$.  
Then there is an $\FPT$~reduction of the \FO~model checking problem for $\cal M$ to the \FO~model checking problem for $\CC$.
In particular, if the former is $\AW[*]$-hard, so is the latter.
\end{lemma}

\subsection{Model theory}
\label{subsec:mt}

Let $\varphi(\overline{x},\overline{y})$ be a $\Sigma$-formula and let $\mathscr C$ be a class of \mbox{$\Sigma$-structures}.
The formula $\varphi$ is \emph{independent} over $\mathscr C$ if for every binary relation $R\subseteq A\times B$ between two finite sets $A$ and $B$ there exists a $\Sigma$-structure $\mathbf C\in\mathscr C$, some tuples $(\overline{u}_a)_{a\in A}$ in $C^{|\overline x|}$, and $(\overline{v}_b)_{b\in B}$ in $C^{|\overline y|}$ such that 
\[
	\mathbf C\models\varphi(\overline{u}_a,\overline{v}_b)\quad\iff\quad R(a,b)\qquad\qquad\text{for all $a\in A$ and $b\in B$.}
\]
The class $\mathscr C$ is \emph{independent} if there is a $\Sigma$-formula $\varphi(\overline x,\overline y)$ that is independent over $\mathscr C$. 
Otherwise, the class $\mathscr C$ is \emph{dependent} (or \emph{NIP}, for Not the Independence Property).
Note that if a class $\CC$ interprets the class of all graphs, then it is independent.\footnote{The converse also holds if the interpretations can use constant symbols, and we can take induced substructures after performing the interpretation, see~\cite{simon2021note}.}

A \emph{monadic lift} of a class $\mathscr C$ of $\Sigma$-structures is a class $\mathscr C^+$ of $\Sigma^+$-structures, where $\Sigma^+$ is the union of $\Sigma$ and a set of unary relation symbols, and $\mathscr C=\{\reduct_\Sigma(\mathbf A): \mathbf A\in\mathscr C^+\}$. 
 A~class~$\mathscr C$ of $\Sigma$-structures is \emph{monadically dependent} (or \emph{monadically} NIP) if every monadic lift of $\mathscr C$ is dependent (or NIP).
 
 The following theorem witnesses that transductions are particularly fitting to the study of monadic dependence:
 \begin{theorem}[Baldwin and Shelah \cite{BS1985monadic}]\label{thm:baldwin-shelah}
 	A class $\mathscr C$ of $\Sigma$-structures is monadically dependent if and only if for every monadic lift $\mathscr C^+$ of $\mathscr C$ (in $\Sigma^+$-structures), every $\Sigma^+$-formula $\varphi(\overline x,\overline y)$ with $|\overline x|=|\overline y|=1$ is dependent over $\mathscr C^+$.
 	Consequently, $\mathscr C$ is monadically dependent if and only if  $\mathscr C$ does not transduce
   the class $\mathscr G$ of all finite graphs.
 \end{theorem}

%  \begin{corollary}
%  	If $\mathscr C$ transduces  $\mathscr D$ and $\mathscr C$ is monadically dependent then $\mathscr D$ is monadically dependent.
%  \end{corollary}
%  \begin{proof}
%  	Otherwise, the class $\mathscr G$ of all finite graphs is a transduction of $\mathscr D$ and, by composition, a transduction of $\mathscr C$, contradicting the monadic dependence of $\mathscr C$.
%  \end{proof}

\subsection{Patterns}\label{subsec:patterns}
For a permutation $\sigma\in\mathfrak S_n$ and parameter $\sym\in\six$, we define the $n\times n$ matrix $F_\sym(\sigma)$ with entry at row $i$ and column $j$ equal to (see~\cref{fig:six-unavoidable}): 
\begin{align*}
  \bullet&\ \  [\sigma(i) =j]\textit{, if $\sym$ is `$=$'},&
  \bullet&\ \  [\sigma(i) \neq j]\textit{, if $\sym$ is `$\neq$'},\\
  \bullet&\ \  [i\le\sigma^{-1}(j)]\textit{, if $\sym$ is `$\le_R$'},&
  \bullet&\ \  [i\ge\sigma^{-1}(j)]\textit{, if $\sym$ is `$\ge_R$'},\\
  \bullet&\ \  [j \le \sigma(i)]\textit{, if $\sym$ is `$\le_C$'},&
  \bullet&\ \  [j \ge \sigma(i)]\textit{, if $\sym$ is `$\ge_C$'.}
\end{align*}
Here, $[\alpha]$ is the Iverson bracket,
with value $1$ if $\alpha$ holds and  $0$ otherwise.
Let $\mathcal F_\sym$ denote the submatrix closure of the matrices $F_\sym(\sigma)$, for all permutations $\sigma$.

We can also define classes $\mathcal F_\sym$ for $\sym\in\set{<_R,>_R,<_C,>_C}$ analogously as above, but replacing the non-strict inequalities $\le$ and $\ge$ by the strict variants $<$ and $>$. 
While changing the subscript $\sym$ in $F_\sym(\sigma)$ from a non-strict inequality to its strict variant affects the matrix entries, we nevertheless have: 
\[\mathcal F_{<R}=\mathcal F_{\le R},\quad 
\mathcal F_{>R}=\mathcal F_{\ge R},\quad 
\mathcal F_{<C}=\mathcal F_{\le C},\quad 
\mathcal F_{>C}=\mathcal F_{\ge C}.\]

A class $\mathcal M$ of $0,1$-matrices is \emph{pattern-avoiding} if it does not include any of the six matrix classes $\mathcal F_\sym$, for $\sym\in\six$.
We now lift this notion to arbitrary alphabets.

For a finite alphabet $A$, letter $a\in A$, and matrix $M$ over $A$, the \emph{$a$-selection} of $M$ is the $0,1$-matrix $s_a(M)$ obtained from $M$ by replacing each occurrence of $a$ by $1$ and each other letter by $0$.
The $a$-selection of a class $\mathcal M$ of matrices is the class $s_a(\mathcal M)$ of $a$-selections of matrices in $\mathcal M$.
Say that $\mathcal M$ is \emph{pattern-avoiding} if every its $a$-selection $s_a(\mathcal M)$ is pattern-avoiding.

\medskip
  In the introduction we define the 25 classes of ordered graphs $\mathscr P$ and  $\mathscr M_{\sym,\lambda,\rho}$, for 
  $\sym\in\sax$ and $\lambda,\rho\in\set{0,1}$. The parameters $=,\neq,\le_R,\ge_R,\le_C$, and $\ge_C$ 
  used for matrices
  are renamed to $=,\neq,\le_l,\ge_l,\le_r$, and $\ge_r$  in the case of ordered graphs, since rows and columns are interpreted as left and right vertices, respectively.  The classes $\mathscr M_{\sym,\lambda,\rho}$ can be alternatively defined as follows.
  
  % The relationship between $\mathscr M_{\sym,\lambda,\rho}$ and $\mathcal F_\sym$ is as follows.
  \medskip
  Let $H$ be an ordered matching with 
  vertices $a_1<\ldots<a_n<b_1<\ldots<b_n$,
  so that there is some $\sigma\in\mathfrak S_n$ such that $a_i$ is matched with $b_{\sigma(i)}$, for $1\le i\le n$. Then for $\sym\in\sax$, 
  define an ordered graph $H[\sym,\lambda,\rho]$  with vertices $a_1<\ldots<a_n<b_1<\ldots<b_n$ 
such that  $[E(a_i,b_j)]$ (the truth value of the adjacency between $a_i$ and $b_j$) is equal to:
\[[\sigma(i)=j],\quad [\sigma(i)\neq j],\quad [\sigma(i)\le j],\quad [\sigma(i)\ge j],\quad [i \le \sigma^{-1}(j)],~\text{or}~[i\ge \sigma^{-1}(j)],
\]
depending on the parameter $\sym\in\sax$,
and for $1\le i<j\le n$, $[E(a_i,a_j)]=\lambda$
and $[E(b_i,b_j)]=\rho$.
Note that 
$([E(a_i,b_j)])_{1\le i,j\le n}=
F_\sym(\sigma)$, where 
 $\sym$ is now treated as an element of $\six$.
 
 The class $\mathscr M_{\sym,\lambda,\rho}$
 is the hereditary closure of the class of all ordered graphs $H[\sym,\lambda,\rho]$, where $H$ is an ordered matching.
 The class $\mathscr P$ is the class 
 of all {ordered permutation graphs},
 and satisfies the following properties
 (see~\cref{subsec:basic-notions}).

 \begin{lemma}\label{lem: class P}
   The class $\mathscr P$ is hereditary, closed under edge complements, and has growth $n!$.
 \end{lemma}

\subsection{Ramsey theory}\label{subsec:ram}

We recall Ramsey's theorem.

\begin{theorem}[Ramsey's theorem \cite{Ramsey1930}]
 \label{th: Ramsey}
 There exists a function $\ramm{\cdot}{\cdot}~:~\Nn \times \Nn \to \Nn$ such that for every $k\geq 1$, $t \geq 1$ 
 the complete graph $K_{\ramm{k}{t}}$ with edges colored by $t$ distinct colors contains a monochromatic clique on $k$ vertices, i.e., a clique whose edges all have the same color. 
\end{theorem}

For every $p \geq 0$ we will denote with ${\rm R}_{t}^{(p)}(\cdot)$ the function ${\rm R}_{t}(\cdot)$ iterated $p$ times.
A well-known variant of~\cref{th: Ramsey} for complete bipartite graphs is the following: 

\begin{theorem}[Bipartite Ramsey's theorem]
 \label{th: bipRamsey}
 There exists a function $\bipramm{\cdot}{\cdot}~:~\Nn \times \Nn \to \Nn$ such that for every $k\geq 1$, $t \geq 1$ 
 the complete graph $K_{\bipramm{k}{t},\bipramm{k}{t}}$ with edges colored by $t$ distinct colors contains a monochromatic biclique $K_{k,k}$. 
\end{theorem}
\medskip
The \emph{order type} of a pair $(x,y)$ of elements of a totally ordered set is the integer $\ot(x,y)$ defined by
\[
	\ot(x,y)=\begin{cases}
		-1&\text{if }x>y\\
		0&\text{if }x=y\\
		1&\text{if }x<y.
	\end{cases}
\]

For structures with two orders, we will use a convenient specific result from Ramsey theory, which  is a special case of the so-called \emph{product Ramsey theorem}
(see e.g.~Proposition 3 in~\cite{bodirsky_2015}
in the special case of the full product of two copies of $(\mathbb Q,<)$. See also the historical comment following it).

For a set with two total orders $(X,<_1,<_2)$, let $\ot_1(x,y)$  denote the order type of $x,y$ with respect to~$<_1$, while $\ot_2(x,y)$ the order type with respect to $<_2$. 

\begin{lemma}\label{lem:grid-ramsey}
Fix a finite set of colors $\Gamma$ with $t:=|\Gamma|$. There exists a function $\gridramm{\cdot}{t} ~ :\Nn \to \Nn$ such that
for every finite set with two total orders $\str M=([k],<_1,<_2)$ there is another finite set with two total orders $\str N=([N],<_1,<_2)$ where $N=\gridramm{k}{t}$, such that
for every coloring $c\from [N]^2\to \Gamma$ there is a substructure $\str M'$ of $\str N$ isomorphic to $\str M$  such that 
 $c(p,q)$ depends only on
 $\ot_{1}(p,q)$ and $\ot_{2}(p,q)$, for all distinct $p,q\in \str M'$. More precisely,  $c(p,q)=\gamma(\ot_1(p,q),\ot_2(p,q))$
 holds for some function $\gamma\from \set{-1,0,1}^2\to \Gamma$.
\end{lemma}

\cref{lem:grid-ramsey} translates to the following statement on permutations.

\begin{lemma}\label{lem:perm-ramsey}
  Fix a finite set of colors $\Gamma$. For every $k\ge 1$ and permutation $\sigma\in\mathfrak S_k$ there is $N\ge 1$ and a permutation $\pi\in\mathfrak S_N$ such that for every coloring $c\from [N]^2\to \Gamma$ there is a set $U\subseteq [N]$ of size $k$ such that $\sigma$ is the subpermutation of $\pi$ induced by $U$, and $c(i,j)$ depends only on 
  $\ot(i,j)$ and $\ot(\pi(i),\pi(j))$, for all $i,j\in U$.
\end{lemma}

\section{Effective equivalence of bounded twin-width and no large rich division}\label{sec:approx-tww}

In this section we show the equivalence between \cref{it:bd-tww} and \cref{it:bd-rd}.
As a by-product, we obtain an $f(\text{OPT})$-approximation algorithm for the twin-width of matrices, or ordered graphs.  
We first show that a large rich division implies large twin-width.
This direction is crucial for the algorithm but \emph{not} for the main circuit of implications.

\begin{lemma}\label{lem:richcertificate}
If $M$ has a $2k(k+1)$-rich division ${\mathcal D}$, then $\tww(M) > k$.
\end{lemma}

\begin{proof}
  We prove the contrapositive.
  Let $M$ be a matrix of twin-width at most~$k$.
  In particular, $M$ admits a $(k,k)$-sequence $\mathcal P_1, \ldots, \mathcal P_{n+m-1}$.
  Let $\mathcal D$ be any division of $M$.
  We want to show that $\mathcal D$ is \emph{not} $2k(k+1)$-rich.
  
  Let $t$ be the smallest index such that either a part $R_i$ of ${\mathcal P}_t^R$ intersects three parts of ${\mathcal D}^R$, or a part $C_j$ of ${\mathcal P}_t^C$ intersects three parts of~${\mathcal D}^C$.
  Without loss of generality we can assume that $C_j \in {\mathcal P}_t^C$ intersects three parts $C'_a,C'_b,C'_c$ of ${\mathcal D}^C$, with $a<b<c$ where the parts $C'_1, \ldots, C'_d$ of the division $\mathcal D$ are ordered from left to right.
  Since ${\mathcal P}^C_t$ is a \kop, the subset $S$, consisting of the parts of ${\mathcal P}^C_t$ intersecting $C'_b$, has size at most $k+1$.
  Indeed, $S$ contains $C_j$ plus at most $k$ parts which $C_j$ is in conflict with.
  
  Here a part $R'_s$ of ${\mathcal D}^R$ is called \emph{red} if there exist a part $R_i$ of ${\mathcal P}_t^R$ intersecting $R'_s$ and a part $C_z$ in $S$ such that the submatrix $R_i \cap C_z$ is not constant (see~\cref{fig:bdtww-to-brd}).
  We then say that $C_z$ is a \emph{witness} of $R'_s$ being red.
  Let $N \subseteq R$ be the subset of rows \emph{not} in a red part of ${\mathcal D}^R$.
  Note that for every part $C_z \in S$, the submatrix $N \cap C_z$ consists of the same column vector repeated $|C_z|$ times.
  Therefore $N \cap C'_b$ has at most $k+1$ distinct column vectors.

  \tikzexternaldisable 
  \begin{figure}[h!]
    \centering
    \begin{tikzpicture}
      \def\xb{0}
      \def\xe{10}
      \def\yb{0}
      \def\ye{6}
      \def\bb{-0.4}
      \def\be{-0.1}
      %partition h
      \foreach \i/\j/\c in {0.5/0.8/orange, 2.5/3.2/orange, 5/5.6/orange, 7.8/8/orange, 8.6/9.4/orange, 5.6/6/blue, 3.4/3.8/blue, 2.2/2.5/black!40!green, 3.2/3.4/black!10!yellow, 4.45/4.8/black!10!yellow, 4.8/5/black!60!gray, 3.8/4.45/brown!85!blue}{
        \fill[\c,opacity=0.6] (\i,\yb) -- (\i,\ye) -- (\j,\ye) -- (\j,\yb) -- cycle ;
      }
      \node at (4.14,-0.22) {$C_z$} ;
      \node[circle,inner sep=-0.03cm] (cj) at (4.7,-1) {$C_j$} ;
      \foreach \i in {0.65, 2.85, 5.3, 7.9, 9}{
        \draw[very thick, orange,opacity=0.6] (cj) -- (\i,0.02) ;
      }
      \node [circle,inner sep=-0.03cm] (ri) at (\xe+0.5,3.45) {$R_i$} ;
      \foreach \i in {2.5,4.5}{
        \draw[very thick, cyan,opacity=0.6] (ri) -- (\xe-0.02,\i) ;
      }
      %partition v
      \foreach \i/\j/\c in {4.3/4.7/cyan, 2.4/2.6/cyan}{
        \fill[\c,opacity=0.6] (\xb,\i) -- (\xe,\i) -- (\xe,\j) -- (\xb,\j) -- cycle ;
      }
      %division h
      \foreach \i in {0.2, 2.2, 6, 7.6, 9.8}{
        \draw[thick] (\i,\yb) -- (\i,\ye) ;
      }
      \foreach \i/\j in {1.1/a,4.1/b,8.7/c}{
        \node at (\i, \ye + 0.4) {$C'_\j$} ;
      }
       %division v
      \foreach \j in {0.2, 0.8, 1.4, 1.75, 2.4, 2.8, 3.1, 4.1, 5, 5.4, 5.9}{
        \draw[thick] (\xb,\j) -- (\xe,\j) ;
      }
      \foreach \i/\j/\c in {2.4/2.8/red, 4.1/5/red}{
        \fill[\c] (\bb,\i) -- (\be,\i) -- (\be,\j) -- (\bb,\j) -- cycle ;
      }
      \node at (-0.65,4.55) {$R'_s$} ;
       %zone
      \draw[very thick,red] (3.8,4.3) -- (4.45,4.3) -- (4.45,4.7) -- (3.8,4.7) -- cycle ;
      \draw[very thick,red] (3.8,2.4) -- (4.45,2.4) -- (4.45,2.6) -- (3.8,2.6) -- cycle ;
      \node at (4.125,4.5) {\textcolor{red}{NC}} ;
    \end{tikzpicture}
    \caption{The division $D$ in black. The column part $C_j \in \mathcal P_t^C$, first to intersect three division parts, in orange. Two row parts of $\mathcal D$ turn red because of the non-constant submatrix $C_z \cap R_i$, with $C_z \in S$ and $R_i \in \mathcal D^R$. After removal of the at most $2k|S|$ red parts, $|S| \leqslant k+1$ bounds the number of distinct columns.}
    \label{fig:bdtww-to-brd}    
  \end{figure}
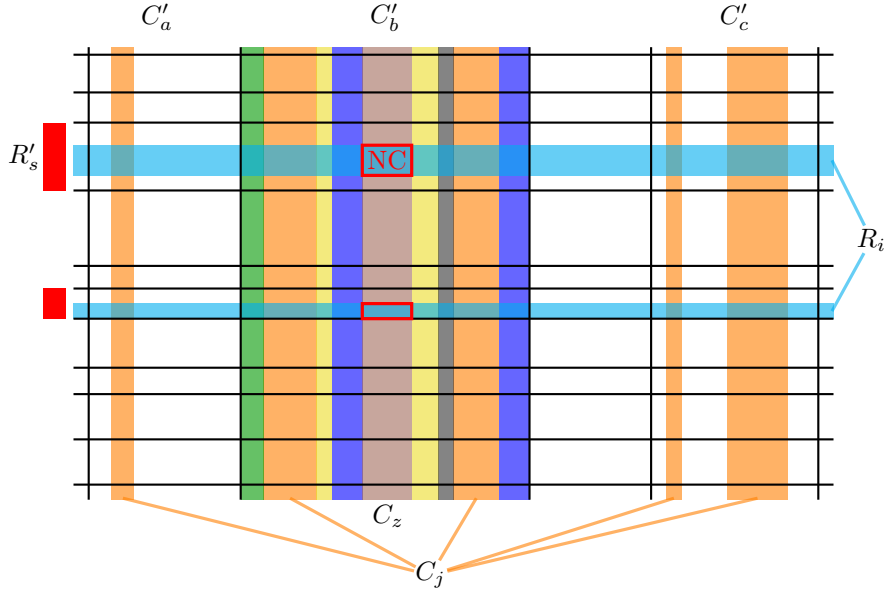
  \tikzexternalenable
  
  Besides, the number of red parts witnessed by $C_z \in S$ is at most $2k$.
  This is because the number of non-constant submatrices $R_i \cap C_z$, with $R_i \in {\mathcal P}_t^R$, is at most $k$ (since $\mathcal P_1, \ldots, \mathcal P_{n+m-1}$ is a $(k,k)$-sequence) and because every $R_i$ intersects at most two parts of ${\mathcal D}^R$ (by definition of $t$).
  Hence the total number of red parts is at most~$2k|S|$, thus at most~$2k(k+1)$.
  Consequently, there is a subset $X$ of at most $2k(k+1)$ parts of ${\mathcal D}^R$, namely the red parts, and a part $C'_b$ of ${\mathcal D}^C$ such that $(R \setminus \cup X) \cap C'_b = N \cap C'_b$ consists of at most $k+1$ distinct column vectors.
  Thus ${\mathcal D}$ is not a $2k(k+1)$-rich division.
\end{proof}

Our main algorithmic result is that  approximating the twin-width of matrices (or ordered graphs) is \FPT.
Let us observe that this remains a challenging open problem for (unordered) graphs.
We finally state~\cref{thm:approx-tww} in the language of matrices, since rich divisions are only natural in that setting.
The previous statement of~\cref{thm:approx-tww}, for ordered binary structures, readily follows.

Indeed recall that the twin-width of an ordered binary structure is defined as the twin-width of its matrix encoding.
Besides a contraction sequence for the matrix can be turned into a contraction sequence for the ordered binary structure.
Every time the $i$-th and $j$-th, say, column parts are merged, we symmetrically merge the $i$-th and $j$-th row parts.
Because our matrix encoding is mixed-symmetric, this produces a sequence with the same error value (see~\cite[Theorem 14]{twin-width1}).
The now-symmetric sequence can then be interpreted as contracting the vertices of a graph, or more generally, the domain elements of a binary structure.  

\begin{reptheorem}{thm:approx-tww}
Given as input an $n \times m$ matrix $M$ over a finite alphabet $A$, and an integer $k$, there is an $2^{2^{O(k^2 \log k)}}(n+m)^{O(1)}$ time algorithm which returns
\begin{itemize}
\item either a $2k(k+1)$-rich division of $M$, certifying that $\tww(M) > k$, 
\item or an $(|A|^{O(k^4)},|A|^{O(k^4)})$-sequence, certifying that $\tww(M) = |A|^{O(k^4)}$.
\end{itemize}
\end{reptheorem}

\begin{proof}
  We try to construct a division sequence ${\mathcal D}_1,\dots, {\mathcal D}_{n+m-1}$ of $M$ such that every ${\mathcal D}_i$ satisfies the following properties $\mathscr P^R$ and $\mathscr P^C$.
 Let $r$ be equal to $4k(k+1)+1$.
\begin{compactitem}
\item $\mathscr P^R$: For every part $R_a$ of $\mathcal D_i^R$, there is a set $Y$ of at most $r$ parts of ${\mathcal D}_i^C$, such that the submatrix $R_a \cap (C \setminus \cup Y)$ has at most $r-1$ distinct row vectors. 
\item $\mathscr P^C$: For every part $C_b$ of $\mathcal D_i^C$, there is a set $X$ of at most $r$ parts of ${\mathcal D}_i^R$, such that the submatrix $(R\setminus \cup X) \cap C_b$ has at most $r-1$ distinct column vectors. 
\end{compactitem}
The algorithm is greedy: Whenever we can merge two consecutive row parts or two consecutive column parts in ${\mathcal D}_i$ so that the above properties are preserved, we do so, and obtain ${\mathcal D}_{i+1}$.
We first show that checking properties $\mathscr P^R$ and $\mathscr P^C$ can be done in fixed-parameter time.

\begin{lemma}\label{lem:testingPR-PC}
Whether $\mathscr P^R$, or $\mathscr P^C$, holds can be decided in time $2^{2^{O(k^2 \log k)}} (n+m)^{O(1)}$.
\end{lemma}
\begin{proof}
  We show the lemma with $\mathscr P^R$, since the case of $\mathscr P^C$ is symmetric.
  For every $R_a \in \mathcal D_i^R$, we denote by $\mathscr P^R(R_a)$ the fact that $R_a$ satisfies the condition $\mathscr P^R$ starting at ``there is a set $Y$.''
  If one can check $\mathscr P^R(R_a)$ in time $T$, one can thus check $\mathscr P^R$ and $\mathscr P^C$ for the current division $\mathcal D_i$ in time $(|\mathcal D_i^R|+|\mathcal D_i^C|)T \leqslant (n+m)T$.

  To decide $\mathscr P^R(R_a)$, we initialize the set $Y$ with all the column parts $C_b \in \mathcal D_i^C$ such that the zone $R_a \cap C_b$ contains more than $r-1$ distinct rows.
  Indeed these parts \emph{have to} be in $Y$.
  At this point, if $R_a \cap (C \setminus \cup Y)$ has more than $(r-1)^{r+1}$ distinct rows, then $\mathscr P^R(R_a)$ is false.
  Indeed, each further removal of a column part divides the number of distinct rows in $R_a$ by at most $r-1$.
  Thus after at most $r$ further removals, more than $r-1$ distinct rows would remain.

  Let us suppose instead that $R_a \cap (C \setminus \cup Y)$ has at most $(r-1)^{r+1}$ distinct rows.
  We keep one representative for each distinct row.
  For every $C_b \in \mathcal D_i^C \setminus Y$, the number of distinct \emph{columns} in zone $R_a \cap C_b$ is at most $|A|^{r-1}$.
  In each of these zones, we keep only one representative for every occurring column vector.
  Now every zone of $R_a$ has dimension at most $(r-1)^{r+1} \times |A|^{r-1}$.
  Therefore the maximum number of distinct zones is $\exp(\exp(O(r \log r)))=\exp(\exp(O(k^2 \log k)))$.

  If a same zone $Z$ is repeated in $R_a$ more than $r$ times, at least one occurrence of the zone will not be included in $Y$.
  In that case, putting copies of $Z$ in $Y$ is pointless: it eventually does not decrease the number of distinct rows.
  Thus if that happens, we keep exactly $r+1$ copies of~$Z$.
  Now $R_a$ has at most $(r+1) \cdot \exp(\exp(O(k^2 \log k))) = \exp(\exp(O(k^2 \log k)))$ zones.
  We can try out all $\exp(\exp(O(k^2 \log k)))^r$, that is, $\exp(\exp(O(k^2 \log k)))$ possibilities for the set $Y$, and conclude whether or not one of them works.
\end{proof}
  
  Two cases can arise.
  
\medskip

\textbf{Case 1.} The algorithm terminates on some division ${\mathcal D}_i$ and no merge is possible.

\noindent Let us assume that ${\mathcal D}_i^R:=\{R_1, \ldots, R_s\}$ and ${\mathcal D}_i^C:=\{C_1, \ldots, C_t\}$, where the parts are ordered by increasing vector indices.
We consider the division $\mathcal D$ of $M$ obtained by merging in $\mathcal D_i$ the pairs $\{R_{2a-1},R_{2a}\}$ and $\{C_{2b-1},C_{2b}\}$, for every $1 \leqslant a \leqslant \lfloor s/2 \rfloor$ and $1 \leqslant b \leqslant \lfloor t/2 \rfloor$. 
%If $s$ is odd, we also merge  $\{R_s\}$ with the last part of ${\mathcal D}^R$. Similar merge if $t$ is odd.
Let $C'_j$ be any column part of ${\mathcal D}^C$.
Since the algorithm has stopped, for every set~$X$ of at most $(r-1)/2$ parts of $\mathcal D^R$, the matrix $(R \setminus \cup X) \cap C'_j$ has at least $r$ distinct (column) vectors.
This is because $(r-1)/2$ parts of~${\mathcal D}^R$ corresponds to at most $r-1$ parts of ${\mathcal D}_i^R$.
The same applies to the row parts, so we deduce that $\mathcal D$ is $(r-1)/2$-rich, that is, $2k(k+1)$-rich.
Therefore, by \cref{lem:richcertificate}, $M$~has twin-width greater than $k$.
   
\medskip

\textbf{Case 2.} The algorithm terminates with a full sequence ${\mathcal D}_1,\dots, {\mathcal D}_{n+m-1}$. 

\noindent Given a division ${\mathcal D}_i$ with ${\mathcal D}_i^R:=\{R_1,\dots,R_s\}$ and ${\mathcal D}_i^C:=\{C_1,\dots,C_t\}$, we now define a partition ${\mathcal P}_i$ that refines ${\mathcal D}_i$ and has small error value.
To do so, we fix a, say, column part $C_j$ and show how to partition it further in ${\mathcal P}_i$.

By assumption on $\mathcal D_i$, there exists a subset $X$ of at most $r$ parts of ${\mathcal D}_i^R$ such that $(R \setminus \cup X) \cap C_j$ has less than $r$ distinct column vectors.
We now denote by $F$ the set of parts $R_a$ of ${\mathcal D}_i^R$ such that the zone $R_a \cap C_j$ has at least $r$ distinct rows and $r$ distinct columns.
Such a zone is called \emph{full}.
Observe that $F \subseteq X$.
Moreover, for every $R_a$ in $X \setminus F$, the total number of distinct column vectors in $R_a \cap C_j$ is at most $\max(r,|A|^{r-1})=|A|^{r-1}$, assuming that the alphabet $A$ has at least two letters.
Indeed, if the number of distinct columns in $R_a \cap C_j$ is at least $r$, then the number of distinct rows is at most $r-1$.

In particular, the total number of distinct column vectors in $(R \setminus \cup F) \cap C_j$ is at most $w := r(|A|^{r-1})^r$; a multiplicative factor of $|A|^{r-1}$ for each of the at most $r$ zones $R_a \in X \setminus F$, and a multiplicative factor of~$r$ for $(R \setminus \cup X) \cap C_j$.
We partition the columns of $C_j$ accordingly to their subvector in $(R\setminus \cup F) \cap C_j$ (by grouping columns with equal subvectors together).
The partition ${\mathcal P}_i$ is obtained by refining, as described for $C_j$, all column parts and all row parts of ${\mathcal D}_i$.

By construction, ${\mathcal P}_i$ is a refinement of ${\mathcal P}_{i+1}$ since every full zone of ${\mathcal D}_i$ remains full in ${\mathcal D}_{i+1}$.
Hence if two columns belong to the same part of~${\mathcal P}_i$, they continue belonging to the same part of~${\mathcal P}_{i+1}$.
Besides, ${\mathcal P}_i$ is a \opar{w} of $M$, and its error value is at most $r \cdot w$ since non-constant zones can only occur in full zones (at most $r$ per part of $\mathcal D_i$), which are further partitioned at most $w$ times in~${\mathcal P}_i$.
To finally get a contraction sequence, we greedily merge parts to fill the intermediate partitions between ${\mathcal P}_i$ and ${\mathcal P}_{i+1}$.
Note that all intermediate refinements of $\mathcal P_{i+1}$ are \opars{w}.
Moreover the error value of a column part does not exceed $r \cdot w$.
Finally the error value of a row part can increase during the intermediate steps by at most $2w$.
All in all, we get a $(w,(r+2) \cdot w)$-sequence.
This implies that $M$ has twin-width at most~$(r+2) \cdot w = |A|^{O(k^4)}$.

The running time of the overall algorithm follows from \cref{lem:testingPR-PC}. 
\end{proof}

%In particular, if the matrix has twin-width $k$, we can find in polytime a sequence for twin width at most $r.w$, which is order of $\alpha ^{r^4}$. However, this bound can be improved by observing that bounded twin width implies bounded VC-dimension $d$ (close to $k$), thus we can change $\alpha^{r-1}$ into roughly $r^d$. Not a big deal.

The approximation ratio, of $2^{O(\text{OPT}^4)}$, can be analyzed more carefully by observing that bounded twin-width implies bounded VC dimension.
Then the threshold $|A|^{r-1}$ can be replaced by $r^d$, where $d$ upperbounds the VC dimension.
As a direct corollary of our algorithm, if the matrix $M$ does not admit any large rich division, the only possible outcome is a contraction sequence.
Considering the size of $A$ as an absolute constant, we thus obtain the following.
\begin{theorem}\label{thm:brd-to-bdtww}
If $M$ has no $r$-rich division, then $\tww(M) = 2^{O(r^2)}$.
\end{theorem}
This is the direction which is important for the circuit of implications.
The algorithm of~\cref{thm:approx-tww} further implies that \cref{thm:brd-to-bdtww} is effective.

\section{Large rich divisions imply large rank divisions}\label{sec:rich-to-gr}

In this section we show how to extract a large rank division from a huge rich division.
We remind the reader that a rank-$k$ division is a $k$-division in which every zone has at least~$k$ distinct rows or at least~$k$ distinct columns.
%A $(k+1)$-rank division is a $k$-rich division since the deletion of $k$ zones in a column of the division leaves a zone with rank at least $k$, hence with at least $k$ distinct row vectors.
%The goal of this section is to provide a weak converse of that statement.
We recall that $\mt$ is the Marcus-Tardos bound of~\cref{thm:marcustardos}.
%For simplicity, we show the following theorem in the case $\mathbb F = \mathbb F_2$, but the proof readily extends to any finite field by setting $K$ to $|\mathbb F|^{|\mathbb F|^k \mt(k|\mathbb F|^k)}$.

\begin{theorem}\label{thm:rd-to-gr}
  Let $A$ be a finite set, and $K := |A|^{k \mt(k^2)}$.
  Every $A$-matrix $M$ with a $K$-rich division ${\mathcal D}$ has a~rank-$k$ division.
\end{theorem}

\begin{proof}
  Without loss of generality, we can assume that ${\mathcal D}^C$ has size at least the size of ${\mathcal D}^R$.
  We color \emph{red} every zone of $\mathcal D$ which has at least $k$ distinct rows or at least $k$ distinct columns.
  We now color \emph{blue} a zone $R_i \cap C_j$ of $\mathcal D$ if it contains a row vector $r$ (of length~$|C_j|$) which does not appear in any non-red zone $R_{i'} \cap C_j$ with $i'< i$.
  We then call $r$ a \emph{blue witness} of $R_i \cap C_j$.
  
  Let us now denote by $U_j$ the subset of ${\mathcal D}^R$ such that every zone $R_i \cap C_j$ with $R_i \in U_j$ is \emph{uncolored}, i.e., neither red nor blue.
  Since the division $\mathcal D$ is $K$-rich, if the number of \emph{colored} (i.e., red or blue) zones $R_i \cap C_j$ is less than $K$, the matrix $(\cup U_j )\cap C_j$ has at least $K$ distinct column vectors.
  So $(\cup U_j )\cap C_j$ has at least $\log_{|A|} K = k \mt(k^2)$ distinct \emph{row} vectors.
  By design, every row vector appearing in some uncolored zone $R_i \cap C_j$ must appear in some blue zone $R_{i'} \cap C_j$ with $i' < i$.
  Therefore at least $k \mt(k^2)$ distinct row vectors must appear in some blue zones within column part $C_j$.
  Since a blue zone contains less than $k$ distinct row vectors (otherwise it would be a red zone), there are, in that case, at least $k \mt(k^2)/k = \mt(k^2)$ blue zones within $C_j$.
  Therefore in any case, the number of colored zones $R_i \cap C_j$ is at least $\mt(k^2)$ per $C_j$.

  Thus, by \cref{thm:marcustardos}, we can find $D'$ a $k^2 \times k^2$ division of~$M$, coarsening~$D$, with at least one colored zone of $D$ in each \emph{cell} of $D'$.
  Now we consider $D''$ the $k \times k$ division of $M$, coarsening $D'$, where each \emph{supercell} of $D''$ corresponds a $k \times k$ square block of cells of $D'$ (see~\cref{fig:rd-to-gr}).
  Our goal is to show that every supercell $Z$ of $D''$ has at least $k$ distinct rows or $k$ distinct columns.
  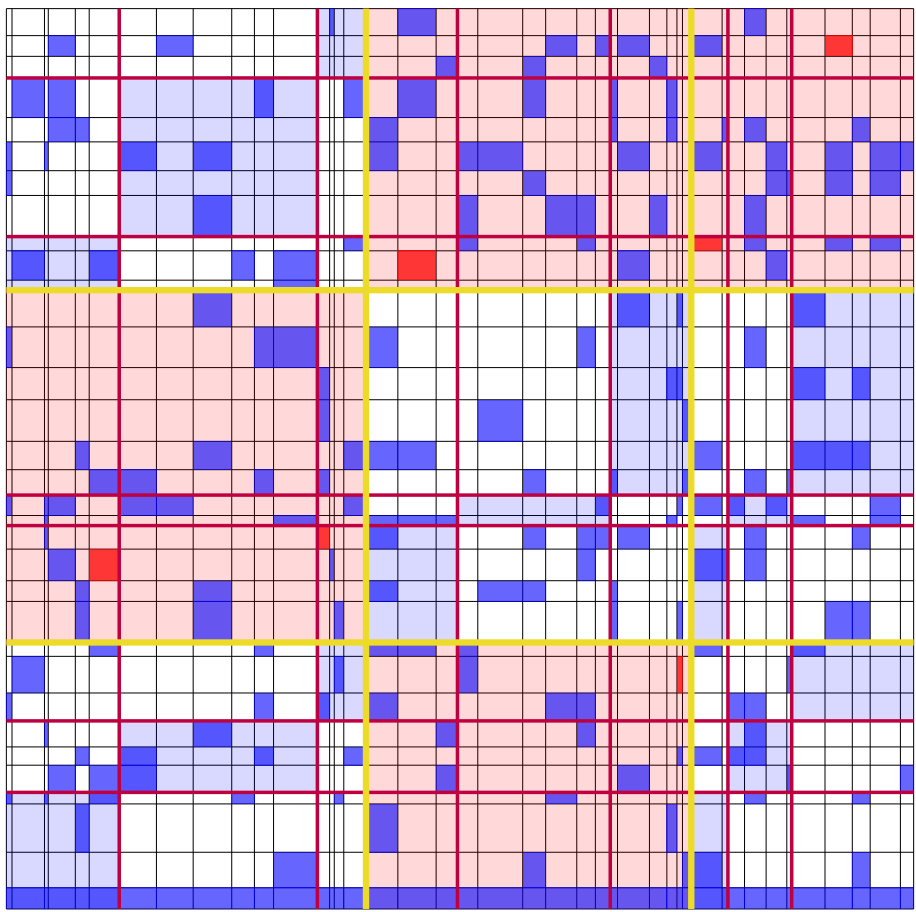
\begin{figure}[h!]
    \centering
    \resizebox{350pt}{!}{
    \begin{tikzpicture}
      \def\hb{0}
      \def\he{10}
       \def\vb{0}
       \def\ve{10}
       \def\e{0}
       %vertical sep
       \def\vert{{0, 0.062, 0.42, 0.462, 0.76, 0.913, 1.244, 1.654, 2.059, 2.484, 2.734, 2.945, 3.428, 3.563, 3.614, 3.719, 3.964, 4.315, 4.737, 4.972, 5.196, 5.691, 5.944, 6.287, 6.492, 6.656, 6.736, 7.088, 7.28, 7.391, 7.452, 7.549, 7.889, 7.953, 8.136, 8.374, 8.602, 8.656, 9.023, 9.324, 9.521, 9.854, 10}}
       %horizontal sep
       \def\hori{{0, 0.235, 0.63, 1.167, 1.295, 1.597, 1.8, 2.089, 2.399, 2.806, 2.96, 3.416, 3.646, 3.997, 4.258, 4.37, 4.597, 4.879, 5.194, 5.656, 6.013, 6.465, 6.873, 6.984, 7.31, 7.468, 7.925, 8.2, 8.52, 8.79, 9.23, 9.47, 9.7, 10}}
       %rich division D
       \foreach \i in {0,...,42}{
         \draw[very thin] (\vert[\i],\vb) -- (\vert[\i],\ve) ;
       }
       \foreach \j in {0,...,33}{
         \draw[very thin] (\hb,\hori[\j]) -- (\he,\hori[\j]) ;
       }
       %meta red zone
       \foreach \i/\j/\k/\l in {0/10/16/22,16/22/31/33,31/22/42/33,16/0/31/10}{
         \fill[red,opacity=0.15] (\vert[\i]+\e,\hori[\j]+\e) -- (\vert[\k]-\e,\hori[\j]+\e) -- (\vert[\k]-\e,\hori[\l]-\e) -- (\vert[\i]+\e,\hori[\l]-\e) -- cycle ;
       }
       %meta blue zone
       \foreach \i/\j/\k/\l in {0/0/6/4,6/4/12/7,12/7/16/10, 16/10/19/14,19/14/25/16,25/16/31/22, 31/10/33/14,33/14/37/16,37/16/42/22, 31/0/33/4,33/4/37/7,37/7/42/10, 0/22/6/25,6/25/12/30,12/30/16/33}{
         \fill[blue,opacity=0.15] (\vert[\i]+\e,\hori[\j]+\e) -- (\vert[\k]-\e,\hori[\j]+\e) -- (\vert[\k]-\e,\hori[\l]-\e) -- (\vert[\i]+\e,\hori[\l]-\e) -- cycle ;
       }
       %blue zones
       \pgfmathsetseed{1} 
       \foreach \i in {0,...,41}{
         \foreach \j in {0,...,32}{
         \pgfmathsetmacro{\ip}{\i+1}
         \pgfmathsetmacro{\jp}{\j+1}
         \pgfmathtruncatemacro\x{7 * random()}
         \ifthenelse{\x = 1}{\fill[blue,opacity=0.6] (\vert[\i]+\e,\hori[\j]+\e) -- (\vert[\ip]-\e,\hori[\j]+\e) -- (\vert[\ip]-\e,\hori[\jp]-\e) -- (\vert[\i]+\e,\hori[\jp]-\e) -- cycle ;}{}
         }
       }
       \foreach \i in {0,...,4,6,7,8,10,11,12,14,15,...,41}{
         \pgfmathsetmacro{\ip}{\i+1}
         \fill[blue,opacity=0.6] (\vert[\i]+\e,\hori[0]+\e) -- (\vert[\ip]-\e,\hori[0]+\e) -- (\vert[\ip]-\e,\hori[1]-\e) -- (\vert[\i]+\e,\hori[1]-\e) -- cycle ;
       }
       \foreach \i/\j in {31/9,31/31}{
         \pgfmathsetmacro{\ip}{\i+1}
         \pgfmathsetmacro{\jp}{\j+1}
         \fill[blue,opacity=0.6] (\vert[\i]+\e,\hori[\j]+\e) -- (\vert[\ip]-\e,\hori[\j]+\e) -- (\vert[\ip]-\e,\hori[\jp]-\e) -- (\vert[\i]+\e,\hori[\jp]-\e) -- cycle ;
       }
       %red zones
       \foreach \i/\j in {17/23,38/31,29/8,5/12,12/13,31/24}{
         \pgfmathsetmacro{\ip}{\i+1}
         \pgfmathsetmacro{\jp}{\j+1}
         \fill[red,opacity=0.75] (\vert[\i]+\e,\hori[\j]+\e) -- (\vert[\ip]-\e,\hori[\j]+\e) -- (\vert[\ip]-\e,\hori[\jp]-\e) -- (\vert[\i]+\e,\hori[\jp]-\e) -- cycle ;
       }
       %Marcus-Tardos grid D'
       \foreach \i in {6,12,16,19,25,31,33,37}{
         \draw[purple,line width=1.1pt] (\vert[\i],\vb) -- (\vert[\i],\ve) ;
       }
       \foreach \j in {4,7,10,14,16,22,25,30}{
         \draw[purple,line width=1.1pt] (\hb,\hori[\j]) -- (\he,\hori[\j]) ;
       }
       %grid rank D''
       \foreach \i in {16,31}{
         \draw[black!10!yellow,line width=2pt] (\vert[\i],\vb) -- (\vert[\i],\ve) ;
       }
       \foreach \j in {10,22}{
         \draw[black!10!yellow,line width=2pt] (\hb,\hori[\j]) -- (\he,\hori[\j]) ;
       }
    \end{tikzpicture}
    }
    \caption{In black (purple, and yellow), the rich division~$D$.
      In purple (and yellow), the Marcus-Tardos division $D'$ with at least one colored zone of $D$ per cell.
      In yellow, the rank-$k$ division~$D''$.
      Each supercell of $D''$ has large rank, either because it contains a red zone (light red) or because it has a diagonal of cells of $D'$ with a blue zone (light blue).}
    \label{fig:rd-to-gr}
  \end{figure}
  If the supercell $Z$ contains a red zone of $D$, the property immediately holds for $Z$.
  If not, each of the $k \times k$ cells of $D'$ within the supercell $Z$ contains at least one blue zone of $D$.
  Let $Z_{i,j}$ be the cell in the $i$-th row block and $j$-th column block of hypercell $Z$, for every $i, j \in [k]$.
  Consider the diagonal cells $Z_{i,i}$ ($i \in [k]$) of $D'$ within the supercell $Z$.
  In each of them, there is at least one blue zone witnessed by a row vector, say, $\tilde{r_i}$.
  Let $r_i$ be the prolongation of $\tilde{r_i}$ up until the two vertical limits of $Z$.
  We claim that every $r_i$ (with $i \in [k]$) is distinct.
  Indeed by definition of a blue witness, if $i < j$, $\tilde{r_j}$ is different from all the row vectors below it, in particular from $r_i$ restricted to these columns.
  So $Z$ has at least $k$ distinct row vectors.
\end{proof}

\section{Rank Latin divisions}\label{sec:latin}

In this section, we show a Ramsey-like result which establishes that every (hereditary) matrix class with unbounded grid rank can encode all the \mbox{$n$-permutations} with some of its $2n \times 2n$ matrices.
In particular and in light of the previous sections, this proves the small conjecture for ordered graphs.

We recall that a rank-$k$ $d$-division of a matrix $M$ is a $d$-by-$d$ division of $M$ whose every zone has rank at least $k$, and \emph{rank-$k$ division} is a short-hand for \emph{rank-$k$ $k$-division}.
Then a matrix class $\mathcal M$ has bounded grid rank if there is an integer $k$ such that no matrix of $\mathcal M$ admits a rank-$k$ division.

Henceforth it will be more convenient to only work with $0,1$-matrices.
\cref{lem:finite-to-binary} allows us to do so.
It directly implies that for every matrix class $\mathcal M$ over a finite alphabet $A$, if $\mathcal M$  unbounded grid rank then there is some $a\in A$ such that the class of $0,1$-matrices $s_a(\mathcal M)$ 
obtained from $\mathcal M$ by replacing $a$ by $1$ and each other letter by $0$, has unbounded grid rank, too.
% We denote by $\ramm{k}{r}$ the smallest integer $t$ such that every clique $K_t$ edge-colored with $r$ colors admits a monochromatic clique of size $k$.
% Similarly we denote by $\bipramm{k}{r}$ the smallest integer $t$ such that every biclique $K_{t,t}$ edge-colored with $r$ colors admits a monochromatic biclique $K_{k,k}$.

% \begin{lemma}\label{lem:rank-selection}
%   Let $A$ be a finite alphabet and $k\in\Nn$.
%   If a matrix $M$ over the alphabet $A$ has more than $k^{|A|-1}$ distinct rows 
%   then there is $a\in A$ such that $s_a(M)$ has more than $k$ distinct rows. 
% \end{lemma}
% \begin{proof}
  
% \end{proof}

\begin{lemma}\label{lem:finite-to-binary}
  Let $A$ be a finite alphabet, $k,d\in\Nn$  and let $M$ be any $A$-matrix admitting a rank-$K$ $D$-division for $K=(k-1)^{|A|-1}+1$ and $D=\bipramm{d}{|A|}$.
  Then for some $a\in A$ the $0,1$-matrix $s_a(M)$ admits a rank-$k$ $d$-division. 
  In particular, if a class $\mathcal M$ of $A$-matrices has unbounded grid rank then there is some $a\in A$ such that $s_a(\mathcal M)$ has unbounded grid rank.
\end{lemma}

% \begin{lemma}\label{lem:finite-to-binary}
%   Let $A$ be a finite alphabet, $k$ be any natural, $r := {|A| \choose 2}$, and $K := \max(\bipramm{k}{r},\ramm{k}{r})$.
%   Let $M$ be an $A$-matrix admitting a rank-$K$ division.
%   Then there is a mapping $p: A \to \set{0,1}$ such that the 0,1-matrix $(p(M_{i,j}))_{i,j}$ admits a rank-$k$ division. 
% \end{lemma}
\begin{proof}
  First note that  if an $A$-matrix $N$ has at least $K$ distinct rows or columns
  then there is some $a\in A$ such that $s_a(N)$ has at least $k$ distinct rows or columns.
  Indeed, suppose that $s_a(N)$ has at most $k-1$ distinct rows, for each $a \in A$.
  Let $a_0 \in A$ be any letter.
  Each row $v$ of $N$ (treated as a matrix with one row) is uniquely determined by the tuple $(s_a(v))_{a\in A\setminus \set {a_0}}$ of all its $a$-selections with $a\neq a_0$.
  As $s_a(v)$ is a row of $s_a(N)$, there are $k-1$ possible values for $s_a(v)$, by assumption.
  Hence, $(s_a(v))_{a \in A \setminus \{a_0\}}$ ranges over a set of size at most $(k-1)^{|A|-1}=K-1$, which yields the conclusion.
  
  We now prove the statement in the lemma.
  Let $\mathcal D$ be a rank-$K$ division of $M$.
  By definition in every cell $C$ of $\mathcal D$ there are (at least) $K$ distinct row or $K$ distinct column vectors. By the above, there is some $a\in A$ such that $s_a(C)$ 
  has at least $k$ distinct row or $k$ distinct column vectors. 
  We label $C$ by $a$.

  There are only $|A|$ possible labels for the cells and $D = \bipramm{d}{|A|}$.
  Thus by \cref{th: bipRamsey}, there is a fixed letter, say, $a\in A$ and a $d$-division $\mathcal D'$ coarsening the $D$-division $\mathcal D$ such that each cell of $\mathcal D'$ contains a cell of $\mathcal D$ labeled by $a$.
  By construction, $\mathcal D'$ is a rank-$k$ $d$-division of $s_a(M)$.
\end{proof}

Let $\id$ be the $k \times k$ identity matrix, and $\ao$, $\az$, $\utr$, and $\ltr$ be the $k \times k$ 0,1-matrices that are full 1, full 0, full-1 upper triangular, and full-1 lower triangular, respectively.
More precisely, the $k \times k$ full-1 upper (resp.~lower) triangular matrix is a 0,1-matrix with a 1 at position $(i,j) \in [k]^2$ if and only if $i \leqslant j$ (resp.~$i \geqslant j$).
Let $A^\mathrm{M}$ be the vertical mirror of matrix $A$, that is, its reflection about a vertical line separating the matrix in two equal parts.\footnote{i.e., column $\lceil n/2 \rceil$ if the number $n$ of columns is odd, and between columns $n/2$ and $n/2+1$ if $n$ is even} 
The following Ramsey-like result states that every $0,1$-matrix with huge rank (or equivalently a huge number of distinct row or column vectors) admits a regular matrix with large rank. 

\begin{theorem}
\label{Trk}
 There is a function $T: \mathbb N^+ \to \mathbb N^+$ such that for every natural $k$, every matrix with at least $T(k)$ rows or at least $T(k)$ columns contains as a submatrix one of the following $k \times k$ matrices: $\id$, $\ao-\id$, $\utr$, $\ltr$, $\id^\mathrm{M}$, $(\ao-\id)^\mathrm{M}$, $\utr^\mathrm{M}$, $\ltr^\mathrm{M}$.  
\end{theorem}

The previous theorem is a folklore result.
For instance, it can be readily derived from Gravier et al.~\cite{Gravier04} or from \cite[Corollary 2.4.]{Ding96} combined with the Erd\H{o}s-Szekeres theorem. 

Let $\rami$ be the set of the eight matrices of~\cref{Trk}.
The first four matrices are called \emph{diagonal}, and the last four (those defined by vertical mirror) are called \emph{anti-diagonal}.
By~\cref{Trk}, if a matrix class $\mathcal M$ has unbounded grid rank, then one can find in $\mathcal M$ arbitrarily large divisions with a matrix of $\rami$ as submatrix in each zone of the division, for arbitrarily large $k$.
We want to acquire more control on the horizontal-vertical interactions between these submatrices of~$\rami$.
We will prove that in large rank divisions, one can find so-called \emph{rank Latin divisions}.

An \emph{embedded submatrix} $M'$ of a matrix $M$ is a submatrix of $M$ together with the implicit information of the position of $M'$ in $M$.
In particular, we will denote by $\row(M')$, respectively $\col(M')$ the rows of $M$, respectively columns of $M$, intersecting precisely at $M'$. 
The argument of $\row(\cdot)$ or $\col(\cdot)$ is implicitly cast in an embedded submatrix of $M$.
In particular, $\row(M)$ simply denotes the set of rows of~$M$.
A \emph{contiguous} (embedded) submatrix is defined by a \emph{zone}, that is, a set of consecutive rows and a set of consecutive columns.
The \emph{$(i,j)$-cell} of a $d$-division $\mathcal D$, for any $i, j \in [d]$, is the zone formed by the $i$-th row block and the $j$-th column block of $\mathcal D$.
We will often denote that zone by $\mathcal D_{i,j}$.
 
A~\emph{rank-$k$ Latin $d$-division} of a matrix $M$ is a $d$-division $\mathcal D$ of $M$ such that for every $i, j \in [d]$ there is a contiguous embedded submatrix $M_{i,j} \in \rami$ in the $(i,j)$-cell of $\mathcal D$ satisfying: 
\begin{compactitem}
 \item $\{\row(M_{i,j})\}_{i,j}$ partitions $\row(M)$, and $\{\col(M_{i,j})\}_{i,j}$ partitions $\col(M)$.
 \item $\row(M_{i,j}) \cap \col(M_{i',j'})$ equals $\ao$ or $\az$, whenever $(i,j) \neq (i',j')$.
\end{compactitem}
Note that since the submatrices $M_{i,j}$ are supposed contiguous, the partition is necessarily a \opar{0}, hence a division.
A~\emph{rank-$k$ pre-Latin $d$-division} is the same, except that the second item need not be satisfied. 

  \tikzexternaldisable 
  \begin{figure}[h!]
    \centering
    \begin{tikzpicture}[scale=.4]
      \def\e{0.05}
      \pgfmathsetseed{17} 
    %\draw[opacity=100] (0, 0) grid (18, 18);

    %Autres zones
    \foreach \i in {0,...,8}{
       \foreach \j in {0,...,8}{
        \pgfmathsetmacro{\ii}{2*\i}
        \pgfmathsetmacro{\ip}{\ii+2}
        \pgfmathsetmacro{\jj}{2*\j}
        \pgfmathsetmacro{\jp}{\jj+2}
        \pgfmathsetmacro{\col}{ifthenelse(\j==mod(\i,3)*3+floor(\i/3),"white",ifthenelse(rnd >.5, "white", "black"))}
        \fill[\col] (\ii,\jj) -- (\ii,\jp) -- (\ip,\jp) -- (\ip,\jj) -- cycle;}
    }
    
    %Zones permutation universelle
    \foreach \a/\b in {0/1,1/0, 2/6,3/6,3/7, 4/12,5/13, 6/2,7/3, 8/8,9/8,9/9,10/15,11/14,11/15, 12/4,12/5,13/5, 14/10,15/11, 16/16,16/17,17/16}{
      \fill (\a,\b) -- (\a,\b+1) -- (\a+1,\b+1) -- (\a+1,\b) -- cycle;
    }

    \draw[line width=3pt, scale=6, color=black!10!yellow, opacity=100] (0, 0) grid (3, 3);

     \foreach \k in {0,...,8}{
        \pgfmathsetmacro{\i}{2*\k}
        \pgfmathsetmacro{\j}{2*((mod(\k,3)*3+floor(\k/3))}
        \draw[line width=2pt, red] (\i+\e,\j+\e) -- (\i+\e,\j+2-\e) -- (\i+2-\e,\j+2-\e) -- (\i+2-\e,\j+\e) -- cycle;      
    }
\end{tikzpicture}
    \caption{A $18 \times 18$ $0,1$-matrix with a rank-$2$ Latin \mbox{$3$-division} (in yellow) where 1~entries are depicted in black, 0~entries, in white, and every $M_{i,j}$ is highlighted in red.}
    \label{fig: Latin}
  \end{figure}
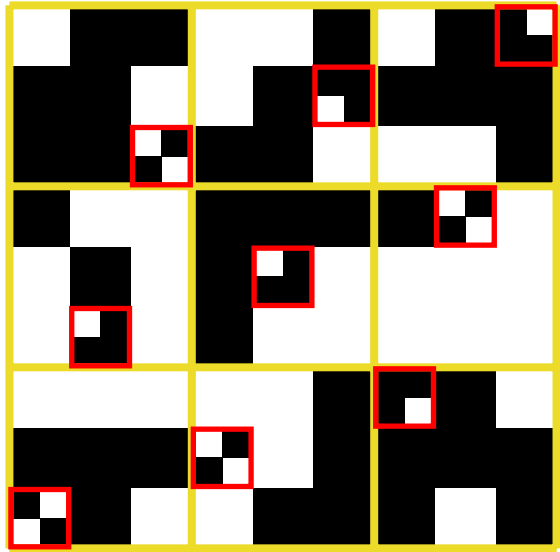
  \tikzexternalenable

We can now state our technical lemma.

\begin{lemma}
\label{lem: cstrams}
 For every positive integer $k$, there is an integer $K$ such that every $0,1$-matrix $M$ with a rank-$K$ division has a submatrix with a rank-$k$ Latin division. 
\end{lemma}

\begin{proof}
  We start by showing the following claim, a first step in the global cleaning process of \cref{lem: cstrams}.
  We recall that $T(\cdot)$ is the function of~\cref{Trk}.
\begin{claim}\label{Claimrk}
  Let $M$ be a $0,1$-matrix with a rank-$T(\kappa)$ $d^2$-division $\mathcal D$.
  There is a $\kappa d^2 \times \kappa d^2$ submatrix $\tilde M$ of $M$ with a rank-$\kappa$ pre-Latin $d$-division; i.e., a $d$-division $\mathcal D'$, coarsening $\mathcal D$, such that the $(i,j)$-cell of $\mathcal D'$ contains $M_{i,j} \in \mathcal N_\kappa$ as a contiguous submatrix, $\{\row(M_{i,j})\}_{i,j \in [d]}$ partitions $\row(\tilde M)$, and $\{\col(M_{i,j})\}_{i,j \in [d]}$, $\col(\tilde M)$.
\end{claim}
\begin{proof}[Proof of the claim]
  Let $\mathcal D^R$ be $(R_1, \ldots, R_{d^2})$ and, $\mathcal D^C$ be $(C_1,\ldots, C_{d^2})$.
  Let $\mathcal D'$ be the coarsening of $\mathcal D$ defined by
  $$\mathcal D'^R := (\bigcup_{i \in [d]} R_i, \bigcup_{i \in [d+1, 2d]} R_i, \ldots, \bigcup_{i \in [(d-1)d+1, d^2]} R_i),~\text{and}$$
  $$\mathcal D'^C := (\bigcup_{j \in [d]} C_j, \bigcup_{j \in [d+1, 2d]} C_j, \ldots, \bigcup_{j \in [(d-1)d+1, d^2]} C_j).$$ 
  By~\cref{Trk}, each cell of $\mathcal D$ contains a submatrix in $\mathcal N_\kappa$. 
  Thus there are $d^2$ such submatrices in each cell of $\mathcal D'$.
  For every $i, j \in [d]$, we keep in $\tilde M$ the $\kappa$ rows and $\kappa$ columns of a single submatrix of $\mathcal N_\kappa$ in the $(i,j)$-cell of $\mathcal D'$, and more precisely, one $M_{i,j}$ in the $(j+(i-1)d,i+(j-1)d)$-cell of $\mathcal D$.
  In other words, we keep in the $(i,j)$-cell of $\mathcal D'$, a submatrix of $\mathcal N_\kappa$ in the $(j,i)$-cell of $\mathcal D$ restricted to $\mathcal D'$.\footnote{Or for readers familiar with the game ultimate tic-tac-toe, at positions of moves forcing the next move in the symmetric cell about the diagonal.}
  The submatrices $M_{i,j}$ are contiguous in $\tilde M$.
  The set $\{\row(M_{i,j})\}_{i,j \in [d]}$ partitions $\row(\tilde M)$ since $j+(i-1)d$ describes $[d]^2$ when $i \times j$ describes $[d] \times [d]$.
  Similarly $\{\col(M_{i,j})\}_{i,j \in [d]}$ partitions $\col(\tilde M)$.
\end{proof}

We recall that $\bipramm{k}{2}$ is the minimum integer $b$ such that every $2$-edge coloring of the biclique $K_{b,b}$ contains a monochromatic $K_{k,k}$.
We set $\bipri{k}{1} := \bipramm{k}{2}$, and for every integer $s \geqslant 2$, we denote by $\bipri{k}{s}$, the minimum integer $b$ such that every $2$-edge coloring of $K_{b,b}$ contains a monochromatic $K_{q,q}$ with $q = \bipri{k}{s-1}$.
We set $\kappa := \bipri{k}{k^4-k^2}$ and $K := \max(T(\kappa),k^2) = T(\kappa)$, so that applying~\cref{Claimrk} on a rank-$K$ division (hence in particular a \mbox{rank-$T(\kappa)$} $k^2$-division) gives a rank-$\kappa$ pre-Latin $k$-division, with the $k^2$ submatrices of $\mathcal N_\kappa$ denoted by $M_{i,j}$ for $i, j \in [k]$.

At this point the zones $\row(M_{i,j}) \cap \col(M_{i',j'})$, with $(i,j) \neq (i',j')$, are arbitrary.
We now gradually extract a subset of $k$ rows and the $k$ \emph{corresponding columns} (i.e., the columns crossing at the diagonal if $M_{i,j}$ is diagonal, or at the anti-diagonal if $M_{i,j}$ is anti-diagonal) within each $M_{i,j}$, to turn the rank pre-Latin division into a rank Latin division.  
To keep our notation simple, we still denote by $M_{i,j}$ the initial submatrix $M_{i,j}$ after one or several extractions. 

For every (ordered) pair $(M_{i,j},M_{i',j'})$ with $(i,j) \neq (i',j')$, we perform the following extraction (in any order of these ${k^2 \choose 2}$ pairs).
Let $s$ be such that all the $M_{a,b}$'s have size $\bipri{k}{s}$.
We find two subsets of size $\bipri{k}{s-1}$, one in $\row(M_{i,j})$ and one in $\col(M_{i',j'})$, intersecting at a constant $\bipri{k}{s-1} \times \bipri{k}{s-1}$ submatrix.
In $M_{i,j}$ we keep only those rows and the corresponding columns, while in $M_{i',j'}$ we keep only those columns and the corresponding rows.
In every other $M_{a,b}$, we keep only the first $\bipri{k}{s-1}$ rows and corresponding columns.

After this extraction performed on the $k^4 - k^2$ zones $\row(M_{i,j}) \cap \col(M_{i',j'})$ (with $(i,j) \neq (i',j')$), we obtain the desired rank-$k$ Latin division (on a submatrix of $M$).
\end{proof}

A simple consequence of~\cref{lem: cstrams} is that every class $\mathcal M$ with unbounded grid rank satisfies $|\mathcal M_n| \geq \lfloor \frac n 2\rfloor!$.
Indeed there is a simple injection from \mbox{$n$-permutations} to $2n \times 2n$ submatrices of any rank-2 Latin $n$-division.  
This is enough to show that classes of unbounded grid rank are not small.
We will need some more work to establish the sharper lower bound of $n!$.

\section{Matrix classes with unbounded grid rank }\label{sec:more-ramsey}
In this section, we prove our main result concerning matrix classes,~\cref{thm:equiv}.
The plan is to refine the cleaning of rank Latin divisions, and prove the following.
\begin{theorem}
\label{thm:Ram6}
Let $\mathcal M$ be a matrix class over a finite alphabet with unbounded grid rank.\\
Then some $a$-selection $s_a(\mathcal M)$ of $\mathcal M$ includes $\mathcal F_\sym$ for some $\sym \in \six$.
\end{theorem}
We refer the reader to \cref{subsec:patterns} for the definition of $\mathcal F_\sym$
and $a$-selections.
\cref{thm:Ram6} will later simplify our task when we move to the growth of ordered graphs.
Moreover, it easily yields~\cref{thm:equiv}, as we will prove in~\cref{thm:m-nip-implies-pa} that each of the classes $\mathcal F_\sym$ is independent and intractable.
% While proving the theorem, we will improve the previous lower bound $|\mathcal M_n| \geq \lfloor \frac n 2\rfloor!$ to $|\mathcal M_n| \geq n!$ for matrix classes $\mathcal M$ of unbounded grid rank,
% and prove that such classes efficiently interpret the class of all graphs. This will prove~\cref{thm:equiv}. 
In addition, we show that the six classes in~\cref{thm:Ram6} constitute a minimal family, in the sense that none of the classes is contained in another.

\subsection{Finding $k!$ distinct $k \times k$ matrices when the grid rank is unbounded}\label{subsec: perm-mats}
In this section, we prove that matrix classes of unbounded grid rank have at least factorial growth. 
Apart from that, we prove the following, weaker variant of~\cref{thm:Ram6}.
  \begin{theorem}
  \label{cor: Ram16}
  Let $\mathcal M$ be a $0, 1$-matrix class with unbounded grid rank.
  Then there exists $\eta:\{-1,1\}\times \{-1,1\}\rightarrow\{0,1\}$ such that $\mathcal F_\eta \subseteq \mathcal M$.
  \end{theorem}
  The sixteen classes $\mathcal F_\eta$ are defined below,
  and include the classes $\mathcal F_\sym$ for $\sym\in\six$. In the following section, we will reduce those sixteen classes down to six.

\medskip
Recall that the \emph{order type} $\ot(x,y)$ of a pair $(x,y)$ of elements in a totally ordered set is equal to $-1$ if $x > y$, $0$ if $x=y$, and $1$ if $x < y$. 

\begin{definition}
\label{def: mats}
Let $k \geq 1$ be an integer and $\eta:\{-1,1\}\times \{-1,1\}\rightarrow\{0,1\}$.
For every $\sigma \in \mathfrak{S}_k$ we define the $k \times k$ matrix $F_{\eta}(\sigma) = (f_{i,j})_{1\leq i,j\leq k}$ by setting for every $i,j\in \left[ k \right]$,
  $$f_{i,j} := \left\{
    \begin{array}{ll}
        \eta(\ot(\sigma^{-1}(j),i),\ot(j,\sigma(i))) & \mbox{if } \sigma(i)\neq j \\
        1 - \eta(1,1) & \mbox{if } \sigma(i) = j. \\        
        \end{array}
    \right. $$
Finally $\mathcal{F}_{\eta}$ is the submatrix closure of $\sg{F_{\eta}(\sigma)~|~\sigma \in \mathfrak{S}_n, n\geq 1}.$
\end{definition}

These matrices generalize reorderings of matrices in $\mathcal{N}_k$.
For example, we find exactly the permutation matrices (reorderings of $\id$) when $\eta$ is constant equal to~$0$ and their complement when $\eta$ is constant equal to~$1$.
See~\cref{fig: encodage} for more interesting examples of such matrices.

The six classes $\mathcal F_\sym$'s with $\sym \in \six$ correspond to six of the sixteen possible encodings~$\eta$.
More specifically, $\mathcal F_\sym = \mathcal F_\eta$ where $\eta$ is defined as follows, depending on $\sym$:
\begin{compactitem}
\item[($=$)] $\eta(x,y)=0$ for every $(x,y) \neq (1,1)$;%, and $\sym$ is $=$;
\item[($\neq$)] $\eta(x,y)=1$ for every $(x,y) \neq (1,1)$;%, and $\sym$ is $\neq$;
\item[($\le_R$)] $\eta(1,1)=\eta(-1,1)=1$ and $\eta(-1,-1)=\eta(1,-1)=0$;%, and $\sym$ is $\le_C$;
\item[($\ge_R$)] $\eta(1,1)=\eta(-1,1)=0$ and $\eta(-1,-1)=\eta(1,-1)=1$;%, and $\sym$ is $\ge_C$;
\item[($\le_C$)] $\eta(-1,-1)=\eta(-1,1)=1$ and $\eta(1,1)=\eta(1,-1)=0$;%, and $\sym$ is $\ge_R$;
\item[($\ge_C$)] $\eta(-1,-1)=\eta(-1,1)=0$ and $\eta(1,1)=\eta(1,-1)=1$.%, and $\sym$ is $\le_R$;
\end{compactitem}
A careful reader might notice that the entries at positions $(i,\sigma(i))$ differ between the encodings of, say, $\mathcal F_{\le_R}$ and $\mathcal F_{\eta}$ with $\eta(1,1)=\eta(-1,1)=1$ and $\eta(-1,-1)=\eta(1,-1)=0$.
The latter class could rather be denoted by $\mathcal F_{<_R}$.
As the $\mathcal F_\sym$'s are closed under taking submatrices, we already observed that $\mathcal F_{\geqslant R} = \mathcal F_{> R}$ holds.

  \tikzexternaldisable 
  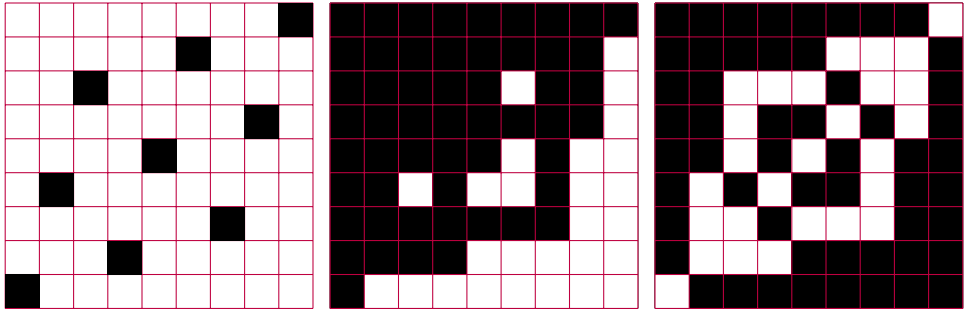
\begin{figure}[h!]
    \centering
    \begin{tikzpicture}[scale = .45]
\begin{scope}
    \draw[purple] (0, 0) grid (9, 9);
    \foreach \i in {0,...,8}{
        \pgfmathsetmacro{\ip}{\i+1}
        \pgfmathsetmacro{\j}{(mod(\i,3)*3+mod(floor(\i/3),3)}
        \pgfmathsetmacro{\jp}{\j+1}
        \fill[black] (\i,\j) -- (\i,\jp) -- (\ip,\jp) -- (\ip,\j) -- cycle;
    }
  \end{scope}

    \begin{scope}[xshift=9.5cm]
    \draw (0, 0) grid (9, 9);    
    
    \foreach \i in {0,...,8}{
        \foreach \j in {0,...,8}{
        \pgfmathsetmacro{\ip}{\i+1}
        \pgfmathsetmacro{\sigi}{(mod(\i,3)*3+mod(floor(\i/3),3)}
        \pgfmathsetmacro{\jp}{\j+1}
        \pgfmathsetmacro{\sigj}{3*mod(\j,3)+mod(floor(\j/3),3)}
        \pgfmathsetmacro{\col}{ifthenelse(\i > \sigj,ifthenelse(\j<\sigi,"white","black"),"black")}
        \fill[\col] (\i,\j) -- (\i,\jp) -- (\ip,\jp) -- (\ip,\j) -- cycle;
        }

    }
    \draw[purple] (0, 0) grid (9, 9);            
    \end{scope}

    \begin{scope}[xshift=19cm]
    \draw (0, 0) grid (9, 9);    
    \foreach \i in {0,...,8}{
        \foreach \j in {0,...,8}{
        \pgfmathsetmacro{\ip}{\i+1}
        \pgfmathsetmacro{\sigi}{(mod(\i,3)*3+mod(floor(\i/3),3)}
        \pgfmathsetmacro{\jp}{\j+1}
        \pgfmathsetmacro{\sigj}{3*mod(\j,3)+mod(floor(\j/3),3)}
        \pgfmathsetmacro{\col}{ifthenelse(\i < \sigj,ifthenelse(\j>\sigi,"black","white"),ifthenelse(\j<\sigi,"black","white"))}
        \fill[\col] (\i,\j) -- (\i,\jp) -- (\ip,\jp) -- (\ip,\j) -- cycle;
        }
    }
    \draw[purple] (0, 0) grid (9, 9);        
    \end{scope}    
    
    \end{tikzpicture}
    \caption{Left: $9 \times 9$ permutation matrix $M_{\sigma}$.
      Center: The matrix $F_{\eta_1}(\sigma)$ with $\eta_1 (1,1)=0$ and $\eta_1 (-1,-1) = \eta_1 (-1,1) = \eta_1(1,-1) =1$.
      Right: The matrix $F_{\eta_2}(\sigma)$ with $\eta_2 (1,1) = \eta_2 (-1,-1) =1$ and $\eta_2 (-1,1) = \eta_2 (1,-1) =0$.
    }
    \label{fig: encodage}
  \end{figure}
  \tikzexternalenable

With the next lemma, we get even cleaner universal patterns out of a large rank Latin division. We use the notation of~\cref{lem:grid-ramsey}.

\begin{lemma}
\label{lem: factl3}
Let $k \geq 1$ be an integer.
Let $M$ be a $0,1$-matrix with a rank-$k$ Latin $N$-division with $N := \gridramm{k}{2}$.
Then there exists $\eta:\{-1,1\}\times \{-1,1\} \rightarrow\{0,1\}$ such that the submatrix closure of $M$ contains the set $\sg{F_{\eta}(\sigma)~|~ \sigma \in {\mathfrak{S}}_k}$.
\end{lemma}

\begin{proof}
  Let $(\mathcal R, \mathcal C)$ be the rank-$k$ Latin $N$-division, with $\mathcal R := \sg{R_1, \ldots, R_N}$ and $\mathcal C := \sg{C_1, \ldots, C_N}$, so that every row of $R_i$ (resp. column of $C_i$) is smaller than every row of $R_j$ (resp. column of~$C_j$) whenever $i < j$.
  Let $M_{i,j}$ be the \emph{chosen} contiguous submatrix of $\rami$ in $R_i \cap C_j$ for every $i, j \in [N]$.
  Recall that, by definition of a rank Latin division, $\{\row(M_{i,j})\}_{i,j \in [N]}$ partitions $\row(M)$ (resp.~$\{\col(M_{i,j})\}_{i,j \in [N]}$ partitions $\col(M)$) into intervals. 
  
  Let $\str N := ([N]^2, <_1, <_2)$, where $<_1$ is the lexicographic order on $[N]^2$ which first orders according to the first coordinate and then the second one, while $<_2$ is the lexicographic order on $[N]^2$ which first orders according to the second coordinate and then the second one.  
  
  We now define a coloring $c : [N]^4 \to \sg{0,1}$ as follows: for every $(i,j) \neq (i',j') \in \left[ N \right]^2$, we let $c((i,j),(i',j')) \in \sg{0,1}$ be the value of the constant entries in $\row(M_{i,j}) \cap \col(M_{i',j'})$. We choose arbitrary colors when $(i,j)=(i',j')$. By~\cref{lem:grid-ramsey}, there are two sets $R,C \in \binom{\left[ N \right]}{k}$ such that for every $i,i',j,j'\in [R]$, the value $c((i,j),(i',j'))$ only depends on $\ot_1((i,j),(i',j'))$ and $\ot_2((i,j),(i',j'))$. In particular when $i\neq i'$ and $j\neq j'$, this value only depends on $\ot(i,i')$ and $\ot(j,j')$.

  Let $\eta:\{-1,1\}\times \{-1,1\}\rightarrow\{0,1\}$ be such that for every $i\neq i'\in [N]$ and $j\neq j'\in [N]$ we have: 
  $$c((i,j),(i',j'))=\eta(\ot(j,j'), \ot(i,i')). $$
  
  In terms of the rank Latin division, this means that for every $i < i' \in R$ and $j < j' \in C$,
  \begin{itemize}
  \item $\col(M_{i,j}) \cap \row(M_{i',j'})$ has constant value $\eta(-1,-1)$,
  \item $\row(M_{i,j}) \cap \col(M_{i',j'})$ has constant value $\eta(1,1)$,
  \item $\col(M_{i',j}) \cap \row(M_{i,j'})$ has constant value $\eta(-1,1)$, and
  \item $\row(M_{i',j}) \cap \col(M_{i,j'})$ has constant value $\eta(1,-1)$.
  \end{itemize}
  
  \tikzexternaldisable 
  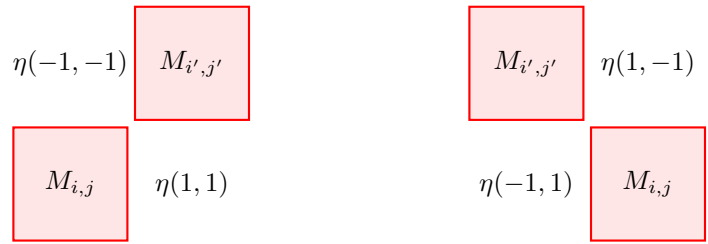
\begin{figure}[h!]
    \centering
    \begin{tikzpicture}
    \def\z{1.5}
    \foreach \i/\j in {0/0,1.6/1.6, 6/1.6,7.6/0}{
      \draw[thick, red] (\i,\j) -- (\i,\j+\z) -- (\i+\z,\j+\z) -- (\i+\z,\j) -- cycle ;
      \draw[fill, red, opacity=0.1] (\i,\j) -- (\i,\j+\z) -- (\i+\z,\j+\z) -- (\i+\z,\j) -- cycle ;
    }
    \node at (0.75,0.75) {$M_{i,j}$} ;
    \node at (8.35,0.75) {$M_{i,j}$} ;
    \node at (2.35,2.35) {$M_{i',j'}$} ;
    \node at (6.75,2.35) {$M_{i',j'}$} ;

    \node at (0.75,2.35) {$\eta(-1,-1)$} ;
    \node at (2.35,0.75) {$\eta(1,1)$} ;
    \node at (6.75,0.75) {$\eta(-1,1)$} ;
    \node at (8.35,2.35) {$\eta(1,-1)$} ;
    \end{tikzpicture}
    \caption{How zones are determined by $\eta$, $\ot(i,i')$, and $\ot(j,j')$.}
    \label{fig:eta}
  \end{figure}
  \tikzexternalenable
  
  In other words, $\row(M_{i,j}) \cap \col(M_{i',j'})$ is entirely determined by $\eta$, $\ot(i,i')$, and $\ot(j,j')$ (see~\cref{fig:eta}).
  
  Let $\sigma \in \mathfrak{S}_k$.
  We now show how to find $F_{\eta}(\sigma)=(f_{i,j})_{1\leq i,j \leq k}$ as a submatrix of $M$. 
  For every $i \in \left[ k \right]$, we choose a row $r_i \in \row(M_{i, \sigma(i)})$ and a column $c_{\sigma(i)} \in \col(M_{i, \sigma(i)})$ such that the entry of $M$ at the intersection of $r_i$ and $c_{\sigma(i)}$ has value $f_{i, \sigma(i)}$.
  This is possible since the submatrices $M_{i,j}$ are in ${\mathcal N}_k$ and have disjoint row and column supports.
  We consider the $k \times k$ submatrix $M'$ of $M$ with rows $\sg{r_i~|~i \in \left[k\right]}$ and columns $\sg{c_i~|~i \in \left[ k \right]}$.

  By design $M' = F_{\eta}(\sigma)$ holds.
  Let us write $M' := (m_{i,j})_{1\leq i,j \leq k}$ and show for example that if $\ot(\sigma^{-1}(j),i)=-1~ \mathrm{and} ~\ot(j,\sigma(i))=1$ for some $i,j\in \left[ k \right]$, then we have $m_{i,j} = \eta(-1,1) = f_{i,j}$.
  The other cases are obtained in a similar way.
  Let $i' := \sigma^{-1}(j) > i$ and $j' := \sigma(i)>j$.
  In $M'$, $m_{i,j}$ is obtained by taking the entry of $M$ associated to the row $r_i$ of the matrix $M_{i,\sigma(i)} = M_{i,j'}$ and the column $c_j$ of $M_{\sigma^{-1}(j), j}= M_{i', j}$.
  The entry $m_{i,j}$ lied in $M$ in the zone $\row(M_{i,j'}) \cap \col(M_{i',j})$ with constant value $\eta(-1,1)$. 
\end{proof}

We now check that $\sigma \in {\mathfrak{S}}_k \mapsto F_{\eta}(\sigma)$ is indeed injective.

\begin{lemma}
\label{lem: factl4}
For every $k\geq1$ and $\eta:\{-1,1\}\times \{-1,1\}\rightarrow\{0,1\}$\emph{:} 
$$\left|\sg{F_{\eta}(\sigma)~|~\sigma \in {\mathfrak{S}}_k }\right| = k!$$
\end{lemma}

\begin{proof}
We let $k\geq1$ and $\eta:\{-1,1\}\times \{-1,1\}\rightarrow\{0,1\}$. 
The inequality $\left|\sg{F_{\eta}(\sigma)~|~\sigma \in {\mathfrak{S}}_k }\right| \leq k!$ simply holds.
We thus focus on the converse inequality.

When we read out the first row (bottom one) of $F_{\eta}(\sigma) = (f_{i,j})_{1\leq i,j\leq k}$ by increasing column indices (left to right), we get a possibly empty list of values $\eta(-1,1)$, one occurrence of $1-\eta(1,1)$ at position $(1,\sigma(1))$, and a possibly empty list of values $\eta(1,1)$.
The last index $j$ such that $f_{1,j} \neq f_{1,j+1}$, or $j = k$ if no such index exists, thus corresponds to $\sigma(1)$.
We remove the first row and the $j$-th column and iterate the process on the rest of the matrix.
\end{proof}

By piecing \cref{lem: cstrams,lem: factl3,lem: factl4} together, we get:

\begin{theorem}\label{thm:fact-implies-bgr}
 Every $0,1$-matrix class $\mathcal M$ with unbounded grid rank satisfies $|\mathcal M_k| \geqslant k!$, for every integer~$k$. 
\end{theorem}
\begin{proof}
  %Let $\mathcal M$ be a class of matrices with unbounded grid rank.
  We fix $$k \geq 1,~n :={\rm R}_{16}(k),~N := {\rm R}_{16}^{(\binom{n}{2}+1)}(k).$$
  Now we let $K := K(N)$ be the integer of~\cref{lem: cstrams} sufficient to get a \mbox{rank-$N$} Latin division.
  As $\mathcal M$ has unbounded grid rank, it contains a matrix~$M$ with grid rank at least~$K$.
  By~\cref{lem: cstrams}, a submatrix $\tilde{M} \in \mathcal M$ of $M$ admits a rank-$N$ Latin division, from which we can extract a rank-$k$ Latin \mbox{$N$-division} (since $k \leqslant N$).
  By \cref{lem: factl3} applied to $\tilde{M}$, there exists $\eta$ such that $\sg{F_\eta(\sigma)~|~\sigma \in \mathfrak{S}_k} \subseteq \mathcal M_k$.
  By \cref{lem: factl4}, this implies that $|\mathcal M_k| \geqslant k!$. 
\end{proof}

\begin{proof}[Proof of~\cref{cor: Ram16}]
  We just showed that for every matrix class $\mathcal M$ of unbounded grid rank, for every integer $k$, there is an $\eta(k):\{-1,1\}\times \{-1,1\}\rightarrow\{0,1\}$ such that $\sg{F_{\eta(k)}(\sigma)~|~\sigma \in \mathfrak{S}_k} \subseteq \mathcal M_k \subseteq \mathcal M$.
As there are only 16 possible \emph{functions~$\eta$}, the sequence $\eta(1),\eta(2),\ldots$ contains at least one function $\eta$ infinitely often.
Besides for every $k' < k$, $\sg{F_{\eta}(\sigma)~|~\sigma \in \mathfrak{S}_{k'}}$ is included in the submatrix closure of $\sg{F_{\eta}(\sigma)~|~\sigma \in \mathfrak{S}_k}$. This proves $\mathcal M\supseteq \mathcal F_\eta$.
\end{proof}
\subsection{Minimal family of six unavoidable classes}\label{sec:six}

\cref{cor: Ram16} shows that each matrix class with unbounded twin-width contains one of the sixteen classes $\mathcal F_\eta$.
We will now see that some of these classes are contained in some others.
We say that a mapping $\eta\from\set{-1,1}\times\set{-1,1}\to\set{0,1}$ \emph{depends only on one coordinate} if either $\eta(x,y)=\eta(x',y)$ for all $x,x',y\in\set{-1,1}$, or $\eta(x,y)=\eta(x,y')$ for all $x,y,y'\in\set{-1,1}$.
Note that there are only six such functions~$\eta$.
Indeed, once we fix the coordinate (first or second) $\eta$ depends on, there are four mappings from $\set{-1,1}$ to $\set{0,1}$.
This adds up to eight mappings but the constant-0 and constant-1 mappings are each counted twice, hence a total of six functions.

These six mappings $\eta$ correspond to the six classes $\mathcal F_=$, $\mathcal F_{\neq}$, $\mathcal F_{\le R}$, $\mathcal F_{\ge R}$, $\mathcal F_{\le C}$, and $\mathcal F_{\ge C}$ defined in~\cref{subsec:patterns}, where we recall that the matrix encoding a given permutation $\sigma$ respectively follows the Iverson brackets
\[[\sigma(i)=j],\quad [\sigma(i)\neq j],\quad [i \le \sigma^{-1}(j)],\quad [i\ge \sigma^{-1}(j)],\quad [j \le \sigma(i)],~\text{and}~\quad[j \ge \sigma(i)]
\]
for its entry at position $(i,j)$.
Thus to establish the refined milestone of the section, \cref{thm:Ram6}, we shall prove the following.

\begin{lemma}\label{lem:reduce eta's}
  Let $\eta\from \set{-1,1}\times \set{-1,1}\to \set{0,1}$.
Then there is a $\gamma\from\set{-1,1}\times\set{-1,1}\to\set{0,1}$ depending only on one coordinate, such that $\mathcal F_\eta\supseteq\mathcal F_{\gamma}$.
\end{lemma}

\begin{proof}
If $\eta$ depends only on one coordinate, we are done.
In particular, we can suppose that $\eta$ is not constant.
Then there are $x, y, x', y' \in \set{-1,1}$ with $x=x'$ or $y=y'$ such that $\eta(x,y) \neq \eta(x',y')$. 

We will assume that $\eta(1,1) \neq \eta(-1,1)$.
The three other cases ($\eta(1,1) \neq \eta(1,-1)$, $\eta(-1,-1) \neq \eta(1,-1)$, and $\eta(-1,-1) \neq \eta(-1,1)$) are similar and correspond to rotating~\cref{fig:16-to-6} by a 90, 180, 270-degree angle, respectively.
We choose $\gamma\from\set{-1,1}\times\set{-1,1}\to\set{0,1}$ by setting $\gamma(x,y)=\eta(x,1)$ for $x,y\in\set{-1,1}$.
By construction $\gamma$ depends only on the first coordinate.
We are left with proving that $\mathcal F_\eta \supseteq \mathcal F_\gamma$.

   Let $\sigma$ be any $k$-permutation.
   We build a $2k$-permutation $\tau$ as a perfect shuffle of $\sigma$ and the identity $k$-permutation.
   More precisely, $\tau$ has its 1 entries at positions $(i,2\sigma(i)-1)$ and $(k+i,2i)$, for every $i \in [k]$. 
   See~\cref{fig:16-to-6} for an illustration on a particular $5$-permutation.
   
   \tikzexternaldisable 
   \begin{figure}[h!]
     \centering
     \begin{tikzpicture}[scale=0.5]
       \begin{scope}
    \foreach \i/\j in {1/3,2/5,3/1,4/2,5/4}{
        \pgfmathsetmacro{\ip}{\i+1}
        \pgfmathsetmacro{\jp}{\j+1}
        \fill[black] (\i,\j) -- (\i,\jp) -- (\ip,\jp) -- (\ip,\j) -- cycle;
    }
    \draw[purple] (1, 1) grid (6, 6);
       \end{scope}

        \begin{scope}[xshift=7cm]
    \foreach \iq/\j in {1/3,2/5,3/1,4/2,5/4}{
      \pgfmathsetmacro{\i}{2 * \iq - 1}
      \pgfmathsetmacro{\ip}{\i+1}
      \pgfmathsetmacro{\jp}{\j+1}
        \fill[black] (\i,\j) -- (\i,\jp) -- (\ip,\jp) -- (\ip,\j) -- cycle;
    }
    \foreach \i/\j in {2/6,4/7,6/8,8/9,10/10}{
        \pgfmathsetmacro{\ip}{\i+1}
        \pgfmathsetmacro{\jp}{\j+1}
        \fill[blue] (\i,\j) -- (\i,\jp) -- (\ip,\jp) -- (\ip,\j) -- cycle;
    }
    \draw[purple] (1, 1) grid (11, 11);
        \end{scope}

        \begin{scope}[xshift=19cm]
    \foreach \iq/\j in {1/3,2/5,3/1,4/2,5/4}{
      \pgfmathsetmacro{\i}{2 * \iq - 1}
      \pgfmathsetmacro{\ip}{\i+1}
      \pgfmathsetmacro{\jp}{\j+1}
        \fill[black] (\i,\j) -- (\i,\jp) -- (\ip,\jp) -- (\ip,\j) -- cycle;
    }
    \foreach \i/\j in {2/6,4/7,6/8,8/9,10/10}{
        \pgfmathsetmacro{\ip}{\i+1}
        \pgfmathsetmacro{\jp}{\j+1}
        \fill[black] (\i,\j) -- (\i,\jp) -- (\ip,\jp) -- (\ip,\j) -- cycle;
    }
     \draw[purple] (1, 1) grid (11, 11);
    \foreach \i/\j in {2/3,4/5,6/7,8/9,10/11}{
      \draw[very thick, blue, fill opacity=0.2, fill=blue] (\i,1) -- (\j,1) -- (\j,6) -- (\i,6) -- cycle ;
    }
       \end{scope}
     \end{tikzpicture}
     \caption{Left: An example of a $k$-permutation $\sigma$ with $k=5$. Center: The matrix of the corresponding $2k$-permutation $\tau$, where the ``initial'' 1 entries are still in black, and those coming from the identity are in blue. Right: The submatrix where $F_\eta(\tau)$ is only populated by the (distinct) values $\eta(1,1)$ and $\eta(-1,1)$.}
     \label{fig:16-to-6}
   \end{figure}
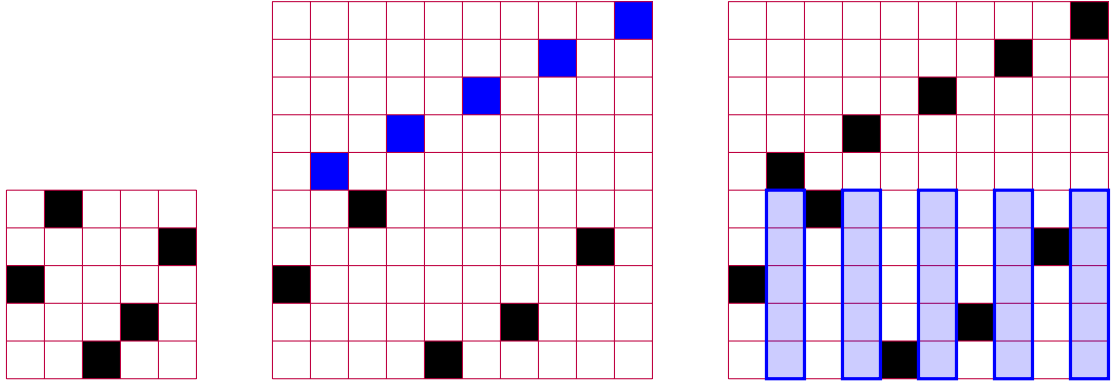
   \tikzexternalenable

We claim that $F_\gamma(\sigma)$ appears as the submatrix $N$ of $F_\eta(\tau)$ obtained by keeping the even-indexed columns and the first $k$ rows (see shaded blue area in the right matrix of~\cref{fig:16-to-6}).    
Indeed as the 1 entry of $\tau$ in each kept column is above the kept rows, $N$ depends only on $\eta(1,1)$ and $\eta(-1,1)$ (recall~\cref{fig:eta}).
Specifically $N_{i,j} = \eta(1,1)$ if $j < \sigma(i)$ and $N_{i,j} = \eta(-1,1)$ otherwise.
As $\gamma(x,y) = \eta(x,1)$, the encoding~$\gamma$ follows~\cref{fig:eta} where $\eta(1,-1)$ is replaced by $\eta(1,1)$, and $\eta(-1,-1)$ is replaced by $\eta(-1,1)$.
Thus it also holds that $F_\gamma(\sigma)_{i,j} = \eta(1,1)$ if $j < \sigma(i)$ and $\eta(-1,1)$ otherwise.
Hence $N = F_\gamma(\sigma)$.
This proves that $F_\gamma(\sigma) \in \mathcal F_\eta$ since $\mathcal F_\eta$ is closed under submatrices.
And we conclude that  $\mathcal F_\gamma \subseteq \mathcal F_\eta$.
\end{proof}

This proves \cref{thm:Ram6}.
\begin{proof}[Proof of~\cref{thm:Ram6}]
  Follows from~\cref{lem:finite-to-binary,cor: Ram16,lem:reduce eta's}.
\end{proof}

We end this section showing that the set of six encoding functions $\gamma$'s depending only on one coordinate is minimal, in the sense of the following lemma. 

\begin{lemma}
 \label{lem:6minimality}
 Let $\gamma, \gamma':\set{-1,1}\times \set{-1,1}\to \set{0,1}$ be two distinct functions depending only on one coordinate. Then $\mathcal F_\gamma$ and $\mathcal F_{\gamma'}$ are incomparable for $\subseteq$, i.e., neither $\mathcal F_\gamma \subseteq \mathcal F_{\gamma'}$ nor $\mathcal F_{\gamma'}\subseteq \mathcal F_\gamma$ hold.
\end{lemma}
\begin{proof}
  We consider the permutation product of two transpositions $\tau:=(135)(24)$ over $[5]$.
  Let $\gamma:\set{-1,1}\times \set{-1,1}\to \set{0,1}$ be an encoding function that only depends on one coordinate and $N:=F_{\gamma}(\tau)$. 
 
 Up to a symmetric argument, we assume that $\gamma(x,y)=\gamma(x,y')$ for every $x,y,y'\in \set{0,1}$, that is, $\gamma$~depends on the first coordinate. 
 Let $\eta:\set{-1,1}\times \set{-1,1}\to \set{0,1}$ only depending on one coordinate be such that $N\in \mathcal{F}_{\eta}$.
 We will show that $\gamma = \eta$, which implies the desired result. 
 
 As $N\in \mathcal{F}_{\eta}$, there exists some permutation $\sigma$ such that $M:=F_{\eta}(\sigma)$ contains $N$ as a submatrix. 
 
 \medskip\noindent
 \mbox{\textbf{Case 1.}}
 \medskip\noindent
  $\eta$ depends only on the first coordinate, i.e., $\eta(x,y)=\eta(x,y')$ for every $x,y,y'\in \set{0,1}$. 
  
  In this case, observe that every row of $M$ consists of a sequence of consecutive entries with value $\eta(-1,-1)= \eta(-1, 1)$, then an entry with value $1-\eta(1,1)$, and then a sequence of consecutive entries with value $\eta(1,1)= \eta(1, -1)$. Moreover the first row (the bottommost in~\cref{fig:6min}) of $N$ has exactly the values $\gamma(-1,-1), \gamma(-1,-1), 1-\gamma(1,1), \gamma(1,1), \gamma(1,1)$, thus we must have $\eta(-1,-1)=\gamma(-1,-1)$ and $\eta(1,1)=\gamma(1,1)$. Thus we conclude that $\eta = \gamma$.

\tikzexternaldisable 
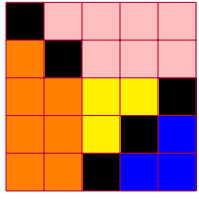
\begin{figure}[h!]
     \centering
    \begin{tikzpicture}[scale=0.5]
       \begin{scope}

    \foreach \i/\j in {1/5,2/4,3/1,4/2,5/3}{
        \pgfmathsetmacro{\ip}{\i+1}
        \pgfmathsetmacro{\jp}{\j+1}
        \fill[black] (\i,\j) -- (\i,\jp) -- (\ip,\jp) -- (\ip,\j) -- cycle;
    }
    
    \foreach \i/\j in {2/5,3/5,3/4,4/5,4/4,5/5,5/4}{
        \pgfmathsetmacro{\ip}{\i+1}
        \pgfmathsetmacro{\jp}{\j+1}
        \fill[pink] (\i,\j) -- (\i,\jp) -- (\ip,\jp) -- (\ip,\j) -- cycle;
    }
   \foreach \i/\j in {1/1,1/2,1/3,1/4,2/1,2/2,2/3}{
        \pgfmathsetmacro{\ip}{\i+1}
        \pgfmathsetmacro{\jp}{\j+1}
        \fill[orange] (\i,\j) -- (\i,\jp) -- (\ip,\jp) -- (\ip,\j) -- cycle;
    }    
   \foreach \i/\j in {3/2,3/3,4/3}{
        \pgfmathsetmacro{\ip}{\i+1}
        \pgfmathsetmacro{\jp}{\j+1}
        \fill[yellow] (\i,\j) -- (\i,\jp) -- (\ip,\jp) -- (\ip,\j) -- cycle;
    }    
   \foreach \i/\j in {4/1,5/1,5/2}{
        \pgfmathsetmacro{\ip}{\i+1}
        \pgfmathsetmacro{\jp}{\j+1}
        \fill[blue] (\i,\j) -- (\i,\jp) -- (\ip,\jp) -- (\ip,\j) -- cycle;
    }    
    \draw[purple] (1, 1) grid (6, 6);
       \end{scope}
    \end{tikzpicture}
    \caption{The matrix $N$ in the proof of~\cref{lem:6minimality}.
      The entries with value $\eta(1,1), \eta(-1,-1), \eta(1,-1),$ $\eta(-1,1), 1-\eta(1,1)$ are respectively associated to the colors blue, yellow, pink, orange, and black.}
     \label{fig:6min}
\end{figure}
\tikzexternalenable   
  
 \medskip\noindent
 \mbox{\textbf{Case 2.}}
 \medskip\noindent
   $\eta$ depends only on the second coordinate, i.e., $\eta(x,y)=\eta(x',y)$ for every $x,x',y\in \set{0,1}$. 
    
 In that case, observe that every column of $M$ consists of a sequence of consecutive entries with value $\eta(1,1)= \eta(-1, 1)$, then an entry with value $1-\eta(1,1)$, and then a sequence of consecutive entries with value $\eta(-1,-1)= \eta(1, -1)$. Moreover the last  column of $N$ has exactly the values $\gamma(1,1), \gamma(1,1), 1-\gamma(1,1), \gamma(-1,-1), \gamma(-1,-1)$, thus we must have $\eta(-1,-1)=\gamma(-1,-1)$ and $\eta(1,1)=\gamma(1,1)$. Now observe that the fourth column of $N$ has exactly the values $\gamma(1,1), 1-\gamma(1,1), \gamma(-1,-1) ,\gamma(1,-1), \gamma(1,-1)$. As $\gamma(1,1)=\eta(1,1)$, we must have $\gamma(1,-1)=\gamma(-1,-1)$.
 Thus $\gamma$ is constant, and so is $\eta$. In particular we have $\gamma = \eta$ and we are done. 
\end{proof}

\begin{corollary}\label{cor:stanley-wilf1}
  The classes $\mathcal F_\sym$ for $\sym\in\six$ are precisely all the matrix classes $\mathcal M$
  with growth $2^{\omega(n)}$ such that every proper subclass of $\mathcal M$ has growth $2^{O(n)}$. 
\end{corollary}
\begin{proof}
  Each of the classes $\mathcal F_\sym$ has growth $n!=2^{\omega(n)}$.
  Suppose $\mathcal M$ is a proper subclass of $\mathcal F_\sym$ with growth $2^{\omega(n)}$. Then $\mathcal M$ has unbounded twin-width, and hence contains one of the classes $\mathcal F_{\sym'}$ for some $\sym'\in\six$. But then $\mathcal F_{\sym'}$ is a proper subclass of $\mathcal F_\sym$, contradicting~\cref{lem:6minimality}. 

  On the other hand, if $\mathcal M$ is some class with growth $2^{\omega(n)}$, then $\mathcal M$ contains some class $\mathcal F_{\sym}$, which also has growth $2^{\omega(n)}$.
\end{proof}

\subsection{Matrix classes of unbounded twin-width are independent}\label{sec:NIP}
The goal of this section is to prove the following.

\begin{theorem}\label{thm:m-nip-implies-pa}
	Let  $\mathcal M$ be a (hereditary) class of matrices over a finite alphabet. If $\mathcal M$ has unbounded twin-width then $\mathcal M$ efficiently interprets the class of all graphs. In particular, $\mathcal M$ is independent and \FO~model checking is $\AW[*]$-hard on $\mathcal M$.
\end{theorem}
Recall (see~\cref{lem:matchings are dependent}) that the class $\mathscr M$ of all ordered matchings efficiently interprets the class of all graphs.
To prove~\cref{thm:m-nip-implies-pa} it remains to show that each of the classes $\mathcal F_\sym$ efficiently interprets $\mathscr M$.

\begin{lemma}\label{lem:independent-matrices}
  For each $\sym\in\six$, 
  the class $\mathcal F_\sym$ efficiently interprets the class $\mathscr M$ of all ordered matchings.
\end{lemma}
\begin{proof}Fix $\sym\in\six$.
  We construct an interpretation $\interp I$ such that 
  $\interp I(\mathcal F_\sym)$ contains all the ordered matchings of $\mathscr M$.
  Given an ordered matching $H$ with vertices $u_1<\ldots<u_n<v_1<\ldots<v_n$,
  let $\sigma\in \frak S_n$ be the permutation such that $u_i$ is adjacent to $v_{\sigma(i)}$, for $1\le i\le n$.
  Let $M$ be the $0,1$-matrix $F_\sym(\sigma)$.
  Clearly, $M$ can be constructed from $H$ in polynomial time. We now describe an interpretation $\mathsf I$ 
  such that $\mathsf I(M)$ is isomorphic to $H$. The interpretation $\mathsf I$ will not depend on $H$, but only on $\sym$.

  Recall that the matrix $M$ is viewed as a structure with two unary predicates $R,C$ indicating the rows and columns, respectively, a total order $<$ on $R\cup C$,
  and a binary relation $E\subseteq  R\times C$ defining the non-zero entries in $M$.
  
  If $\sym$ is `$=$' then the interpretation $\mathsf I$ simply forgets the unary predicates $R$ and $C$.
  If $\sym$ is `$\neq$' then the interpretation $\mathsf I$ replaces edges with non-edges.
  In either case, $\mathsf I(M)$ is isomorphic to $H$.

  Suppose $\sym$ is `$\leqslant_C$', the other cases being symmetric.
  Recall that in the matrix $M=F_\sym(\sigma)$ the entries of $M$ are defined by the Iverson bracket $[j \leqslant \sigma(i)]$.
  Hence, given $i$, $\sigma(i)$ is the largest value $j$ such that $M$ has entry $1$ at position $(i,j)$.
This can be defined by a first-order formula:
\[\phi(x,y):= R(x)\land C(y)\land E(x,y)\land \forall z.(C(z)\land y<z)\rightarrow \neg E(x,z).\]
The interpretation $\mathsf I$, given  $M=F_\sym(\sigma)$ with relations $R,C,<$, and $E$, outputs the matching with the same domain $R\cup C$, order $<$, and edge relation defined by the formula $\phi(x,y)\lor\phi(y,x)$.
Then $\mathsf I(M)$ is isomorphic to the ordered matching $H$.
\end{proof}
\cref{thm:m-nip-implies-pa} now follows:
\begin{proof}[Proof of~\cref{thm:m-nip-implies-pa}]
By~\cref{lem:matchings are dependent,lem:independent-matrices} and transitivity,
each of the classes $\mathcal F_\sym$ efficiently interprets the class of all graphs. By~\cref{thm:Ram6}, every matrix class $\mathcal M$ with unbounded twin-width contains one of the classes $\mathcal F_\sym$. Hence $\mathcal M$ efficiently interprets the class of all graphs. It follows that $\mathcal M$ is independent and \FO~model checking is $\AW[*]$-hard on $\mathcal M$ (see~\cref{cor:AW-hard}).
\end{proof}

\subsection{Proof of~\cref{thm:equiv}}
We can now conclude with the proof of \cref{thm:equiv}, which we restate below for arbitrary finite alphabets, with all the conditions negated to ease the reasoning.

\begin{reptheorem}{thm:equiv}
  Given a class $\mathcal M$ of matrices over a finite alphabet $A$, the following are equivalent.
  \begin{enumerate}[label=($\neg${\roman*}),ref=$\neg$\roman*]
  \item \label{nit:bd-tww} $\mathcal M$ has unbounded twin-width.
  \item \label{nit:bd-gr} $\mathcal M$ has unbounded grid rank.
  \item \label{nit:ramsey} $\mathcal M$ is not pattern-avoiding.
  \item \label{nit:nip} $\mathcal M$  interprets the class of all graphs.
  \item \label{nit:m-nip} $\mathcal M$  transduces the class of all graphs.
  \item \label{nit:fact-speed} $\mathcal M$ has growth at least $n!$.
  \item \label{nit:exp-speed} $\mathcal M$ has growth at least $2^{\omega(n)}$.
  \item \label{nit:lin-mc} \FO~model checking is not $\FPT$~on $\mathcal M$. 
  (The implication to \cref{nit:lin-mc}  holds if $\FPT\neq \AW[*]$.)  
   \item \label{nit:bd-rd} For all $r \in\mathbb N$ there is a matrix $M \in \mathcal M$ which admits an $r$-rich division.
  \end{enumerate}
  \end{reptheorem}
\begin{proof}
  The implication \cref{nit:bd-tww}$\rightarrow$\cref{nit:bd-rd}
  is by \cref{thm:brd-to-bdtww}, and
  \cref{nit:bd-rd}$\rightarrow$\cref{nit:bd-gr} is by~\cref{thm:rd-to-gr}.
  By~\cref{lem:finite-to-binary} if $\mathcal M$ has unbounded grid rank then some $a$-selection $s_a(\mathcal M)$ has unbounded grid rank.
  Then \cref{nit:bd-gr}$\rightarrow$\cref{nit:ramsey} follows from~\cref{thm:Ram6} applied to $s_a(\mathcal M)$. The implication
  \cref{nit:ramsey}$\rightarrow$\cref{nit:fact-speed} is clear, as each of the classes $\mathcal F_\sym$ has factorial growth.
  The implications 
  \cref{nit:ramsey}$\rightarrow$\cref{nit:nip},
   and~\cref{nit:ramsey}$\rightarrow$\cref{nit:lin-mc} (assuming $\FPT\neq\AW[*]$) 
   are by \cref{lem:independent-matrices}.
   The implications \cref{nit:nip}$\rightarrow$\cref{nit:m-nip} and \cref{nit:fact-speed}$\rightarrow$\cref{nit:exp-speed} are immediate. The implication
   \cref{nit:m-nip}$\rightarrow$\cref{nit:bd-tww}
   is by~\cite{twin-width1}
   \cref{nit:exp-speed}$\rightarrow$\cref{nit:bd-tww} is by \cite{twin-width2},
   whereas \cref{nit:lin-mc}$\rightarrow$\cref{nit:bd-tww} follows from~\cite{twin-width1} and~\cref{thm:approx-tww}.
\end{proof}

\section{Classes of ordered graphs with unbounded twin-width}\label{sec:matchings}

We now move to the world of hereditary classes of ordered graphs.
In this language, we will refine the lower bound on the growth of classes of ordered graphs, in order to match the conjecture of Balogh, Bollob\'as, and Morris~\cite{Balogh06}.
We will also establish that bounded twin-width, NIP, monadically NIP, and tractable (provided that $\FPT \neq \AW[*]$) are all equivalent. 
This will prove our main result,~\cref{thm:main}.

\subsection{Twenty-five unavoidable classes of ordered graphs}\label{sec:jump}
In the context of graph classes,
we rename the parameters $\six$ to $\sax$,
since we interpret the rows/columns as left/right vertices, respectively.

In this section we prove the following two theorems.
  \begin{theorem}\label{thm:25}
    There exist 25 classes of ordered graphs with unbounded twin-width, namely the classes $\mathscr P$ and $\mathscr M_{\sym, \lambda, \rho}$ for $\sym\in\sax$ and $\lambda,\rho\in\set{0,1}$, such that every hereditary class of ordered graphs with unbounded twin-width includes at least one of these classes.  %The class $\mathscr P$ has growth $n!$ and the classes $\mathscr M_{\sym, \lambda, \rho}$ have growth at least $\lfloor\frac n 2\rfloor!$. In particular, they all have unbounded twin-width.
  \end{theorem}

\begin{theorem}\label{thm:hardness}
  Let $\CC$ be a  class of ordered graphs with unbounded twin-width. Then $\CC$ efficiently interprets the class of all graphs. In particular, \FO~model checking is $\AW[\ast]$-hard and $\CC$ is independent.
\end{theorem}

Let  $\sigma\in\mathfrak S_n$ be a permutation.
For a parameter $\sym\in\sax$,
an \emph{$(\sym,\sigma)$-matching} is an ordered graph~$G$ with vertices 
$a_1<\ldots<a_n<b_1<\ldots<b_n$
such that  $a_i$ and $b_j$ are adjacent in $G$ if and only if there 
is a $1$ on position $(i,j)$ in the
matrix $F_\sym(\sigma)$, where now $\sym$ is treated as an element of $\six$.
In other words, the adjacency between $a_i$ and $b_j$ is defined by the Iverson bracket 
\[[\sigma(i)=j],\quad [\sigma(i)\neq j],\quad [i \le \sigma^{-1}(j)],\quad [i\ge \sigma^{-1}(j)],\quad [j \le \sigma(i)],~\text{or}~\quad[j \ge \sigma(i)]
\]
depending on the parameter $\sym\in\sax$.
The vertices  ${a_1,\ldots,a_n}$ are called the \emph{left} vertices, while the vertices  ${b_1,\ldots,b_n}$ are the \emph{right} vertices in the $(\sym,\sigma)$-matching $G$.

\medskip

As a first step, we 
use~\cref{thm:Ram6} to find arbitrary $(\sym,\sigma)$-matchings in graph classes of unbounded twin-width.

  \begin{lemma} 
  \label{lem:bip1}
  Let $\mathscr C$ be a hereditary class of ordered graphs with unbounded twin-width.
  Then there exists $\sym\in\sax$ such that for every  $n\ge 1$ and permutation $\sigma\in\mathfrak{S}_n$, the class $\mathscr C$ contains an $(\sym,\sigma)$-matching.
  \end{lemma}

  \begin{proof}
    Let $\mathcal M$ be the submatrix closure of the set of adjacency matrices of graphs in $\mathscr C$, along their respective orders.
    $\mathcal M$ has unbounded twin-width (see last paragraph of~\cref{subsec:graph-theory}), and hence unbounded grid rank.
    By \cref{thm:Ram6}, there exists some $\sym\in\six$ such that $\mathcal{F}_{\sym}\subseteq \mathcal M$.
    
    Let $\sigma\in \mathfrak{S}_n$ be a permutation.
    We construct a $(\sym,\sigma)$-matching $G\in\mathscr C$.
    Consider its associated \emph{matching permutation} $\widetilde{\sigma}\in \mathfrak{S}_{2n}$ defined by
  \[\widetilde{\sigma}(i):=\begin{cases}
          \sigma(i)+n&\text{if }i \leq n\\
      \sigma^{-1}(i-n) &\text{if } n+1\leq i \leq 2n.
    \end{cases} \]
  In other words $M_{\widetilde{\sigma}}$ consists of the two blocks $M_{\sigma}$ and $M_{\sigma^{-1}}$ on its anti-diagonal. 
  We have $F_{\sym}(\widetilde{\sigma})\in \mathcal M$, so there exists a graph $H\in \mathcal C$ such that $F_{\sym}(\widetilde{\sigma})$ is a submatrix of its adjacency matrix.
  Denote by $U_1, U_2$ the (disjoint) ordered sets of vertices corresponding to the rows indexed respectively by $\sg{1, \ldots,n}$ and $\sg{n+1, \ldots,2n}$, such that $\max(U_1)<\min(U_2)$.
  Take similarly $V_1, V_2$ associated to the column indices.
  If $\max(U_1)<\min(V_2)$ we let $A=U_1$ and $B=V_2$; otherwise, $\min(U_2) > \max(U_1) \geq \min(V_2) > \max(V_1)$ and we let $A=V_1$ and $B=U_2$.
    Then, if $a_1 < \dots < a_n$ are the elements of $A$ and $b_1 < \dots < b_n$ are the elements of $B$, we have $a_n < b_1$ and $a_ib_j\in E(H)$ if and only if $F_\sym(\sigma)$ has a $1$ on position $(i,j)$. Hence  $G=H[A\cup B]$ is an $(\sym,\sigma)$-matching in $\mathscr C$.
  \end{proof}
  
  As a second step, we make the $(\sym,\sigma)$-matchings more organized, by 
  controlling the edges among the left and right parts.
  Let $f,g\from\set{-1,1}\to\set{0,1}$ be two functions.
An $(\sym,\sigma)$-matching $G$ with vertices $a_1<\ldots<a_n<b_1<\ldots<b_n$  is 
\emph{$(f,g)$-regular} if the following hold for $1\le i<j\le n$:
\begin{align*}
[E(a_i,a_j)]&=f(\ot({\sigma(i)},\sigma(j)))\text{\qquad\qquad \it \text{and}}&
[E(b_i,b_j)]&=g(\ot({\sigma^{-1}(i)},\sigma^{-1}(j))).\\
\end{align*}
Let $\mathscr R_{\sym,f,g}$ denote the hereditary closure of the class of all $(f,g)$-regular $(\sym,\sigma)$-matchings, for all permutations~$\sigma$.

\medskip

  We now further improve the statement of~\cref{lem:bip1} to obtain one of the 96 classes~$\mathscr R_{\sym,f,g}$. To this end, from 
  a huge $(\sym,\pi)$-matching we will extract a large $(f,g)$-regular $(\sym,\sigma)$-matching, for some $f$ and $g$.

  \begin{lemma}
    \label{lem:bip2}
  Let $\CC$ be a hereditary class of ordered graphs with unbounded twin-width.
  Then there exist $\sym\in\sax$  and $f,g\from\set{-1,1}\to\set{0,1}$ such that $\CC\supseteq\mathscr R_{\sym,f,g}$.
  \end{lemma}
  \begin{proof}Let $\sym\in\sax$ be given by~\cref{lem:bip1}. We first show that for every $k$ and permutation $\sigma\in \mathfrak S_k$
there is some regular $(\sym,\sigma)$-matching $H\in \CC$.
    Fix a permutation $\sigma\in \mathfrak S_k$.
    Let $N$ and $\pi\in \mathfrak S_N$  be given by~\cref{lem:perm-ramsey}.
By~\cref{lem:bip1} some $(\sym,\pi)$-matching $G$ belongs to $\CC$. Denote its vertices $a_1<\ldots<a_N<b_1<\ldots<b_N$.

Let $c\from [N]^2\to \set{0,1}^2$ be such that 
\[c(i,j)=([E(a_i,a_j)],[E(b_{\pi(i)},b_{\pi(j)}])\qquad\text{for $i,j\in[N]$},\]
where $[E(u,v)]=1$ if the vertices $u,v$  are adjacent in $G$ 
and $0$ otherwise.

Apply~\cref{lem:perm-ramsey} to $c$.
There is a subset $U\subseteq [N]$ 
such that the subpermutation of $\pi$ induced by $U$ is isomorphic to $\sigma$ and
such that $c(i,j)$ depends only on $\ot(i,j)=\ot(a_i,a_j)$ and on $\ot(\pi(i),\pi(j))=\ot(b_{\pi(i)},b_{\pi(j)})$ for 
$u,v\in U$.
Hence, the subgraph $H$ of $G$ 
induced by $\setof{a_i}{i\in U}\cup\setof{b_{\pi(i)}}{i\in U}$ is an $(f,g)$-regular $(\sym,\sigma)$-matching in $\CC$, for some $f,g\from\set{-1,1}
\to\set{0,1}$.

\medskip
Let $\sigma_1,\sigma_2,\sigma_3,\ldots$
be a sequence of permutations such that 
$\sigma_n$ is a subpermutation of $\sigma_{n+1}$, for all $n\ge 1$,
and such that for every $k$ and permutation $\sigma\in\mathfrak S_k$ there is some $n$ such that $\sigma$ is a subpermutation of $\sigma_n$.
For each $n\ge 1$ let $f_n,g_n\from\set{-1,1}\to\set{0,1}$ be such that there is an $(f_n,g_n)$-regular $(\sym,\sigma_n)$-matching in~$\CC$.
Since there are finitely many possible pairs $(f_n,g_n)$,
by taking a subsequence we may assume that  $f_n=f$ and $g_n=g$ for some fixed $f,g\from\set{-1,1}\to\set{0,1}$.
Note that if $\sigma$ is a subpermutation of $\sigma'$
then the $(f,g)$-regular $(\sym,\sigma)$-matching is an induced subgraph of the $(f,g)$-regular $(\sym,\sigma')$-matching.
It follows that $\CC$ contains all $(f,g)$-regular $(\sym,\sigma)$-matchings.
  \end{proof}
  
 \begin{lemma}\label{lem:from 96 to 25}
   If one of the functions $f,g\from\set{-1,1}\to\set{0,1}$ is non-constant then $\mathscr R_{\sym,f,g}\supseteq \mathscr P$.
   Otherwise, $\mathscr R_{\sym,f,g}=\mathscr M_{\sym,\lambda,\rho}$,
   where $\lambda,\rho\in\set{0,1}$ are such that $\lambda=f(-1)=f(1)$ and $\rho=g(-1)=g(1)$.
 \end{lemma} 
\begin{proof}Consider an $(f,g)$-regular $(\sym,\sigma)$-matching $G$ and let $G[L]$ and $G[R]$ be its ordered subgraphs induced by the left and right vertices, respectively.

  If $f(1)=0$ and $f(-1)=1$ then by definition, $G[L]$ is isomorphic to $G_\sigma$.
  Similarly, if $g(1)=0$ and $g(-1)=1$ then $G[R]$ is isomorphic to $G_{\sigma^{-1}}$.
  It follows that if $f(1)<f(-1)$ or $g(1)<g(-1)$ then $\mathscr R_{\sym,f,g}$ contains $\mathscr P$.

  On the other hand, if $f(1)>f(-1)$ then $G[L]$ is isomorphic to the edge complement of $G_\sigma$,
  and if $g(1)>g(-1)$ then $G[R]$ is isomorphic to the edge complement of $G_{\sigma^{-1}}$. Therefore, if $f(1)>f(-1)$ or $g(1)>g(-1)$ then $\mathscr R_{\sym,f,g}$ contains the class of edge complements of ordered graphs in $\mathscr P$,
but this class is again
 $\mathscr P$ (see~\cref{lem: class P}).
%  , since the edge complement of $G_\sigma$ is $G_{\bar \sigma}$ where $\bar\sigma$
% is such that $\bar\sigma(i)<\bar\sigma(j)$ if and only if $\sigma(i)>\sigma(j)$.

If $f$ and $g$ are constantly equal to $\lambda$ and $\rho$, respectively, 
then $\mathscr R_{\sym,f,g}=\mathscr M_{\sym,\lambda,\rho}$ by definition.
\end{proof}

To prove \cref{thm:hardness}, it suffices to 
show that each of the classes $\mathscr R_{\sym,\lambda,\rho}$ efficiently interprets the class of all graphs.

\begin{lemma}\label{lem: find middle point}
For every $\sym\in\sax$ and $f,g\from\set{-1,1}\to\set{0,1}$, the class $\mathscr R_{\sym,f,g}$ efficiently interprets the class $\mathscr M$ of all ordered matchings, via an interpretation $\interp I$ which maps $n$-element structures to $n$-element structures. In particular, $\mathscr R_{\sym,f,g}$ has growth at least $\lfloor\frac n 2\rfloor!$.
\end{lemma}
\begin{proof}
  Suppose $\sym$ is `$\leqslant_r$', the other cases being similar.
Let $H$ be an ordered matching with left vertices $L$ and right vertices $R$, 
and let $\sigma\from L\to R$ be the bijection mapping each $v\in L$ to 
its neighbor in $R$.

Let $G$ be the  $(f,g)$-regular $(\sym,\sigma)$-matching, with vertices $L$ and $R$.
As $G$ is regular, it is fully determined
by $H$, $\sym,\lambda,\rho$, 
and can be constructed from $H$ in polynomial time.

\medskip
We now show how to define the edges of $H$ from $G$, using an \FO~formula not depending on $H$.

The following formula with parameter $z$ and two variables $x,y$ expresses that $x\le z$ and $y$ is the least neighbor of $E$ larger than $z$:
	\[\mu(x,y;z)\equiv (x\le z)\land (y>z)\land E(x,y)\land \forall y'. (z<y'< y)\rightarrow \neg E(x,y').\]
	
	Let $\rho(z)$ be the formula expressing that the set of pairs 
$x,y$ satisfying $\mu(x,y;z)$ defines the graph of a bijection between $\setof{x}{x\le z}$ and $\setof{x}{x>z}$.
There is a unique vertex in $G$ satisfying $\rho(z)$, namely the largest element of $L$,
	since  $|L|=|R|$.
	Consider the formula 
  \[\phi(x,y)\equiv \exists z.\rho(z)\land (\mu(x,y;z)\lor\mu(y,x;z)).\] Then 
	$\phi(x,y)$ defines in $G$ the matching 
  $H$, directed from $L$ to $R$.
	Hence, $\phi(x,y)\lor\phi(y,x)$ defines 
  the edges of the matching $H$ in $G$,
  whereas $x<y$ defines the order of $H$.
  Together those formulas form an interpretation  $\interp I$ such that $\interp I(G)=H$. 

  As the interpretation $\interp I$ from~\cref{lem:matchings are dependent} preserves the domains of the structures, and there are $\lfloor\frac n 2\rfloor!$ ordered matchings with $n$ elements, it follows that the class $\mathscr R_{\sym,\lambda,\rho}$ has growth at least $\lfloor\frac n 2\rfloor!$.
\end{proof}

\medskip
\cref{thm:25,thm:hardness} now follow.
\begin{proof}[Proof of~\cref{thm:25}]
  By 
  \cref{lem:bip2,lem: find middle point,lem:from 96 to 25}, every class with unbounded twin-width contains one of the classes $\mathscr M_{\sym,\lambda,\rho}$ or $\mathscr P$. The former have growth at least $\lfloor \frac n 2\rfloor!$ by~\cref{lem: find middle point}, while the latter has growth $n!$ by~\cref{lem: class P}. In particular, they have unbounded twin-width by~\cref{thm:small}.
  % It remains to show that $\mathscr P$ has growth $n!$. This is because the mapping which maps an $n$-permutation $\pi$ to the ordered graph $G_\pi\in\mathscr P$ is injective, as $\pi$ can be recovered as the unique permutation satisfying $\pi(i)>\pi(j)\Leftrightarrow ij\in E(G_\pi)$, for all $1\le i<j\le n$.
\end{proof}
\begin{proof}[Proof of~\cref{thm:hardness}]
  Follows by \cref{lem:bip2,lem: find middle point,lem:matchings are dependent}.
\end{proof}

% \begin{corollary}
%   Let $\CC$ be a class of ordered graphs with unbounded twin-width.
%   Then:
%   \begin{itemize}
%     \item $\CC$ is independent,
%     \item \FO~model checking is $\AW[*]$-hard on $\CC$,
%     \item $\CC$ has growth at least $\lfloor\frac n 2\rfloor!$.
%   \end{itemize}
% \end{corollary}
In the next section, we improve the lower bound on the growth of the classes $\mathscr M_{\sym,\lambda,\rho}$ further, to match the growth  of the hereditary closure of the class of matchings.

\subsection{Lowerbounding the growth of $\mathscr M_{\sym,\lambda,\rho}$}

There is still a bit of work to get the exact value of $\sum_{k=0}^{\lfloor n/2\rfloor}\binom{n}{2k}\,k!$ conjectured in \cite{Balogh06} as a lower bound of the growth of hereditary classes of ordered graphs
with superexponential growth.
We show how to derive this bound for $\mathscr M_{\sym,\lambda,\rho}$ in each case of $\sym, \lambda, \rho$. 

% \begin{lemma}
% \label{lem:strict_nonstrict}
% 	For every $\lambda,\rho\in\{0,1\}$ we have
% 	$\mathscr M_{\leq_r,\lambda,\rho}=\mathscr M_{<_r,\lambda,\rho}$.
% \end{lemma}
% \begin{proof}
% 	Let $M\in \mathscr M_{=,0,0}$ be an ordered matching with vertices $a_1<\dots<a_n<b_1<\dots<b_n$ and matching function $\sigma$.
% 	Let $G=\enc_{\leq_r,\lambda,\rho}$.
% 	We consider the matching $M'$ with vertices $a_1'<a_1<\dots<a_n'<a_n<b_1<b_1'<\dots<b_n<b_n'$, where $a_i$ is matched with $b_{\sigma(i)}'$ and $a_i'$ is matched with $b_{\sigma(i)}$. Let $H=\enc_{<_r,\lambda,\rho}(M')$.
% 		Then $a_i$ is adjacent to $b_j$ if and only if $j\leq\sigma(i)$.  It follows that the subgraph of $H$ induced by $\{a_1,\dots,a_n,b_1,\dots,b_n\}$ is $G$, thus $G$ and its induced subgraphs belong to $\mathscr M_{<_r,\lambda,\rho}$. Thus $\mathscr M_{\leq_r,\lambda,\rho}\subseteq\mathscr M_{<_r,\lambda,\rho}$.
% 		The converse inclusion is proved in a similar way.
% \end{proof}

\begin{lemma}
\label{lem:inj}
	For every $\sym\in\sax$ and $\lambda,\rho\in\set{0,1}$ and every positive integer $n$ we have
	\begin{align}\label{eq:balogh}
|(\mathscr M_{\sym,\lambda,\rho})_n|\geq \sum_{k=0}^{\lfloor n/2\rfloor}\binom{n}{2k}\,k!,
  \end{align}
  with equality in the case when $\sym$ is `$=$' and $\lambda=\rho=0$.
\end{lemma}
\begin{proof}
We first observe that symmetries allow to reduce the number of classes to consider.

\begin{claim}
\label{cl:red}
	It is sufficient to consider the following $7$ classes: $\mathscr M_{\sym,0,\rho}$ with $\sym\in\{=,\leq_r,\geq_r\}$ and $\rho\in\set{0,1}$, and $\mathscr M_{=,1,1}$.
\end{claim}
\begin{claimproof}
	Considering the complements of the graph and/or a reverse linear order we can first restrict our study to the case where $\sym\in \{=,\leq_r\}$. If $\lambda=0$ these classes are in the claimed list, so assume $\lambda=1$. If $\sym$ is `$=$' we consider the class with reversed linear order, which has type $\mathscr M_{=,\lambda',1}$ for some $\lambda'\in\set{0,1}$, and is therefore in the claimed list.
		Finally, if the class is
		$\mathscr M_{\leq_r,1,\rho}$ 	we consider the class of its edge complements, which is equal\footnote{To be more explicit, the class of edge complements of graphs in $\mathscr M_{\leq_r,1,\rho}$ directly  corresponds to a class which could be denoted $\mathscr M_{>_r,0,\rho'}$, which is however equal to $\mathscr M_{\ge_r,0,\rho'}$
    as the classes are hereditary; see analogous discussion following~\cref{def: mats}.} to $\mathscr M_{\geq_r,0,\rho'}$ where $\rho'=1-\rho$.
\end{claimproof}

  Let $\mathscr M_{\sym,\lambda,\rho}$ be one of the 7 classes above.
  We now prove that~\cref{eq:balogh} holds for any given $n\ge 1$.
In the case of $\mathscr M_{=,1,1}$ we assume $n>3$ and verify the claim for the values $n=1,2,3$ by separately. Indeed, the first $3$ values of the growth function are $1,2$, and $6$, so~\eqref{eq:balogh} holds.

    In the general case, 
     we obtain the claimed number of distinct $n$-element ordered graphs $G$ in $\mathscr M_{\sym,\lambda,\rho}$ by the following process.

      Pick a number $k$ with $0\le 2k\le n$ and $2k$ elements $a_1<\ldots<a_k<b_1<\ldots<b_k$ in $[n]$.
      Pick a matching $E$ between the $a_i$'s and $b_j$'s in one of $k!$ possible ways.
      If $k=0$ let $m=n$, otherwise, 
      let $m=b_1-1$.
     Denote $L=[1,m]$ and $R=[m+1,n]$.
     Then $E$ is an `ordered partial matching' with vertices $L\cup R$ (see~\cref{fig:M=00}).
      Construct the ordered graph $G$ with vertices $[n]=L\cup R$ and edges:
      \begin{itemize}
        \item $E$ if  $\sym$ is `$=$'
        \item $\setof{uv'}{uv\in E, u<v\le v', v'\in R}$ if $\sym$ is `$\ge_r$',
        \item $\setof{uv'}{uv\in E, u<v'\le v, v'\in R}$ if $\sym$ is `$\le_r$',        
      \end{itemize}
      and additionally form in $G$ a clique on $L$  if $\lambda=1$ and a clique on $R$ if $\rho=1$.

      Clearly, the number of possible outcomes of the above process is at most $\sum_{k=0}^{\lfloor n/2\rfloor}\binom{n}{2k}\,k!$.
      Observe that every ordered graph $G\in \mathscr M_{=,0,0}$ can be obtained as a result of the above process, so $|(\mathscr M_{=,0,0})_n|\le \sum_{k=0}^{\lfloor n/2\rfloor}\binom{n}{2k}\,k!$.

      We now argue that different  choices made above lead to non-isomorphic outcomes $G$. Hence, the number of possible outcomes is exactly
      $\sum_{k=0}^{\lfloor n/2\rfloor}\binom{n}{2k}\,k!$.      
      To this end, we show how to recover $E$ from $G$.

      Suppose first that $\lambda=0$.
      Then $m$ is the largest vertex in $[n]$ such that $[1,m]$ forms an independent set, since $m+1=b_1$ is adjacent to $a_{\sigma^{-1}(1)}<b_1$
      (unless $k=0$, but then $[1,n]$ is still an independent set).
      Also,
       for $1\le i\le k$, the vertex  $b_{\sigma(i)}$ is the 
      unique, resp. smallest, resp. largest, neighbor of $a_i$ which is larger than $m$, depending on $\sym\in\set{=,\ge_r,\le_r}$.
      Moreover, all remaining vertices $u\in L\setminus\set{a_1,\ldots,a_k}$ are isolated in $G$. This allows us to  recover $a_1,\ldots,a_k,b_1,\ldots,b_k$ as well as the matching  $E$. 
      
      Suppose now that $\lambda,\rho=1$ and $\sym$ is `$=$'.
      Then $m$ is the largest vertex $v$ such that $[1,v]$ forms a clique in $G$,
  unless $m=1,k=1,a_1=1,b_1=2$,
  so that $G$ 
  is the ordered graph $G_n$ obtained from the clique on $[2,n]$ by adding the edge joining $1$ with $2$.
  So unless $G=G_n$, we can recover $m$ as described above.
  If $G=G_n$, then $m$ can be still determined, since $m+1$ is the smallest vertex $v'$ such that $[v',n]$ forms a clique, unless $m=n-1,k=2,a_{1}=n-1,b_1=n$  so that $G$ is the ordered graph $G_n'$ obtained from the  clique on $[1,n-1]$ by adding the edge joining $(n-1)$ with $n$. 
  However, if $G=G_n=G_n'$ implies $n=3$,
    but we have assumed that $n>3$. 
Hence, we can determine $m$, given $G$. Then $E$ is recovered from $G$ by removing all the edges $uv$ with $u<v\le m$ or $m<u<v$.

 \begin{figure}\centering
  \includegraphics[scale=0.55,page=10]{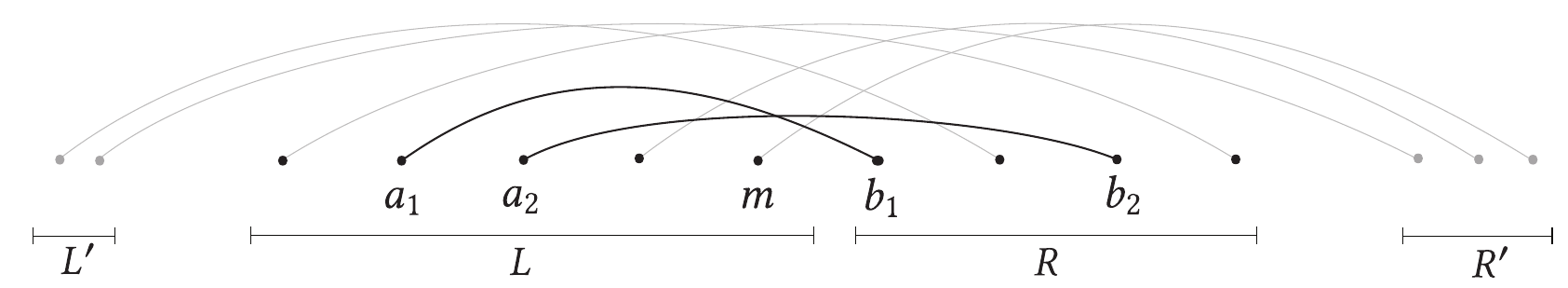}
  \caption{
    The construction from~\cref{lem:inj}. The process described in the proof starts with an ordered partial matching, marked with black vertices and edges above. This is then completed to an ordered matching, by matching the isolated vertices with newly added vertices. The vertices $R'$ are inserted after $\max R$ if $\sym$ is `$=$' or `$\ge_r$', and between $\max L$ and $\min R$ if $\sym$ is `$\le_r$'.
  }
  \label{fig:M=00}
\end{figure}

    \medskip
      
      It remains to verify that $G\in\mathscr M_{\sym,\lambda,\rho}$. 
      Let $A'=L\setminus\set{a_1,\ldots,a_k}$ and $B'=R\setminus\set{b_1,\ldots,b_k}$.
      Extend the ordered partial matching $E$ to an ordered matching $H$ (see \cref{fig:M=00}) as follows.
      Create a set $L'$ of $|B'|$ new vertices matched with $B'$ arbitrarily, such that $\max L'<\min L=1$.
      Also, create a set $R'$ of $|A'|$ new vertices matched with $A'$ arbitrarily, such that $\min R'>\max R=n$ if $\sym$ is `$=$' or `$\ge_r$',
      and $m=\max L<\min R'<\max R'<\min R=m+1$ if $\sym$ is `$\le_r$'.
      Then $H$ is an ordered matching with vertices  $(L\cup L')\cup (R\cup R')$.
      Construct the ordered graph $G'$ with the same vertices and order, and with edges:
      \begin{itemize}
        \item $E(H)$ if  $\sym$ is `$=$'
        \item $\setof{uv'}{uv\in E(H), v\le v', v'\in R\cup R'}$ if $\sym$ is `$\ge_r$',
        \item $\setof{uv'}{uv\in E(H), v'\le v, v'\in R\cup R'}$ if $\sym$ is `$\le_r$',        
      \end{itemize}
      and additionally form in $G'$ a clique on $L\cup L'$ if $\lambda=1$ and a clique on $R\cup R'$ if $\rho=1$. 
      Then $G'=H[\sym,\lambda,\rho]$ in the terminology of~\cref{subsec:patterns}, so $G'\in\mathscr M_{\sym,\lambda,\rho}$. As $G$ is an induced subgraph of $G'$, this shows $G\in\mathscr M_{\sym,\lambda,\rho}$.
Altogether, this proves~\cref{eq:balogh}.
    \end{proof}

\subsection{Proof of~\cref{thm:main}}
Finally, we can piece together the proof of our main result,~\cref{thm:main}, 
which we restate below in full generality for arbitrary classes of ordered, binary structures (see \cref{subsec:fmt} for the definition of an atomic type $\tau$ and the interpretation $\interp I_\tau$).

\begin{reptheorem}{thm:main}
	Let $\CC$ be a hereditary class of finite ordered binary structures.
	Then either $\CC$ satisfies conditions~(i)-(v), or $\CC$ satisfies conditions~(i')-(v') below:
	\begin{multicols}{2}%		
	\begin{enumerate}[label={(\roman*)},ref=\roman*]
			\item\label{git:tww} $\CC$ has bounded twin-width			
			 \item\label{git:matrix} $\CC$ has bounded~grid rank
			\item\label{git:small} $\CC$ has growth $2^{\Oof(n)}$
			
			 \item\label{git:transduce} $\CC$ does not transduce the class of all graphs
			 \item\label{git:easy} \FO~model checking  is \FPT~on $\CC$
			
	\end{enumerate}
	\begin{enumerate}[label={(\roman*')},ref=\roman*']
	\item\label{git:utww} 	$\CC$ has unbounded twin-width
	\item\label{git:patterns} there is some atomic type $\tau(x,y)$ such that 
	  $\interp I_\tau(\CC)$ contains $\mathscr P$ or one of the 24 classes $\mathscr M_{\sym,\lambda,\rho}$
	 \item\label{git:big} $\CC$ has growth at least $\sum_{k=0}^{\lfloor n/2 \rfloor} {n \choose 2k}k!\ge\lfloor \frac n 2\rfloor!$
	 
	  \item\label{git:interp} $\CC$ interprets the class of all graphs
	  \item\label{git:hard} \FO~model checking  is $\AW[*]$-hard on~$\CC$.
\end{enumerate}
		\end{multicols}
\end{reptheorem}

\begin{proof}%[Proof of~\cref{thm:main}]
  If $\CC$ has bounded twin-width then~\cref{git:small} holds by~\cite{twin-width2}, \cref{git:transduce} holds by~\cite{twin-width1}, and~\cref{git:easy} holds by~\cite{twin-width1} and~\cref{thm:approx-tww},
  and finally, \cref{git:matrix} holds by~\cref{thm:equiv}~\cref{it:bd-gr},\cref{it:bd-tww}.

  Conversely, if $\CC$ has unbounded twin-width then the class $\mathcal M=\setof{M(\str A)}{\str A\in\CC}$ of adjacency matrices of $\CC$ has unbounded twin-width. Recall that $M(\str A)$ is the adjacency matrix of $\str A$ over the alphabet $A_\Sigma$ consisting of atomic types of pairs of elements (see definition preceding \cref{lem:reducing matrices to structures}). By \cref{thm:equiv} \cref{it:bd-tww} \cref{it:lin-mc}, model checking is $\AW[*]$-hard on $\cal M$. 
 By~\cref{lem:reducing matrices to structures}, this yields $\AW[*]$-hardness for $\CC$, proving~\cref{git:hard}.
  By~\cref{thm:equiv} \cref{it:ramsey},\cref{it:bd-tww} there is some letter (atomic type) $\tau\in A_\Sigma$ such that the selection $s_\tau(\cal M)$ contains one of the classes $\mathcal F_\sym$, so in particular, $s_\tau(\cal M)$ has unbounded twin-width. 

\begin{claim}
  The class $\interp I_\tau(\CC)$ has unbounded twin-width.
\end{claim}
\begin{proof}
  Let $\interp I_\tau'$ be the interpretation which, given $\str A\in\CC$, produces the 
  ordered directed graph with the same domain and order as $\str A$, such that there is an edge from $u$ to $v$ (possibly $u=v$) if and only if $\tau(u,v)$ holds in $\str A$. 
  Note that for $\str A\in \CC$, the matrices $s_\tau(M(\str A))$ and $M(\interp I_\tau'(\str A))$ are equal, 
  as both matrices have entry equal to $1$ if $\tau(a,b)$ holds in $\str A$ and $0$ otherwise. 
  Hence, $s_\tau(\cal M)$ is equal to the class of adjacency matrices of the ordered graphs in $\interp I_\tau'(\CC)$.
As $\cal M$ has unbounded twin-width, also $\interp I_\tau'(\CC)$ has unbounded twin-width.
  As the class $\interp I_\tau(\CC)$ interprets the class $\interp I_\tau'(\CC),$ it follows that $\interp I_\tau(\CC)$ has unbounded twin-width, too.
\end{proof}
   
Since $\interp I_\tau(\CC)$ is a class of ordered graphs, we may apply the results of the preceding sections. In particular, $\interp I_\tau(\CC)$ contains one of the 25 classes $\mathscr P$ and $\mathscr M_{\sym,\lambda,\rho}$ by~\cref{thm:25}, proving~\cref{git:patterns}. By~\cref{lem:inj}, $\interp I_\tau(\CC)$, and so also $\CC$, has growth at least $\sum_{k=0}^{\lfloor n/2\rfloor}\binom{n}{2k}\,k!$, proving~\cref{git:big}. By~\cref{thm:hardness}, $\interp I_\tau(\CC)$, and so also $\CC$, interprets the class of all graphs. This proves~\cref{git:interp}, and concludes the theorem.
\end{proof}

\subsection{Minimality}
\begin{lemma}
	None of the classes $\mathscr P$ and $\mathscr M_{\sym,\lambda,\rho}$ is included in another.
\end{lemma}
\begin{proof}
We first prove that the class $\mathscr P$ is incomparable with the classes $\mathscr M_{\sym,\lambda,\rho}$: The class $\mathscr P$ contains the ordered graph $\xy
\POS(0,0)*\cir<1pt>{}="a", (3,0)*\cir<1pt>{}="b",
	(6,0)*\cir<1pt>{}="c",
	(9,0)*\cir<1pt>{}="d",
	(12,0)*\cir<1pt>{}="e",
	(15,0)*\cir<1pt>{}="f",
\POS"a"\ar@/^.5ex/@{-} "b",
\POS"c"\ar@/^.5ex/@{-} "d",
\POS"e"\ar@/^.5ex/@{-} "f",
\endxy$ (as witnessed by the permutation $(214365)$) but none of the classes $\mathscr M_{\sym,\lambda,\rho}$ does (as all the ordered graphs in these classes admits a vertex partition into two intervals $L$ and $R$ with $\max L<\min R$, each 
	inducing a complete or an edgeless graph).
On the other hand, none of 
$\xy
\POS(0,0)*\cir<1pt>{}="a", (3,0)*\cir<1pt>{}="b",
	(6,0)*\cir<1pt>{}="c",
\POS"a"\ar@/^1ex/@{-} "c"
\endxy\,$ and $\xy
\POS(0,0)*\cir<1pt>{}="a", (3,0)*\cir<1pt>{}="b",
	(6,0)*\cir<1pt>{}="c",
\POS"a"\ar@/^.5ex/@{-} "b",
\POS"b"\ar@/^.5ex/@{-} "c"
\endxy$ are ordered permutation graphs, but each class $\mathscr M_{\sym,\lambda,\rho}$ contains one of these.
Indeed, as this set of two ordered graphs is closed by complement and order reversal it is sufficient to consider the next $6$ classes: $\mathscr M_{\sym,0,\rho}$ with $s\in\{=,\leq_r,\geq_r\}$ and $\mathscr M_{=,1,1}$ (see \cref{cl:red}). For $\mathscr M_{=,1,1}$ we immediately check that $\xy
\POS(0,0)*\cir<1pt>{}="a", (3,0)*\cir<1pt>{}="b",
	(6,0)*\cir<1pt>{}="c",
\POS"a"\ar@/^.5ex/@{-} "b",
\POS"b"\ar@/^.5ex/@{-} "c"
\endxy\in \mathscr M_{=,1,1}$. It is easily checked that $\xy
\POS(0,0)*\cir<1pt>{}="a", (3,0)*\cir<1pt>{}="b",
	(6,0)*\cir<1pt>{}="c",
\POS"a"\ar@/^1ex/@{-} "c"
\endxy\in \mathscr M_{\sym,0,0}$ for $s\in\{=,\leq_r,\geq_r\}$ and that $\xy
\POS(0,0)*\cir<1pt>{}="a", (3,0)*\cir<1pt>{}="b",
	(6,0)*\cir<1pt>{}="c",
\POS"a"\ar@/^.5ex/@{-} "b",
\POS"b"\ar@/^.5ex/@{-} "c"
\endxy\in \mathscr M_{\sym,\lambda,\rho}$ for $s\in\{=,\leq_r,\geq_r\}$ and $(\lambda,\rho)\neq(0,0)$. Hence $\mathscr P$ is incomparable with classes $\mathscr M_{\sym,\lambda,\rho}$.

Thus we can restrict our attention to the comparison of classes $\mathscr M_{\sym,\lambda,\rho}$ and $\mathscr M_{\sym',\lambda',\rho'}$.
We consider the ordered matching 
$M=\xy
\POS(0,0)*\cir<1pt>{}="a", (3,0)*\cir<1pt>{}="b",
	(6,0)*\cir<1pt>{}="c",
	(9,0)*\cir<1pt>{}="d",
	(12,0)*\cir<1pt>{}="e",
	(15,0)*\cir<1pt>{}="f",
	(18,0)*\cir<1pt>{}="g",
	(21,0)*\cir<1pt>{}="h",
\POS"a"\ar@/^1.5ex/@{-} "g",
\POS"b"\ar@/^1.5ex/@{-} "h",
\POS"c"\ar@/^.75ex/@{-} "e",
\POS"d"\ar@/^.75ex/@{-} "f"
	\endxy$ and let $G=M[\sym,\lambda,\rho]$ be the graph in $\mathscr M_{\sym,\lambda,\rho}$ originating from $M$. Assume $G$ is an induced subgraph of a graph $G'\in\mathscr M_{\sym',\lambda',\rho}'$, originating from an ordered matching $M'$, so that $G'=M'[\sym',\lambda',\rho']$, and that $M'$ has a minimal number of vertices.
Note that there is exactly one vertex-partition of $G'$ into intervals $L,R$ with $\max L<\min R$ such that both $L$ and $R$ are either cliques or independent sets. This readily implies $\lambda'=\lambda$ and $\rho'=\rho$.
 By deleting the internal edges of the parts $L$ and $R$ we reduce to the case where $\lambda=\rho=0$. 
Every ordered graph in $\mathscr M_{\sym',0,0}$ with parts $L$ and $R$ with $|L|=|R|=4$ has the following properties: if $s'$ is `$=$' then the maximum degree is at most $1$, if $s'$ is `$\neq$' then the minimum 
	degree is at least $3$, if $s'$ is `$\geq_r$' then the last vertex has maximum degree in $R$, if $s'$ is `$\leq_r$' then $\min R$ has maximum degree in $R$, if $s'$ is `$\leq_l$' then the first vertex has maximum degree in $L$, and if $s'$ is `$\leq_r$' then $\min L$ has maximum degree in $L$. 
When $s$ is $=,\neq,\leq_r,\geq_r,\leq_l,\geq_l$ the degrees of the vertices of $G$ are $(1,1,1,1,1,1,1,1)$, $(3,3,3,3,3,3,3,3)$,
	$(3,4,1,2,4,3,2,1)$,
	$(2,1,4,3,1,2,3,4)$, 
	$(4,3,2,1,3,4,1,2)$, and
	$(1,2,3,4,2,1,4,3)$. Thus
	it is easily checked that the only possible choice is $s'=s$.
\end{proof}

By the same argument as in~\cref{cor:stanley-wilf1}, we get:
\begin{corollary}\label{cor:stanley-wilf2}
  The 25 classes  $\mathscr P$ and $\mathscr M_{\sym,\lambda,\rho}$  for $\sym\in\sax$ and $\lambda,\rho\in\set{0,1}$ are precisely all the hereditary classes of ordered graphs $\CC$
  with growth $2^{\omega(n)}$ such that every proper subclass of $\CC$ has growth $2^{O(n)}$. 
\end{corollary}

\newcommand{\tp}{\mathrm{tp}}%type

\newcommand{\tup}[1]{\bar{#1}} %tuple
\newcommand{\St}[1][]{\mathrm{S}^{#1}}
\newcommand{\Types}[1][]{\mathrm{Types}^{#1}}

\renewcommand{\cal}[1]{\mathcal {#1}}
\renewcommand{\le}{\leqslant}
\renewcommand{\ge}{\geqslant}
\renewcommand{\phi}{\varphi}
\newcommand{\N}{\mathbb{N}}
\newcommand{\Z}{\mathbb{Z}}
\newcommand{\Ll}{{\mathcal L}}
\newcommand{\Rr}{{\mathcal R}}

\newcommand{\Av}{\mathrm{Av}}
\renewcommand{\subset}{\subseteq}
\newcommand{\ind}[1][]{%
  \mathrel{
    \mathop{
      \vcenter{
        \hbox{\oalign{\noalign{\kern-.3ex}\hfil$\vert$\hfil\cr
              \noalign{\kern-1.5ex}
              $\vert$\cr\noalign{\kern-.3ex}}}
      }
    }\displaylimits_{#1}
  }
}

\newcommand{\nind}[1][]{%
  \mathrel{
    \mathop{
      \vcenter{
        \hbox{\oalign{\noalign{\kern-.3ex}\hfil$\nmid$\hfil\cr
              \noalign{\kern-1.5ex}
              $\vert$\cr\noalign{\kern-.3ex}              
              }}
      }
    }\displaylimits_{#1}
  }
}

\section{Model-theoretic characterizations}
\label{sec:model theory}
In this section, we present further model-theoretic characterizations of classes of bounded twin-width,
as well as prove more general results concerning arbitrary classes of structures, over an arbitrary signature. In particular, we generalize 
the implications~
\cref{it:bd-gr}$\rightarrow$
\cref{it:ramsey}$\rightarrow$
\cref{it:m-nip} from Theorem~\ref{thm:equiv}
to arbitrary classes of structures,
by proving~\cref{thm:summary-mt0}.
Namely, we show that every monadically dependent class of structures excludes certain grid-like patterns,
and every class of structures which excludes such grid-like patterns satisfies  a property generalizing bounded grid rank.

We start with defining the  notion of a restrained class,
generalizing the notion of bounded grid rank for matrices.
First, we introduce a notion generalizing the concept of the number of distinct rows in a zone of a matrix, in arbitrary structures.

\medskip
In this section, whenever $\str S$ is a structure then we identify $\str S$ with its domain, when writing e.g. $a\in\str S$ or $A\subset \str S$. We also write $\str S^{\tup x}$ for the set of all \emph{valuations} $\tup a$ of a set of variables $\tup x$ in $\str S$,
where a valuation is a function $\tup a\from \tup x\to \str S$.

% If $\phi(\tup x;\tup y)$ is a formula with free variables contained in $\tup x$ and $\tup y$,
% and $\tup b\in \str S^{\tup y}$ is a tuple of elements of $\str S$
% then  $\phi(\str S;\tup  b)$ denotes the set of all 
% tuples $\tup a\in\str S^{\tup x}$ such that $\phi(\tup a;\tup b)$ holds in $\str S$.
Let $\Delta(\tup u;\tup v)$ be a finite set of formulas $\theta(\tup u;\tup v)$ with free variables contained in $\tup u$ and $\tup v$.
For a structure $\str S$,
 tuple  $\tup a\in \str S^{\tup u}$ and a set $B\subset \str S$  define the \emph{$\Delta$-type of $\tup a$ over $B$} as:
 \[\tp^\Delta(\tup a/B)=\setof{(\theta,\tup b)\in \Delta\times B^{\tup v}}{\str S\models\theta(\tup a;\tup b)}.\]		
% \end{definition}
% Equivalently -- up to bijection -- $\tp^\theta(\tup a/B)$ is the set of formulas $\theta(\tup u;\tup b)$ with parameters $\tup b$ from $B$ that are satisfied by $\tup a$ in $\str S$. 
For a set $A\subset \str S$, denote
\[\Types[\Delta](A/B):=\setof{\tp^\Delta(\tup a/B)}{\tup a\in A^{\tup u}}.\]
% Note that if $\Types[\Delta](A/B)\le k$ then
% $\Types[\hat\Delta](B/A)\le 2^{k|\Delta|}$,
% where $\theta(\tup v; \tup u)\in\hat\Delta(\tup v;\tup u)$ if and only if $\theta(\tup u;\tup v)\in\Delta(\tup u;\tup v)$.
\begin{example}
	Let $M$ be a 0-1 matrix, viewed as an (ordered) binary structure  with the unary predicate $R\subset M$ indicating the rows and the binary relation $E$ defining the entries of the matrix. Let $\Delta(u,v)=\set{E(u,v)}$.
	Let  $B\subset M\setminus R$ be a set of columns of $M$.
	Then, for a row $a\in R$, the set $\tp^\Delta(a/B)$ corresponds to the set of those columns $b\in B$ with a non-zero entry in row $a$.
	For a set of rows $A\subset R$,
	$|\Types[\Delta](A/B)|$ is the number of distinct rows in the submatrix of $M$ with rows $A$ and columns $B$.
\end{example}

% \medskip

% If $\Sigma$ is a binary signature then denote
% \[\Types[\Sigma](A/B):=\setof{\tp^R(a/B)}{\tup a\in A, R\in\Sigma}.\]
% Similarly as above, if $\Types[\Sigma](A/B)\le k$ then $\Types[\Sigma](B/A)\le 2^{|\Sigma| k}$.
% Moreover, if $\Types[R](A/B)\le k_R$ for all $R\in\Sigma$ then $\Types[\Sigma](A/B)\le \prod_{R\in \Sigma}k_R$.

% \begin{definition}
% 	Fix $k,t\in\N$. An ordered structure $\str A$ 
% 	over a binary signature $\Sigma$
% 	is \emph{$(k,t)$-simple} if for every two partitions 
% 	$\cal L,\cal R$ of the domain of $\str A$ into convex subsets, if $|\cal L|\ge t$ and $|\cal R|\ge t$ then there is some $A\in \cal L$ and $C\in \cal R$ such that $\Types[\theta](A/B)\le k$,
% 	for all quantifier-free formulas $\theta(x,y)$ with two free variables. 
% \end{definition}
% A class $\CC$ is $(k,t)$-simple if every $\str A\in\CC$ is $(k,t)$-simple.

% \medskip

% The following result generalizes 
%  Theorem~\ref{thm:brd-to-bdtww} to classes of binary, ordered structures.

% \begin{proposition}\label{prop:restrained-tww}
% 	Let $\CC$ be a class of finite, ordered binary structures. Then $\CC$ has bounded twin-width if and only if $\CC$ is $(k,t)$-simple, for some $k,t\in\N$.
% \end{proposition}
% \begin{proof}	
% \end{proof}

Let $\phi(x;\tup y)$ be a formula and $\str S$ a structure.
A \emph{$\phi$-definable disjoint family} 
is a family $\cal R$ of pairwise disjoint of subsets of $\str S$, where for each $R\in\cal R$ there is $\tup b\in\str S^{\tup y}$ with $R=\setof{a\in\str S}{\str S\models\phi(a;\tup b)}$.
For example, if $\str S$ is a finite ordered structure and 
$\cal R$ is a  partition of $\str S$ into convex sets, then $\cal R$ is a $\phi$-definable family of pairwise disjoint sets, for $\phi(x;y_1,y_2)=y_1\le x\le y_2$.

% \paragraph{$\theta$-types over a set.}

\begin{definition}[Restrained class]\label{def:restrained}
A  class $\CC$ of structures is 
\emph{restrained}
if the following condition holds.
Let $\phi(x;\bar y)$, $\psi(x;\bar z)$ and be formulas  over the signature of $\CC$,
and let $\Delta(\tup u;\tup v)$ be a finite set of formulas.
Then there are natural numbers $t$ and $k$ such that for any $\str S\in \CC$ and any
$\phi$-definable disjoint family $\cal R$  and $\psi$-definable disjoint family $\cal L$ with  $|\cal R|=|\cal L|\ge t$ there are $R\in \cal R$ and $L\in \cal L$ with $|\Types[\Delta](R/L)|< k$.
\end{definition}

The following proposition is an analogue of the statement that matrices of bounded grid rank have bounded twin-width.

\begin{proposition}\label{prop:restrained-tww}
	Let $\CC$ be a class of finite, ordered binary structures. If $\CC$ is restrained then $\CC$ has bounded  twin-width.
\end{proposition}
\begin{proof}
	Consider the formula $\phi(x;y_1,y_2)\equiv y_1\le x\le y_2$..
Then every partition of a structure $\str S\in\CC$ into convex sets is a $\phi$-definable 
disjoint family. 

Let $\Delta(u,v)$ be the set of atomic formulas over the signature of $\CC$ with free variables $u$ and $v$.
Let $k,t$ be the numbers as in the definition of a restrained class, applied to the formulas $y_1\le x\le y_2$ and $z_1\le x\le z_2$.
Thus, for every two convex partitions $\cal R$ and $\cal L$ of $\str S\in\CC$ into at least $t$ parts each
there are $R\in\cal R$ and $L\in\cal L$ with 
\begin{align}\label{eq:something}
|\Types[\Delta](R/L)|< k
\end{align}

Let $M$ be the adjacency matrix of $\str S$,
that is, the $\str S\times\str S$ matrix 
whose entry at position $(a,b)$ 
is the atomic type of $(a,b)$ in $\str S$
(up to a certain encoding as described in~\cref{subsec:graph-theory}, which is however irrelevant here).
We  show that $M$ has grid rank less than
$p:=2^{k2^{|\Delta|}}$.

Suppose that $\str S$ has grid rank 
at least $p$.
Then there are divisions $\cal L$ of the rows of $M$ and $\cal R$ of the columns of $M$,
 into  $p>k$ parts each,
 such that each zone has at least $p$ different rows or at least $p$ different columns.
 As the size of the alphabet of $M$ is at most $2^{|\Delta|}$, this implies that each zone has at least $k$ different rows.
 Hence, for all  $L\in \cal L$ and $R\in\cal R$ 
 we have that $|\Types[\Delta](L/R)\ge k|$, which contradicts~\cref{eq:something}.
\end{proof}

We now define a notion which generalizes the notion of avoiding certain patterns. In this case, rather than defining patterns which encode  all permutations, it is more convenient to define patterns which encode grids, in the following way.

Fix any signature $\Sigma$ and  a first-order formula $\phi(\tup x,\tup y, z)$, where $\tup x$ and $\tup y$ are sets of variables and $z$ is a single variable. An \emph{$m\times n$ grid}  defined by $\phi$  in a structure $\str S$
is a triple of sets $A\subset \str S^{\tup x}$, $B\subset  \str S^{\tup y}$ and $C\subset\str S$ with $|A|=m$, $|B|=n$ and $|C|=m\times n$, such that the relation 
\[\setof{(\tup a,\tup b,c)\in A\times B\times C}{\str S\models\phi(\tup a,\tup b,c)}\]
is the graph of a bijection from  $A\times B$ to $C$. 
More explicitly, for each $c\in C$ there is a unique pair $(\tup a,\tup b)\in A\times B$ such that $\str S\models \phi(\tup a,\tup b,c)$,
and conversely, for each $(\tup a,\tup b)\in A\times B$ there is a unique $c\in C$ such that $\str S\models\phi(\tup a,\tup b,c)$.

\begin{definition}[Defining large grids]\label{def:grids}
	A class of structures $\CC$ \emph{defines large grids} if there is a formula $\phi(\tup x,\tup y,z)$ such that for all $n\in\N$, $\phi$ defines an  $n\times n$ grid in some structure $\str S\in \CC$. 
\end{definition}

Intuitively,  if $\CC$ defines large grids then  the product of two sets $A\times B$ can be represented by a set of single elements $C$ in some structure $\str S\in\CC$. Hence 
an arbitrary relation $R\subset A\times B$ can be represented by some subset of $C$, so $\CC$ is  monadically independent. This is 
expressed by the following.
% more formally proved below.

\begin{lemma}\label{lem:no-large-grids}
	If $\CC$ defines large grids then $\CC$ is monadically independent.
\end{lemma}
\begin{proof}Suppose $\CC$ defines large grids and let $R\subset A\times B$ be a binary relation between two finite sets $A,B$. Then there is a structure $\str S\in\CC$ 
	and sets $A',B',C$ such that  $\phi(\tup x,\tup y,z)$ defines a grid $(A',B',C)$ in  $\str S$,
	and $|A|=|A'|$ and $|B|=|B'|$.
	Fix an arbitrary bijections $f\from A\to A'$ and $g\from B\to B'$.
	Let $U\subseteq C$ be such that for all $c\in C$,
\[c\in U\qquad \Leftrightarrow\qquad  \str S\models \phi(f(a),g(b),c)\text{ for some $a\in A$ and $b\in B$ with $R(a,b)$}.\]
As $\phi$ defines a bijection between $A\times B$ and $C$, it follows that 
	 for all $a\in A$ and $b\in B$, 
	 \[R(a,b)\quad\Leftrightarrow\quad \str S\models 
	 \exists z.\phi(f(a),g(b),c)\land U(z).\]
	 
Consider the monadic lift $\CC^+$ of $\CC$ 
which consists of all structures $\str S\in\CC$ expanded with a unary predicate $U$ which is interpreted as a finite set.
The above shows that the formula $\exists z.\phi(\tup x,\tup y,z)\land U(z)$ is independent over  $\CC^+$. Hence, $\CC$ is not monadically dependent.
\end{proof}

The following result
generalizes the implications~
\cref{it:bd-gr}$\rightarrow$
\cref{it:ramsey}$\rightarrow$
\cref{it:m-nip} from Theorem~\ref{thm:equiv}
to arbitrary classes of structures -- 
 finite or infinite, ordered or unordered, and over an arbitrary signature. 
It also involves a notion which we call \emph{1-dimensionality}, 
which is a model-theoretic notion originating from Shelah (it is called \emph{finite satisfiability dichotomy} in~\cite{braunfeld2021characterizations}).
This is a central tool in the study of monadically dependent classes.
It is defined in the following section, in terms of a variant of forking independence -- a key concept in stability theory, generalizing e.g. independence in vector spaces or algebraic independence.

\begin{theorem}\label{thm:restrained}
  Let $\CC$ be a class of structures over a fixed signature.\\
  Then the implications
	\cref{it:reg-trans}$\leftrightarrow$\cref{it:reg-md}$\leftrightarrow$\cref{it:reg-lg}$\leftrightarrow$\cref{it:reg-1-dim}$\rightarrow$\cref{it:reg-reg} hold among the following conditions.
	\begin{enumerate}[label={(\arabic*)},ref=\arabic*]
		\item\label{it:reg-trans} $\CC$ does not transduce the class of all graphs,
		\item\label{it:reg-md} $\CC$ is monadically dependent,
		\item\label{it:reg-lg} $\CC$ does not define large grids,
		\item\label{it:reg-1-dim} $\CC$ is $1$-dimensional (cf. Def.~\ref{def:1-dim}),
		\item\label{it:reg-reg} $\CC$ is restrained.
	\end{enumerate}
\end{theorem}

The equivalence
 $\cref{it:reg-trans}\leftrightarrow\cref{it:reg-md}$ is Theorem~\ref{thm:baldwin-shelah}, due to Baldwin and Shelah.
The implication 
$\cref{it:reg-md}\rightarrow\cref{it:reg-lg}$ is Lemma~\ref{lem:no-large-grids}.
The implication $\cref{it:reg-lg}\rightarrow\cref{it:reg-1-dim}$ is due to Shelah (cf. Prop.~\ref{prop:fs dichotomy}).
The implication~\cref{qit:1-dim}$\rightarrow$\cref{qit:mNIP} has been recently proved by Braunfeld and Laskowski~\cite{braunfeld2021characterizations}.
 Our contribution is the implication~\cref{it:reg-1-dim}$\rightarrow$\cref{it:reg-reg}.

We believe this result may be of independent interest, and possibly of broader applicability than 
just in the context of ordered structures. For example, by Theorem~\ref{thm:restrained}, all graph classes of bounded twin-width (without an order) and all interpretations of {nowhere-dense} classes are restrained.
Conversely, every class of structures which is not restrained  defines large grids.

\medskip

Theorem~\ref{thm:restrained} allows us to 
provide further, model theoretic characterizations of 
hereditary classes of finite, ordered, binary structures of bounded twin-width:

\begin{theorem}
	Let $\CC$ be a hereditary class of finite, ordered, binary structures. Then
	 the following conditions are equivalent:
	\begin{enumerate}[label={(\arabic*)},ref=\arabic*]
		\item\label{qit:trans} $\CC$ does not transduce the class of all finite graphs,
		\item\label{qit:mNIP} $\CC$ is monadically dependent,
		% \item\label{qit:fmNIP} $\CC$ is finitely monadically NIP,
		\item\label{qit:grids} $\CC$ does not {define large grids},% (cf. Def.~\ref{def:grids}),
		\item\label{qit:1-dim} $\CC$ is $1$-dimensional,% (cf. Def.~\ref{def:1-dim}),
		\item\label{qit:restrained} $\CC$ is a {restrained} class,% (cf. Def.~\ref{def:restrained}),
		\item\label{qit:tww} $\CC$ has bounded twin-width,
		\item\label{qit:dependent} $\CC$ is dependent.
	\end{enumerate}
\end{theorem}
\begin{proof}
The implications $\cref{qit:trans}\rightarrow\cref{qit:mNIP}\rightarrow\cref{qit:grids}\rightarrow\cref{qit:1-dim}\rightarrow\cref{qit:restrained}$ follow from Theorem~\ref{thm:restrained}.
The implication~\cref{qit:restrained}
$\rightarrow$\cref{qit:tww} is proved in Proposition~\ref{prop:restrained-tww}.
The implication~\cref{qit:tww}$\rightarrow$\cref{qit:trans} is by~\cite{twin-width1} (cf. \cref{thm:tww implies mNIP}). This proves the equivalence of the first six items.

The implication~\cref{qit:mNIP}$\rightarrow$\cref{qit:dependent} is immediate,
whereas 
the implication~\cref{qit:dependent}$\rightarrow$\cref{qit:tww} is by Theorem~\ref{thm:main},\cref{git:interp},\cref{git:tww}.
\end{proof}

\medskip
To prove Theorem~\ref{thm:restrained},
it remains to prove  that every 1-dimensional class is restrained. First, we need to define 1-dimensionality.

\subsection{1-dimensionality}
We now introduce a wholly model-theoretic notion which can be used to characterize bounded twin-width,
but also arbitrary monadically dependent classes of structures.
For this, we first recall some basic notions from model theory.

By  a \emph{model} we mean a structure which is typically infinite, as opposed to the structures considered earlier, which were typically finite.
We give a brief account of basic notions from model theory in Appendix~\ref{app:model-theory}, although they are not needed to follow the main text below.

The \emph{elementary closure} of a class of structures
$\CC$ is the class of all models $\str M$ that satisfy every sentence $\phi$ that holds in all structures $\str S\in\CC$. 
In particular, if $\CC$ does not define large grids,
then neither does its elementary closure. This is because for any fixed $n\in\N$ the existence of an $n\times n$-grid defined by a fixed formula $\phi(\tup x,\tup y,z)$ can be expressed by a first-order sentence $\phi'$
which existentially quantifies $(|\tup x|+|\tup y|)\cdot n+n^2$ variables, corresponding to sets $A,B,C$ of $\tup x$-tuples, $\tup y$-tuples and single vertices,
and then checks that $\phi$ defines a bijection between $A\times B$ and $C$.

By the compactness theorem (cf. Thm.~\ref{thm:compactness}), if $\CC$ defines large grids, then its elementary closure contains a structure that defines a grid $(A,B,C)$
with $A$ and $B$ of arbitrarily large infinite cardinalities.

% For infinite models $\str M$, it makes sense to consider the  condition that $\str M$ (treated as a singleton class)
% does not define infinite grids (that is, $\omega\times\omega$-grids).
% By the compactness theorem (cf. Theorem~\ref{thm:compactness}), $\CC$ does not define large grids if and only if every $\str M$ in the elementary closure of $\CC$ does not define infinite grids.

\begin{definition}[Elementary extension]
	Let $\str M,\str N$ be two models.
	Then $\str N$ is an \emph{elementary extension} of $\str M$,
	written $\str M\prec \str N$, 
	if the domain of $\str M$ is contained in the domain of $\str N$, and  for  every 
	formula $\phi(\tup x)$ and tuple $\tup a\in\str M^{\tup x}$ of elements of $\str M$,
	\[\str M\models\phi(\tup a)\text{\quad if and only if \quad}\str N\models \phi(\tup a).\]	
\end{definition}
 In other words, it doesn't matter if we evaluate formulas in $\str M$ or in $\str N$. In particular, $\str M$ and $\str N$ satisfy the same sentences.

 \medskip
 A \emph{formula $\phi(\tup x)$ with parameters} from $C\subset \str N$
is a formula using constant symbols denoting elements from $C$.
Such a formula can be evaluated in $\str N$
on a tuple $\tup a\in\str N^{\tup x}$, as expected.
Note that if $\str M\prec \str N$ and $\phi(\tup x)$ is a formula with parameters from $\str N$ and $\tup a\in\str M^{\tup x}$ then it is not necessarily the case that $\str M\models \phi(\tup a)$ if and only if $\str N\models \phi(\tup a)$, although this does hold for formulas with parameters from $\str M$.
\medskip
\begin{definition}[Independence]\label{def:independence}
	Let $\str M$ be a model and $\str N$ its elementary extension. For a tuple $\tup a\in \str N^{\tup x}$ and a set $B\subset \str N$ say that $\tup a$ is \emph{independent} from $B$ over $\str M$, denoted $\tup a\ind[\str M] B$,
	if for every formula $\phi(\tup x)$ with parameters  from  $B\cup \str M$ such that $\str N\models \phi(\tup a)$ there is some $\tup c\in \str M^{\tup x}$ such that $\str N\models \phi(\tup c)$.
\end{definition}
Abusing notation, if $B$ is enumerated by a tuple $\tup b$, 
then we may write $\tup a\ind[\str M]\tup b$.
We write $\nind[\str M]$ for the negation of the relation
$\ind[\str M]$.

% Equivalently,
% $\tp(\tup a/BC)$ is finitely satisfiable in $C$
% (cf. Appendix~\ref{sec:basic}).
% If $\tup a\ind[C]{} B$ does not hold, then we may say that $B$ \emph{isolates} $\tup a$ from $C$.
% Say that $B$ \emph{isolates} $\tup a$ from $C$ if there is some formula $\phi$ with parameters from $B$ which is satisfied by $\tup a$ but not by any $\tup c\in C^{\tup x}$
% Equivalently, there is a formula $\phi(\tup x,\tup y)$ such that $B$ $\phi$-isolates $\tup a$ from $C$.
% Write  $\tup a\ind[C]{}B$ if $\tup a\ind[C]\phi B$ for all $\tup y$ and $\phi(\tup x,\tup y)$.

\begin{example}\label{ex:order-ind}
	Let $\str N$ be $(\mathbb R,\le)$
    and let $\str M$ be 
    the union  of the open intervals 
    $]0,1[$ and $]8,9[$, equipped with the relation $\le$.
	Then $\str M\prec \str N$. This is easy to derive from the fact that $(\mathbb R,\le)$ has quantifier elimination, that is, every formula $\phi(\tup x)$ is equivalent to a quantifier-free formula. 
	\begin{figure}[h!]
		\centering
		\includegraphics[scale=0.35,page=1]{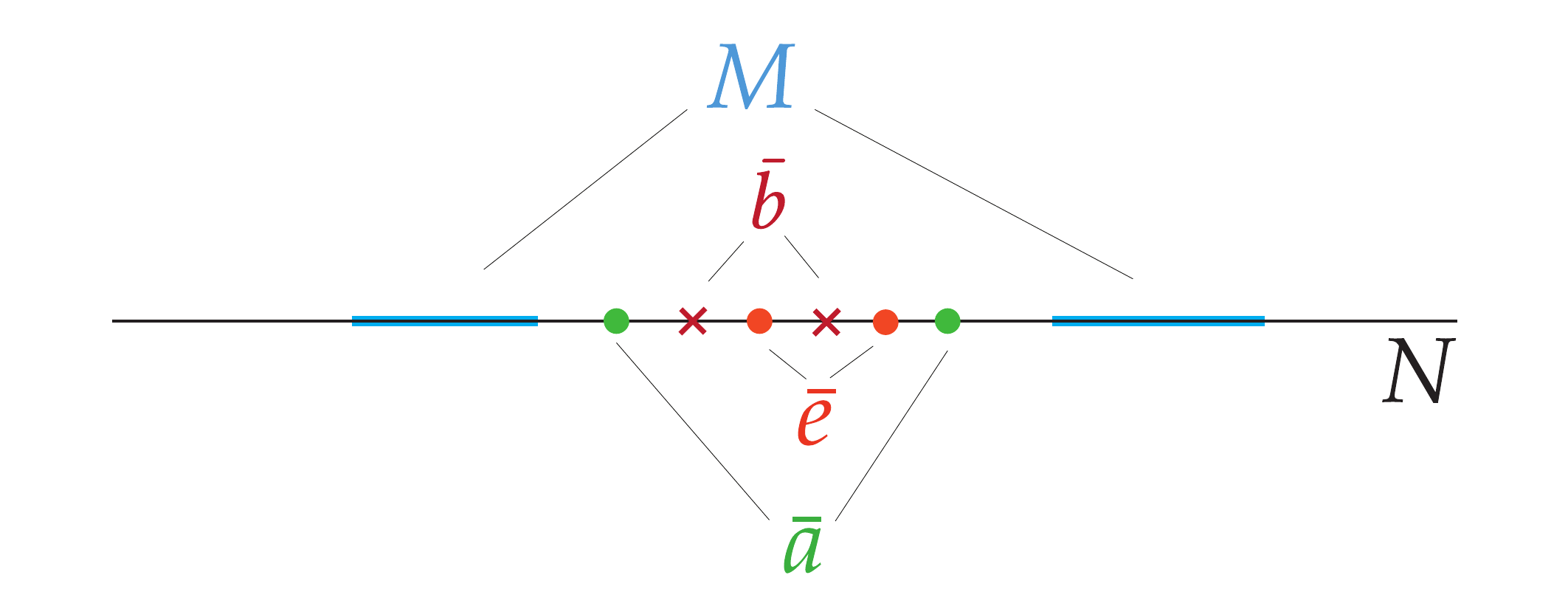}
		\caption{The structures $\str M\prec \str N$, a set $B$ and two tuples, $\tup a,\tup e$, with 
		$\tup a\ind[\str M]B$ and $\tup e\nind[\str M]B$.
		}\label{fig:ind1}
	\end{figure}
	Figure~\ref{fig:ind1} illustrates independence over  $\str M$. 
	
	% This is analysed more precisely  below.

	% Assume more generally that $\str M\prec \str N$ and $\str N=(\mathbb R,\le)$.
	% Let $B\subset \str N$. 

	% Let  $\tup a\in\str N^{\tup x}$. 
	% Then $\tup a\ind[\str M]\tup b$ if and only if every open interval $I$ with endpoints in $B\cup\str M\cup\set{-\infty,+\infty}$ containing an element of $\tup a$ also contains an element of $\str M$.
	
	% For the left-to-right implication, suppose $I$ has endpoints $m_1,m_2$ and contains  $\tup a(x)$ for some $x\in\tup x$. Then the formula $\phi(x)=m_1<x<m_2$ is satisfied by $\tup a$, so it is satisfied by some element $m\in\str M$, and so $m\in I$.

	% The converse implication can be shown using the fact that every formula $\phi(\tup x)$ is equivalent to a boolean combination of inequalities. We omit the details.
% 	We show the converse implication. Suppose that $\tup a$ satisfies a formula $\phi(\tup x)$ with parameters from $\str M\tup b$ that is not satisfied by any tuple $\tup c\in\str M^{\tup x}$. By quantifier elimination, we may assume $\phi(\tup x)$ has no quantifiers, so is a boolean 
% combination of conditions of the form $c<x$ or $x=c$ for $x\in\tup x$ and $c\in\str M\tup b$. Rewriting $\phi$, we may assume it is a disjunction of conjunctions of formulas of the form $c<x$, $c=x$ or $c>x$, for $x\in\tup x$ and $c\in\str M\tup b$. 
	
% 	suppose $\phi(\tup a)$ holds for some formula 

	% \begin{quote}
	% 	$4\ind[\str M]5$,\quad
    %     $6\ind[\str M]5$,\quad
	% 	$(4,6)\nind[\str M]5$,\quad
	% 	$4\ind[\str M](5,6)$,\quad
    %     $(2,3)\ind[\str M]5$,\quad $(6,7)\ind[\str M]5$.
	% \end{quote}
\end{example}

\begin{example}Let
	$\str M\prec \str N$.
	Then  $\tup a\ind[\str M]{}\str M$ for every $\tup a\in\str N^{\tup x}$
	(cf. Lemma~\ref{lem:satisfaction}).
\end{example}

\begin{definition}[1-dimensionality]\label{def:1-dim}
	A model $\str M$ is \emph{1-dimensional}
	if for every $\str M\prec \str N$, tuples  $\bar a,\bar b$ of elements of $\str N$ and $c\in\str N$ a single element, 
if $\tup a\ind[\str M]{}\tup b$ then  $\tup ac\ind[\str M]{}\tup b$ or $\tup a\ind[\str M]{}\tup bc$.
A class $\CC$ of structures is \emph{1-dimensional} if every model in the elementary closure of $\CC$ is 1-dimensional.
\end{definition}

\begin{example}
	Any total order $(X,\le)$ is 1-dimensional.
As an illustration, in the situation in Fig.~\ref{fig:ind1}, consider the tuples $\tup a,\tup b$ marked therein. Then $\tup a\ind[\str M]\tup b$.
Let $c\in\str N$. If $c$ belongs to the interval $]b_1,b_2[$
then $\tup a\ind[\str M]\tup bc$. Otherwise, $\tup ac\ind[\str M]\tup b$.
\end{example}

\begin{example}\label{ex:grid-2-dim}
	Let $\str N=(\mathbb R\times \mathbb R, \sim_1,\sim_2)$ where for $i=1,2$, the relation $\sim_i$ denotes equality of the $i$th coordinates.
	Let $\str M$ be the induced substructure of $\str N$ with domain $I\times I$ for some infinite subset $I\subset \str N$. Then $\str M\prec \str N$. In the situation depicted in Fig.~\ref{fig:ind2}, $a\ind[\str M]b$ but both $ac\nind[\str M]b$ and $a\nind[\str M]bc$. So $\str M$ is not $1$-dimensional. 

	\begin{figure}[h!]
		\centering
		\includegraphics[scale=0.35,page=2]{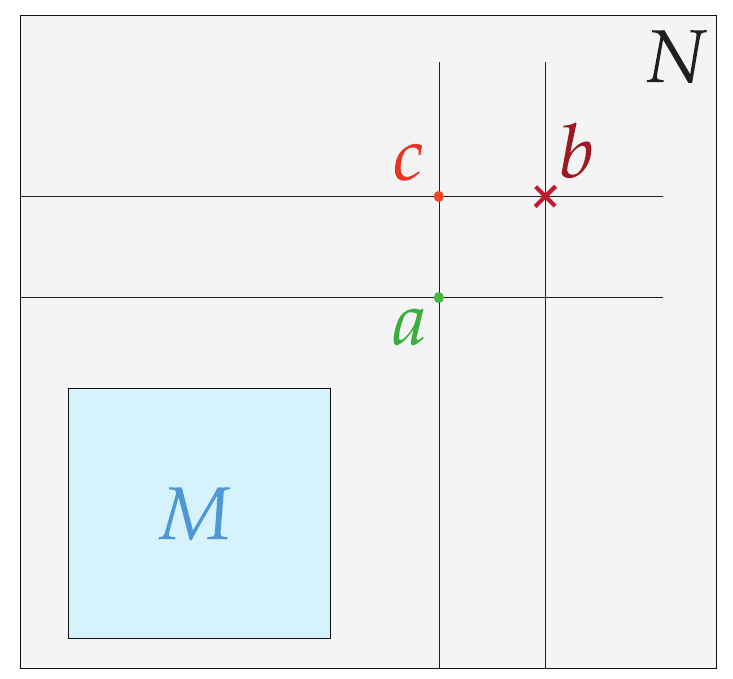}
		\caption{The structures $\str M\prec \str N$ and elements $a,b,c\in\str N$ with 
		$a\ind[\str M]b, ac\nind[\str M]b$ and $a\nind[\str M]bc$.
		}\label{fig:ind2}
	\end{figure}
\end{example}

The  following result  is essentially \cite[Lemma 2.2]{shelah:hanfnumbers} (see also~\cite{braunfeld2021characterizations}).

\begin{proposition}\label{prop:fs dichotomy}
	If a  model $\str M$ does not define large grids
    then it is 1-dimensional. 
\end{proposition}
In particular, every class $\CC$ that does not define large grids is 1-dimensional, proving the implication~\cref{it:reg-lg}$\rightarrow$\cref{it:reg-1-dim} in Theorem~\ref{thm:restrained}.

\subsection{Proof of Theorem~\ref{thm:restrained}}\label{sec:mt-proof}
In this section we prove that every 1-dimensional class is restrained, proving the remaining implication~\cref{it:reg-1-dim}$\rightarrow$\cref{it:reg-reg} in Theorem~\ref{thm:restrained}.

Let $\CC$ be a class which is not restrained. We  show that $\CC$ is not 1-dimensional.
We first construct a structure $\str M$ in the elementary closure of $\CC$ that witnesses that $\CC$ is not restrained in a convenient way.

\medskip

\begin{lemma}\label{lem:preparation}
	Suppose $\CC$ is not restrained.
	Then there exist:
    \begin{itemize}
		\item  formulas $\phi(z,\tup x),\psi(z,\tup y),\theta(\tup u,\tup v)$,
        \item  a structure $\str M$ in the elementary closure of $\CC$, 
        \item an elementary extension $\str N$ of $\str M$,
        \item tuples $\tup a_0,\tup a_1\in \str N^{\tup x}$ and $\tup b_0\in\str N^{\tup y}$,
    \end{itemize}
    such that the following properties hold:
    \begin{enumerate}
        \item\label{it:same-types}  $\tp(\tup a_0/\str M)=\tp(\tup a_1/\str M)$,
        \item\label{it:many-types} 
        the set $\Types[\theta](A/B)$ is infinite, where $A=\phi(\str N,\tup a_1)$ and $B=\psi(\str N,\tup b_0)$,
        \item\label{it:ind-base} $\tup a_1 \ind[\str M]\tup a_{0}\tup b_{0}$,
        \item\label{it:inconsistent}  $\phi(z;\tup a_0)\land \phi(z,\tup a_1)$ is not satisfiable in $\str N$,
        \item\label{it:no-sol} $\psi(z;\tup b_0)$ is not satisfiable in $\str M$.
    \end{enumerate}    
\end{lemma}

The proof of the lemma is a  standard application of basic tools from model theory: compactness, (mutually) indiscernible sequences and Morley sequences, which are recalled in Appendix~\ref{app:model-theory}. The proof of Lemma~\ref{lem:preparation} is in Appendix~\ref{app:preparation-lemma}. Using the lemma, we now show that $\CC$ is not $1$-dimensional.

% \begin{proposition}\label{prop:disjoint families ind}
%     Let $\str M$ be as in Lemma~\ref{lem:preparation}. Then $\str M$ is not $1$-dimensional.

%     Suppose $\str M$ is $1$-dimensional and
%     let $\str M\prec \str N\prec \str U$.
% Let $\phi(x;\tup y)$, $\psi(x,\tup z)$ be  formulas and $\bar a_\omega\in \str U^{\tup y}$ 
% satisfy \[\bar a_\omega\ind[\str M]\str N.\] Let  two sequences $\setof{\bar a_i}{i<\omega}\subset \str N^{\tup y}$ and $\setof{\bar b_j}{j<\omega}\subset \str N^{\tup z}$ be given such that:
% \begin{itemize}
%     \item $\setof{\bar b_j}{j<\omega}\subset \str N^{\tup z}$ is indiscernible over $\str M$,
%     \item $\tup a_i$ and  $\tup a_\omega$ have equal types over $\str M\bar a_{<i} \bar b_{<\omega}$, for $i<\omega$.
%     \item The families $(\phi(x;\bar a_i))_{i<\omega}$ and $(\psi(x;\bar b_j))_{j<\omega}$ are both pairwise inconsistent.
% \end{itemize}

% Then for any formula $\theta(\bar u;\bar v)$
% and  $i,j<\omega$
% the set $\Types[\theta](A/B)$ is finite, where $A = \phi(\str N;\bar a_i)$ and $B = \psi(\str N;\bar b_j)$.
% \end{proposition}
Through the rest of Section~\ref{sec:mt-proof},
we fix $\str M$ and $\str N$ as in Lemma~\ref{lem:preparation} and use the notation from the lemma.

\begin{claim}\label{claim:counting}
There is an elementary extension $\str N'$ of $\str N$,
and a tuple $\tup a$ of elements in $\phi(\str N';\tup a_1)$, a tuple $\tup b$ of elements in $\psi(\str N';\tup b_0)$ 
such that $\tup a\nind[\str M]\tup b.$
    %  More precisely, there are $\tup a\in\phi(\str N';\tup a_1)^{\tup u}$ and
	%  $\tup b_+,\tup b_-\in\psi(\str N';\tup b_0)^{\tup v}$
	%  such that:
	%  \begin{itemize}
	% 	 \item $\theta(\tup a;\tup b_+)\land \neg\theta(\tup a;\tup b_-)$ holds,
	% 	 \item $\theta(\tup v;\tup b_+)\land\neg \theta(\tup v;\tup b_-)$ is  not satisfiable in $\str M$.
	%  \end{itemize}
    \end{claim}
    \begin{proof}
		Let $\hat\theta(\tup v;\tup u)=\theta(\tup u;\tup v)$.
		As $\Types[\theta](A/B)$ is infinite by~\cref{it:many-types}, 
		so is $\Types[\hat\theta](B/A)$.

		Let $E_{\theta}(\tup u;\tup u')$ be the formula defining the equivalence relation on $\psi(\str N;\tup b_0)^{\tup u}$ such that 
		\[E_\theta(\tup b,\tup b')\quad\Leftrightarrow\quad\tp^{\hat\theta}(\bar b/\phi(\str N;\tup a_1)) = \tp^{\hat\theta}(\bar b'/\phi(\str N;\tup a_1))\qquad\qquad\text{for  $\bar b,\bar b'\in\psi(\str N;\tup b_0)^{\tup u}$}.\] More precisely, 
        \[E_\theta(\tup v;\tup v')\quad\equiv\quad \psi(\tup v;\tup b_0)\land \psi(\tup v';\tup b_0)\land \forall \tup u. \phi(\tup u;\tup a_1)\implies (\theta(\tup u;\tup v)\iff \theta(\tup u';\tup v)).\]
		Hence the formula  $E_\theta$ defines infinitely many classes in $\str N$.
        By compactness  there is an elementary extension $\str N'$ of $\str N$ such that  $E_\theta$ induces more than $2^{{|\str M|}}$ equivalence classes in~$\str N'$ (cf. Lemma~\ref{lem:many-classes}). 
		As there are at most $2^{|\str M|}$ distinct $\hat\theta$-types in $\str M$, there exist $\tup b',\tup b''\in\psi(\str N';\tup b_0)$ such that $\tp^{\hat\theta}(\tup b'/\str M)=\tp^{\hat\theta}(\tup b''/\str M)$ and $\neg E_\theta(\tup b',\tup b'')$ holds in $\str N'$. Hence $\theta(\tup a;\tup b')\triangle\theta(\tup a;\
		\tup b'')$ holds in $\str N'$ for some $\tup a\in \phi(\str N';\tup a_1)^{\tup u}$. The claim follows for $\tup b=\tup b'\tup b''$.
% By assumption, 
%         \[\phi(\str N') \ind[\str M] \psi(\str N').\]
%         By Corollary~\ref{cor:counting fs}, $\Types[\theta](\psi(\str N')/\phi(\str N'))$ has cardinality at most $2^{{|\str M|}}$, 
%         which is in contradiction with the number of classes of $E_\theta$ in $\str N'$.
    \end{proof}

\begin{claim}\label{cl:2}
	Let $\str N'$ be an elementary extension of $\str N$ and  $a\in \phi(\str N';\tup a_1)$. Then $\tup a_1\nind[\str M]\tup a_0 a$, as witnessed by the formula
	\begin{align}\label{eq:zeta}\tag{$\ast$}
		\zeta(\tup y;a,\tup a_0) := \phi(a;\tup y)\land \neg\exists x. \phi(x; \bar y)\land \phi(x; \bar a_0).
		\end{align}		
\end{claim}
\begin{proof}
We have that $\zeta(\tup a_1;a,\tup a_0)$ holds since $\phi(x;\tup a_1)$ and $\phi(x;\tup a_0)$ are inconsistent by~\cref{it:inconsistent}. Assume that there is some $\tup a' \in \str M^{\tup y}$ such that $\zeta(\tup a';a,\tup a_0)$ holds. Then \[\exists x.\phi(x;\tup a')\wedge \phi(x;\tup a_1)\] holds in $\str N'$, as witnessed by $x=a$. By property~\cref{it:same-types} and as $\tup a'$ is in $\str M$, this implies that \[\exists x.\phi(x;\tup a')\wedge \phi(x;\tup a_0)\] holds in $\str N'$, contradicting  $\zeta(\tup a';a,\tup a_0)$. Thus  $\zeta(\tup y;a,\tup a_0)$ is not satisfiable in $\str M$. 
In particular, $\tup a_1\nind[\str M]\tup a_0a$, proving the claim.\end{proof}

Fix tuples $\tup a,\tup b$ as in Claim~\ref{claim:counting}.
We show that if $\str M$ is $1$-dimensional then
$\tup a_1\tup a\ind[\str M]\tup a_0\tup b_0\tup b$, implying $\tup a\ind[\str M]\tup b$, contrary  to Claim~\ref{claim:counting}.

\begin{claim}\label{cl:3}Suppose $\str M$ is $1$-dimensional, $\str N'$ is an elementary extension of $\str N$, and let $\tup a$ be a tuple in $\phi(\str N';\tup a_1)$ and $\tup b$  a tuple in $\psi(\str N';\tup b_0)$. Then 
    \[\tup a_1\tup a\ind[\str M]\tup a_{0}\tup b_0\tup b.\]    
	In particular, $\tup a\ind[\str M]\tup b$.
\end{claim}

\begin{proof}
	 We show the statement by induction on the length of $\tup a$ and $\tup b$. The base case where $\tup a$ and $\tup b$ are empty is given by  property~\cref{it:ind-base} in Lemma~\ref{lem:preparation}.

Assume we know the result for $\tup a, \tup b$ and we want to add an element $b\in\psi(\str N';\tup b_0)$ to $\tup b$. By 1-dimensionality, one of the two cases holds: \[\tup a_1\tup a\,  b \ind[\str M]\tup a_0\tup b_0\tup b\qquad\text{or}\qquad \tup a_1\tup a   \ind[\str M]\tup a_0\tup b_0\tup b\,b.\] 
Note that property~\cref{it:no-sol} implies $b\nind[\str M]\tup b_0$, excluding the first case,
% The first case is excluded by Claim~\ref{cl:1}, 
so the second case must hold, as required.

% We show that $\psi(x;\bar b_0)$ is not satisfiable in $\str M$, which excludes the first case.
% If $\psi(m;\bar b_0)$ for some $b\in\str M$, by indiscernibility of $(\bar b_j:j<\omega)$ over $\str M$, we have that $\psi(m;\bar b_j)$ holds for all $j$. This contradicts pairwise inconsistency of $(\psi(x;\bar b_j):j<\omega)$.
% Hence the second case occurs, as required.

Now assume we want to add $a\in\phi(\str N';\tup a_1)$ to $\tup a$. 
% Consider the formula \[ \zeta(\tup y;a,\tup a_0) := \phi(a;\tup y)\land \neg\exists x. \phi(x; \bar y)\land \phi(x; \bar a_0).\] Then $\zeta(\tup a_1;a,\tup a_0)$ holds since $\phi(x;\tup a_1)$ and $\phi(x;\tup a_0)$ are inconsistent. Assume that there is some $\tup a' \in \str M^{\tup y}$ such that $\zeta(\tup a';a,\tup a_0)$ holds. Then \[\exists x.\phi(x;\tup a')\wedge \phi(x;\tup a_1)\] holds in $\str N'$, as witnessed by $x=a$. By indiscernibility of the sequence $(\tup a_i:i<\omega)$ over $\str M$, \[\exists x.\phi(x;\tup a')\wedge \phi(x;\tup a_0)\] holds in $\str N'$. But that contradicts  $\zeta(\tup a';a,\tup a_0)$. Thus  $\zeta(\tup y;a,\tup a_0)$ is not satisfiable in $\str M$. 
% In particular, $\tup a_1\nind[\str M]\tup a_0a$.
By  1-dimensionality, 
\[\tup a_1 \tup a\, a \ind[\str M]\tup a_0\tup b_0 \tup b \qquad\text{or}\qquad \tup a_1\tup a\ind[\str M]\tup a_0\tup b_0\tup b a,\]
but the second possibility is excluded by Claim~\ref{cl:2}, and the first one concludes the inductive step. This proves Claim~\ref{cl:3}.
\end{proof}

\medskip
Since there are $\tup a\in \phi(\str N';\tup a_1)$ and $\tup b\in\psi(\str N';\tup b_0)$ with
$\tup a\nind[\str M]\tup b$ by Claim~\ref{claim:counting}, Claim~\ref{cl:3} implies $\str M$ is not $1$-dimensional.
This finishes the proof of Theorem~\ref{thm:restrained}.

\appendix

\section{Ordered matchings interpret all graphs}\label{app:matchings}
Here we give a proof of the following.
\begin{replemma}{lem:matchings are dependent}The class $\mathscr M$ of ordered matchings efficiently interprets the class of all graphs.
\end{replemma}
\begin{proof}We prove that the class of all  graphs interprets in $\mathscr M$. 
  Before describing the interpretation of graphs in ordered matchings, we show how the ordered matching $M_G$ corresponding to an ordered  graph $G$ is constructed, in polynomial time. 
	
	Let $G$ be an ordered graph with vertices $v_1<\dots<v_n$ and edges $e_1,\dots,e_m$.
	For $i\in[n]$ and $1\leq j\leq {\rm d}(v_i)$ we define $\epsilon_{i,j}$ as the index of the $j$th edge incident to $v_i$.
	The left vertices of $M_G$ will be
	(in order)	$v_1,\dots, v_n,x,{e_1'}^-,e_1',{e_1'}^+,\dots,{e_m'}^-,e_m',{e_m'}^+$, and $y'$.
	The right vertices of $M_G$ will be 
	(in order) $x'$, $\epsilon_{n,1},\dots,\epsilon_{n,{\rm d}(v_n)},v_n',\dots$, $\epsilon_{1,1},\dots,\epsilon_{1,{\rm d}(v_n)},v_1'$,$y,e_m,\dots,e_1$. The matching $M_G$ matches $v_i$ and $v_i'$, $x$ and $x'$, $y$ and $y'$, $e_i'$ and $e_i$, and finally $\epsilon_{i,j}$ either with ${e_{\epsilon_{i,j}}'}^-$ or ${e_{\epsilon_{i,j}}'}^+$, depending on whether $v_i$ is the smallest or biggest incidence of $e_{\epsilon_{i,j}}$  (see~\cref{fig:M2G}).

\begin{figure}[ht]
	\begin{center}
		\includegraphics[width=\textwidth]{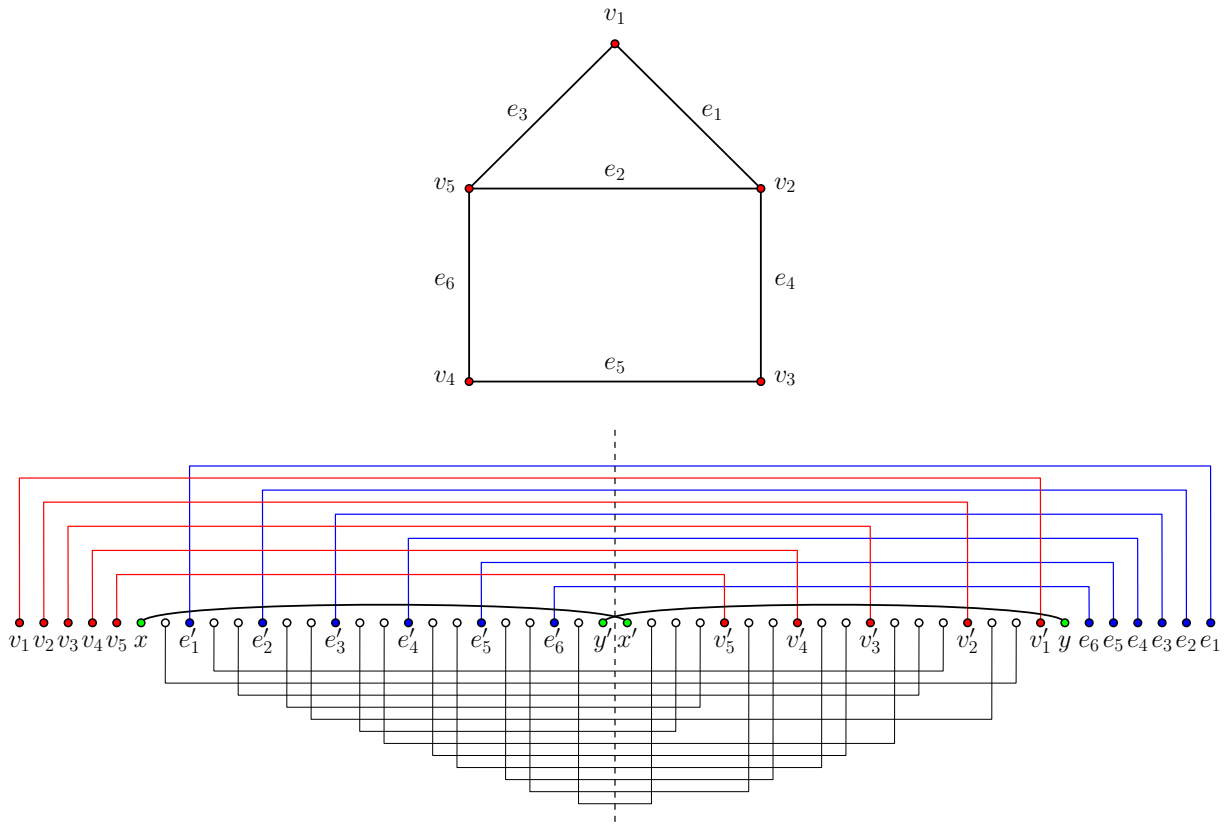}
	\end{center}
	\caption{Encoding of a graph in a matching. }
	\label{fig:M2G}
\end{figure}

We now prove that there is a simple interpretation $\mathsf G$, which reconstructs $G$ from $M_G$.
First note that $x'$ is definable as the minimum vertex adjacent to a smaller vertex, and $y'$ is definable as the maximum vertex adjacent to a bigger vertex. Also, $x$ is definable from $x'$ and $y$ is definable from $y'$.
%(To define these vertices using a formula defining a function in any ordered graph, $x$ may for example be defined as the minimum neighbor of $x'$ if $x'$ as at least a neighbor, and $x'$ otherwise.) 
Now we can define $v_1,\dots,v_n$ to be the vertices smaller than $x$, ordered with the order of $M_G$. 
	Two vertices $v_i<v_j<x$ are adjacent 
	in the interpretation if there exists an element $e_k>y$ adjacent to a vertex $e_k'$ preceded in the order by an element ${e_k'}^-$ 
	and followed in the order by an element ${e_k'}^+$ with the following properties:
	${e_k'}^-$ is adjacent to a vertex $z^-$ strictly between the neighbor $v_i'$ of~$v_i$ and the neighbor of the successor of $v_i$ in the order and, similarly, ${e_k'}^+$ is adjacent to a vertex $z^+$ strictly between the neighbor $v_j'$ of~$v_j$ and the neighbor of the successor of $v_j$ in the order.
\end{proof}

\section{Reducing the model checking problem for matrices to structures}\label{app:reduction}
In this appendix, we prove:
\begin{replemma}{lem:reducing matrices to structures}
  Let $\CC$ be a class of ordered binary structures and let $\mathcal M=\setof{M(\str A)}{\str A\in\CC}$ be the class of adjacency matrices of structures in $\CC$.  
  Then there is an $\FPT$~reduction of the \FO~model checking problem for $\cal M$ to the \FO~model checking problem for $\CC$.
  \end{replemma}
  \begin{proof}
    Let $\phi$ be a sentence in the signature of $\cal M$. Hence, $\phi$ may use the unary relations $R$ and $C$ denoting the rows/columns respectively, the total order $<$, as well as binary relations $E_\tau(x,y)$, for each atomic type $\tau \in A_\Sigma$. 
    Without loss of generality, by rewriting the sentence $\phi$ to an equivalent one if necessary, we may assume that the variables are partitioned into \emph{row variables} or a \emph{column variables}, and for every subformula of $\phi$ of the form $\exists x.\psi$, the subformula $\psi$ is of the form $R(x)\land\psi'(x)$ if $x$ is a row variable, and $\psi$ is of the form 
    $C(x)\land\psi'(x)$ if $x$ is a column variable.
  
    In this case, construct from $\phi$ a formula $\phi'$ in the signature of $\CC$ as follows:
  \begin{itemize}
    \item replace each atom $R(x)$ or $C(x)$ by $\top$ (denoting true),
    \item if $x$ is a row variable and $y$ is a column variable then replace each atom $x=y$ or $y=x$ by $\bot$ (denoting false),
     each atom $E_\tau(x,y)$ or $E_\tau(y,x)$ by $\tau(x,y)$,
    \item if $x,y$ are both row variables or both column variables, then replace each atom $E_\tau(x,y)$ by $\bot$.
  \end{itemize}
  Then $\phi$ holds in $M(\str A)$ if and only if $\phi'$ holds in $\str A$. This yields the reduction.
  \end{proof}

\section{Model theoretic preliminaries}\label{app:model-theory}

\subsection{Basic notions from model theory}\label{sec:basic}

\subparagraph{Models and theories.} In model theory, structures are called \emph{models},
and we will therefore denote them $\str M,\str N$, etc.
They will typically be infinite.

A (first-order) \emph{theory} is a set $T$ of sentences over a fixed signature. A \emph{model of a theory} $T$ is a model $\str M$ (finite or not) which satisfies all the sentences in $T$, which is denoted $\str M\models T$. We say that $T$ \emph{has a model} if there is some model $\str M$ of $T$.

\emph{The theory} of a class of structures $\CC$ 
is the set $T$ of all sentences $\phi$ such that $\str S\models \phi$ for all $\str S\in \CC$. Trivially, every structure in $\CC$ is a model of $T$, but typically, $T$ has also other models. 
Those can be constructed using the compactness theorem:

\begin{theorem}[Compactness of first-order logic]\label{thm:compactness}
    Let $T$ be a theory such that  every finite subset $T'\subset T$ has a model. Then $T$ has a model.
\end{theorem}
For example, let $\CC$ be a class of structures over a signature $\Sigma$, and assume that 
$\CC$ contains structures of arbitrarily large finite size.
Then the models of the theory of $\CC$ also include infinite models of arbitrarily large cardinality.
To see this, consider the theory $T$ of $\CC$
and let $\Sigma'$ extend the signature of $\CC$ 
by an arbitrary set  $C$ of constant symbols. For $c,d\in C$,
let $\phi_{cd}$ be the $\Sigma'$-sentence $c\neq d$.
Then $T\cup\setof{\phi_{cd}}{c,d\in C,c\neq d}$ satisfies the assumption of the compactness theorem, so it has a model $\str M$, and this model has at least the cardinality of $C$.

\subparagraph{Elementary extensions.}
Let $\str M,\str N$ be two models such that the domain of $\str M$ is contained in the domain of $\str N$.
Then $\str N$ is an \emph{elementary extension} of $\str M$,
written $\str M\prec \str N$,
if  for  every 
formula $\phi(\tup x)$ and tuple $\tup a\in\str M^{\tup x}$ of elements of $\str M$,
\[\str M\models\phi(\tup a)\text{\quad if and only if \quad}\str N\models \phi(\tup a).\] In other words, it doesn't matter if we evaluate formulas in $\str M$ or in $\str N$.

A typical way of constructing an elementary extension of $\str M$ 
is by considering the following theory, called the \emph{elementary diagram} of $\str M$. Let $\Sigma$ be the signature of $\str M$,
and let $\Sigma'=\Sigma\cup\str M$, where the elements of $\str M$ are viewed as constant symbols.

     For a $\Sigma$-formula $\phi(\tup y)$ and tuple $\tup a\in \str M^{\tup y}$ write $\phi(\tup a)$ for the $\Sigma'$-sentence 
     obtained by replacing the variables in $\tup y$ by constants in $\str M$, according to $\tup a$. 
     Let $T$ be the $\Sigma'$-theory consisting of all sentences $\phi(\tup a)$, for all $\Sigma$-formulas $\phi(\tup x)$ and tuples $\tup a$ such that $\str M\models\phi(\tup a)$.

Pick a model $\str N'$ of $T$, and 
let $\str N$ denote the $\Sigma$-structure obtained from $\str N'$ by forgetting the constants in $\str M\subset \Sigma'$.
    The interpretation of the constants $m\in\str M$ of $\Sigma'$ in $\str N'$ yields a function $i\from \str M\to \str N$.
    By the definition of $T$,
    for any formula $\phi(\tup y)$ and tuple  $\tup a\in\str M^{\tup y}$, 
    \[\str M\models\phi(\tup a)\text{ if and only if }\str N\models \phi(i(\tup a)).\]
    Therefore, we may view (identyfing each $m\in\str M$ with  $i(m)\in\str N$) the $\Sigma$-structure $\str N$ as an elementary extension of $\str M$.
    
    Reassuming, models of the elementary diagram of $\str M$ correspond precisely to elementary extensions of $\str M$. In particular, by extending the elementary diagram of $\str M$ by an arbitrary set of constants, from compactness we get that $\str M$ has elementary extensions of arbitrarily large cardinality (unless $\str M$ is finite). More generally, we have the following.

\begin{lemma}\label{lem:many-classes}
    Let $\str M$ be a model and let $\alpha(\tup x,\tup x')$ be a formula with $|\tup x|=|\tup x'|$ defining an equivalence relation in $\str M$ with infinitely many classes. Then for every cardinality $\mathfrak n$
    there is an elementary extension $\str N\succ\str M$
    in which 
    $\alpha$ defines an equivalence relation with at least $\mathfrak n$ equivalence classes. 
\end{lemma}
\begin{proof}
To simplify notation, assume that $|\tup x|=|\tup x'|=1$.
The  case of $|\tup x|=|\tup x'|=k>1$ proceeds similarly, or can be deduced from the case $k=1$ by extending the domain of $\str M$ by $\str M^k$ and the $k$ projection functions.

Let $\alpha(x,x')$ be  formula defining an equivalence relation ${\sim}$ in $\str M$ with infinitely many classes.
Let $\Sigma$ be the signature of $\str M$.
Fix any set of constants $C$ and let $\Sigma'$ extend  $\Sigma$  by   $C\cup\str M$,
where all the added elements are  constant symbols. For any $c,d\in C$ consider the $\Sigma'$-sentence $\phi_{cd}=\neg\alpha(c, d)$. 
Let $T$ be the $\Sigma'$-theory consisting of:
\begin{itemize}
    \item the sentences $\phi_{cd}$, for all $c\neq d$ in $C$,
    \item the elementary diagram of $\str M$.
\end{itemize} 
We show  that every $T'\subset T$ containing finitely many sentences of the form $\phi_{cd}$ has a model. Let $C'\subset C$ be the finite set of constants appearing in the sentences $\phi_{cd}\in T'$. Let $\str M'$  be the model $\str M$ together with 
each constant $c$ in $\str M\subset \Sigma'$ interpreted as the corresponding element  $c\in \str M$,
and
constants in $C'$ interpreted as pairwise $\sim$-inequivalent elements of $\str M$, and constants in $C\setminus C'$ interpreted as arbitrary elements of $\str M$. This can be done, since there are infinitely many pairwise $\sim$-inequivalent elements in $\str M$. This shows that $T'$ has a model.

By compactness, $T$ has a model $\str N'$. 
This model can be seen as an elementary extension of $T$ together with 
a set of $|C|$ elements which are pairwise inequivalent with respect to the equivalence relation defined by $\alpha$ in $\str N$. Since $C$ was taken arbitrary, this proves the lemma.
\end{proof}

% \subparagraph{Types.}
% Let $\tup x$ be a set of variables.
% A \emph{type} $p(\tup x)$ with variables $\tup x$ 
% is a set of formulas $\phi(\tup x)$ with free variables $\tup x$. A tuple $\tup a\in\str S^{\tup x}$ satisfies a type $p(\tup x)$ if $\str S\models\phi(\tup a)$ for all $\phi\in p(\tup x)$.
% Every tuple $\tup a\in\str S^{\tup x}$ has a largest type,
% namely the set of all formulas $\phi(\tup x)$ such that $\str S\models\phi(\tup a)$. This type is called \emph{the type} of $\tup a$ in $\str S$, and is denoted $\tp(\tup a)$.

% A type with variables $\tup x$ over a signature $\Sigma$ may be seen as a theory over the signature $\Sigma\cup\tup x$,
% where the elements of $\tup x$ are treated as constants. 

\subparagraph{Parameters.}
Let $\str M$ be a model over a signature $\Sigma$ and let $A\subset \str M$ be a set of elements. We may view $\str M$ 
as a model over a signature $\Sigma\cup A$, where the elements of $A$ are seen as constant symbols, interpreted in $\str M$ in the expected way: a constant $a\in A$ is interpreted as the element $a\in\str M$. We call the elements of $A$ \emph{parameters} in this context.
A $\Sigma$-formula \emph{with parameters} from $A$ is a formula over the signature $\Sigma\cup A$.

\subparagraph{Types.}
A \emph{type} with variables $\tup x$ and parameters from $A$,
or a \emph{type over} $A$
is a set $p$ of formulas $\phi(\tup x)$ with parameters from $A$. We may write $p(\tup x)$ to indicate that $p$ has variables~$\tup x$.

If $p(\tup x)$ is a type over $A$ and $B\subset A$ then $p|B$ denotes the subset of $p$ consisting of all formulas with parameters from $B$.
If $\tup b\in\str M^{\tup x}$ is a tuple of elements of $\str M$ then \emph{the type}  of $\tup b$ \emph{over} $A$  in $\str M$ 
is the set of formulas $\phi(\tup x)$ with parameters from $A$ that are satisfied by $\tup b$ in $\str M$. 
This type is denoted $\tp(\tup b/A)$ or $\tp_{\tup x}(\tup b/A)$. Note that 
$\tp(\tup b/A)$ is related to the notion of $\theta$-types
as follows, for every formula $\theta(\tup x;\tup y)$ and tuple $\tup a\in A^{\tup y}$:
\[\theta(\tup x;\tup a)\in \tp(\tup b/A)
\iff \tup a\in \tp^\theta(\tup b/A).\]
In particular, $\tp(\tup b/A)$ is uniquely determined by $\setof{\tp^\theta(\tup b/A)}{\theta(\tup x;\tup y)\text{ is a formula}}$.

A type $p(\tup x)$ is \emph{satisfiable in a set} $C$ if there is some tuple $\tup c\in C^{\tup x}$ which satisfies all the formulas in $p$. 
A type $p(\tup x)$ with parameters from $A\subset \str M$ is \emph{satisfiable} if it is satisfiable 
in some elementary extension $\str N$ of $\str M$.
By compactness, this is equivalent to saying that for any finite conjunction $\phi(\tup x)$ of formulas in $p(\tup x)$ we have $\str M\models\exists \tup x.\phi(\tup x)$.

A type $p(\tup x)$ with parameters from $A$ is 
\emph{complete} if it is satisfiable and for every formula
$\phi(\tup x)$ with parameters from $A$, either $\phi$ or $\neg \phi$ belongs to $p$. Equivalently, $p(\tup x)$ is the type over $A$ of  some tuple $\tup b\in \str N^{\tup x}$, for some elementary extension $\str N$ of $\str M$.
We sometimes say that a type is \emph{partial} to emphasise that it may not be complete.
We denote the set of complete types with variables $\tup x$ and parameters from $A$ by $\St[\tup x](A)$ or simply $\St(A)$, if  $\tup x$ is understood from the context.
Note that we have ommitted the model $\str M$ from the notation. Indeed,
 if $\str M'$ is a model containing the parameters $A$
and satisfying the same sentences with parameters from $A$
as $\str M$, then $\str M$ and $\str M'$ have 
identical sets of complete types $p(\tup x)$ with parameters from $A$.
 Hence, $\St[\tup x](A)$ does not depend on $\str M$, but only on the set of sentences satisfied by $A$ in $\str M$.

\subsection{Finite satisfiability}
A (partial) type $p(\tup x)$ with parameters from $A$ is \emph{finitely satisfiable} in $C$ if every finite subset $p'\subset p$  is satisfiable in $C$. 
Note that $\tup a\ind[\str M]B$ (cf. Def.~\ref{def:independence}) if and only if $\tp(\tup a/\str MB)$ is finitely satisfiable in~$\str M$.

\begin{lemma}\label{lem:satisfaction}
    A type $p(\tup x)$ with parameters from $\str M$ 
    is finitely satisfiable in $\str M$ if and only if it is satisfiable. Consequently, $\tup a\ind[\str M]\str M$ for all $\tup a$ in an elementary extension of $\str M$.
\end{lemma}
\begin{proof}
    For the right-to-left implication, assume that $p$ is satisfied by some tuple $\tup c\in\str N^{\tup x}$ for some elementary extension $\str N$ of $\str M$. Pick a finite $p'\subset p$, and suppose $p'=\set{\phi_1,\ldots,\phi_k}$. 
    Consider the formula $\psi:=\phi_1\land\cdots\land \phi_k$.
    Note that $\psi$ may use some parameters from $\str M$.
    So we may write $\psi$ as $\psi=\psi'(\tup x,\tup a)$ 
    where $\psi'(\tup x,\tup z)$ is a formula and $\tup a\in \str M^{\str z}$. 
    
    The formula $\exists_{\tup x}\psi'(\tup x,\tup a)$ holds in $\str N$, as witnessed by $\tup c$. As $\str N$ is an elementary extension of $\str M$, this formula also holds in $\str M$.
    So there is some $\tup m\in \str M$ satisfying $\psi'(\tup b,\tup a)$. Therefore, $p'$ is satisfied by $\tup m$ in $\str M$, proving that $p$ is finitely satisfiable in $\str M$.

    \medskip
    The left-to-right implicaiton is a basic application of the compactness theorem. 
    
    Consider the signature $\Sigma'=\Sigma\cup \str M\cup\tup x$
    extending $\Sigma$ by constant symbols for each element of $\str M$ and each variable in $\tup x$.
    Let $T$ be the theory over $\Sigma'$ consisting of:
    \begin{itemize}
        \item For every formula $\phi(\tup x)\in p$, 
         the $\Sigma'$-sentence obtained from $\phi(\tup x)$ by viewing each parameter $a\in \str M$ as the constant $a\in\str M\subset \Sigma'$, and each variable $x\in\tup x$ as the constant  $x\in \tup x\subset \Sigma'$. 
        \item the elementary diagram of $\str M$.
    \end{itemize}

    Then every finite subset $T'$ of $T$ has a model.
    Indeed, let $p'$ be the set of formulas $\phi(\tup x)$ which occur (as $\Sigma'$-sentences) in $T'$.
    Since $p(\tup x)$ is finitely satisfiable in $\str M$,
    $p'(\tup x)$ is satisfied by some tuple $\tup m\in\str M^{\tup x}$. The pair $(\str M,\tup m)$ may be seen as a $\Sigma'$-structure, where a constant $m\in\str M$ is interpreted by the corresponding element of $\str M$,
    and a constant $x\in\tup x$ is interpreted as $\tup m(x)$.
    Then $(\str M,\tup m)$ is a model of $T'$.

    By compactness, $T$ has a model $\str N'$. 
    This model can be seen as an elementary extension $\str N$ of $\str M$ together with a tuple $\tup c\in\str N^{\tup x}$ of elements (obtained by the interpretation of the constants $\tup x$ in $\str N'$), such that 
$\str N\models \phi(\tup c)$ for every formula $\phi(\tup x)\in p$.
Hence, $\tup c$ satisfies $p(\tup x)$ in $\str N.$
\end{proof}

\subparagraph{Finite satisfiability  and filters.}
Recall that a \emph{filter} on a set $U$ is a nonempty set $F\subset P(U)$ that is closed under taking supersets (if $A\subset B$ then $A\in F$ implies $B\in F$), under binary intersections, and does not contain the empty set. A filter is an \emph{ultrafilter} if for every $A\subset U$, either $A\in F$ or $U\setminus A\in F$. Every filter is contained in some ultrafilter, by the Kuratowski-Zorn lemma.

% Note that a complete type $p(\tup x)\in S^{\tup x}_{\str M}(A)$ with parameters from $A\subset \str M$ 
% is the same thing as an ultrafilter on the set of $A$-definable subsets of $\str M^{\tup x}$, where
% an \emph{$A$-definable} subset of $\str M^{\tup x}$ is a subset of the form $\phi(\str M)$, for some formula $\phi(\tup x)$ with parameters from $A$. 

\medskip

Let $\str N$ be a model, $A\subset\str N$ be a set and $\tup x$ be a set of variables. Fix a filter $F$  on $A^{\tup x}$.
The \emph{average (partial) type}  over $\str N$ is the partial type denoted $\Av_{F}(\tup x)$ such that for every formula $\phi(\tup x)$ with parameters from $\str N$,
\[ \phi(\tup x)\in \Av_{F}(\tup x) \iff \{\tup a \in A^{\tup x} : \str N \models \phi(\tup a)\} \in F.\]

This is a consistent partial type: if say $\phi_1(\tup x),\ldots,\phi_n(\tup x) \in \pi(\tup x)$, then since any finitely many elements of $F$ have non-empty intersection, there is $\tup a\in A^{\tup x}$ which satisfies the conjunction $\phi_1(\tup x)\wedge \cdots \wedge \phi_n(\tup x)$. Hence this conjunction is consistent, indeed we have shown that $\Av_F(\tup x)$ is finitely satisfiable in $A$.

If $F$ is an ultrafilter on $A^{\tup x}$, then $\Av_F(\tup x)$ is a complete type: for every formula $\phi(\tup x)$, either $\phi(\tup x)\in \Av_F(\tup x)$ or $\neg \phi(\tup x)\in \Av_F(\tup x)$.

\begin{lemma}\label{lem:fs filters}
Let $\pi(\tup x)$ be a partial type, then $\pi(\tup x)$ is finitely satisfiable in $A$  if and only if there is a filter $F$ on $A^{\tup x}$ such that $\pi(\tup x) \subseteq \Av_F(\tup x)$.
\end{lemma}
\begin{proof}
We have already observed that $\Av_F(\tup x)$ is finitely satisfiable in $A$. Conversely, assume that $\pi(\tup x)$ is finitely satisfiable in $A$, then define $F_0\subseteq  P(A^{\tup x})$ by: $F_0 = \{\phi(A) : \phi(\tup x)\in \pi(\tup x)\}$. The fact that $\pi(\tup x)$ is finitely satisfiable in $A$ implies that any finitely many elements of $F_0$ have non-empty intersection. Let $F$ be the filter generated by $F_0$. Then we have $\pi(\tup x)\subseteq \Av_F(\tup x)$.
\end{proof}

\begin{lemma}
Let $p(\tup x)\in \St(\str N)$ be a complete type, then $p(\tup x)$ is finitely satisfiable in $A$ if and only if there is an ultrafilter $F$ on $A^{\tup x}$ such that $p(\tup x) = \Av_F(\tup x)$.
\end{lemma}
\begin{proof}
We have already seen that if $F$ is an ultrafilter on $A^{\tup x}$, then $\Av_F(\tup x)$ is a complete type over $\str N$, which is finitely satisfiable in $A$. Conversely, if $p(\tup x)\in \St(\str N)$ is finitely satisfiable in $A$, then by the previous lemma, there is a filter $F_0$ on $A^{\tup x}$ such that $p(\tup x)\subseteq \Av_{F_0}(\tup x)$. Extend $F_0$ to an ultrafilter $F$ on $A^{\tup x}$. Then $p(\tup x)\subseteq \Av_{F_0}(\tup x)\subseteq \Av_{F}(\tup x)$. But since $p(\tup x)$ is a complete type, one cannot add any formulas to it without making it inconsistent. Since $\Av_F(\tup x)$ is consistent, we must have $p(\tup x)= \Av_F(\tup x)$.
\end{proof}

\begin{lemma}\label{lem:extending fs}
Let $\pi(\tup x)$ be a partial type finitely satisfiable in $A$. Then there is a complete type $p(\tup x)\in \St(\str N)$ finitely satisfiable in $A$ which extends $\pi(\tup x)$.
\end{lemma}
\begin{proof}
Let $F$ be a filter on $A^{\tup x}$ such that $\pi(\tup x)\subseteq \Av_F(\tup x)$. Let $F'$ be an ultrafilter extending $F$ and let $p(\tup x) = \Av_{F'}(\tup x)$. Then $p$ is finitely satisfiable in $A$ and extends $\pi$.
\end{proof}

\begin{lemma}\label{lem:fs invariant}
Let $p(\tup x)\in \St(\str N)$ be finitely satisfiable in $A$. Then $p$ is \emph{$A$-invariant}, that is: for any formula $\phi(\tup x;\tup y)$ and tuples $\tup b, \tup b'\in \str N^{\tup y}$, we have:
\[ \tp(\tup b/A)  = \tp(\tup b'/A) \Longrightarrow \phi(\tup x;\tup b)\in p \leftrightarrow \phi(\tup x;\tup b')\in p.\]
\end{lemma}
\begin{proof}
If $\tp(\tup b/A) = \tp(\tup b'/A)$, then the formula $\phi(\tup x;\tup b)\triangle \phi(\tup x;\tup b')$ has no solution in $A$. Since $p$ is finitely satisfiable in $A$ that formula cannot be in $p$. Hence as $p$ is a complete type, the formula $\phi(\tup x;\tup b)\leftrightarrow \phi(\tup x;\tup b')$ is in $p$ as required.
\end{proof}

% \begin{lemma}\label{lem:number of fs types}
% Let $A \subseteq \str M$ be infinite and let $\phi(\tup x;\tup y)$ a formula, then there are at most $2^{2^{|A|}}$ many complete $\phi$-types over $\str M$ that are finitely satisfiable in $A$.
% \end{lemma}
% \begin{proof}
% Any complete $\phi$-type over $\str M$ that is finitely satisfiable in $A$ is equal to the average $\phi$-type of some ultrafilter on $A^{\tup x}$. There are at most $2^{2^{|A|}}$ many ultrafilters on $A^{\tup x}$. Hence the result follows.
% \end{proof}

\subsection{Indiscernible sequences}

\begin{definition}
Let $\str M$ be a structure and $A\subseteq \str M$. Let $I$ be a linear order. A sequence $(\tup a_i:i\in I)$ of tuples of $\str M$ is \emph{indiscernible over $A$} if for any $n<\omega$ and indices \[i_1 < \cdots < i_n \qquad i'_1 < \cdots < i'_n\] in $I$, we have \[ \tp(a_{i_1},\ldots, a_{i_n}/A) = \tp(a_{i'_1},\ldots,a_{i'_n}/A).\]
\end{definition}

Another way to state this is that the sequence $(\tup a_i:i\in I)$  is {indiscernible over $A$} if for any $n<\omega$, indices \[i_1 < \cdots < i_n \qquad i'_1 < \cdots < i'_n\] in $I$ and formula $\theta(\tup x_1,\ldots,\tup x_n)$ with parameters in $A$, we have  \[ (1) \qquad \str M\models \theta(\tup a_{i_1},\ldots, \tup a_{i_n}) \leftrightarrow \theta(\tup a_{i'_1},\ldots,\tup a_{i'_n}).\]

If $\Delta$ is a set of formulas with parameters in $A$, we will say that the sequence $(\tup a_i:i\in I)$ is \emph{$\Delta$-indiscernible} if $(1)$ holds for each $\theta$ in $\Delta$. If $\Delta$ and $I$ are both finite, then this is expressible by a single first order formula.

\begin{definition}
Two sequences $(\tup a_i:i\in I)$ and $(\tup b_j : j\in J)$ are \emph{mutually indiscernible} over $A$ if $(\tup a_i:i\in I)$ is indiscernible over $A\cup \{\tup b_j :j\in J\}$ and $(\tup b_j:j\in J)$ is indiscernible over $A\cup \{\tup a_i:i\in I\}$.
\end{definition}

An equivalent definition is that the sequences $(\tup a_i:i\in I)$  and $(\tup b_j : j\in J)$ are \emph{mutually indiscernible} over $A$  if for any $n<\omega$, indices \[i_1 < \cdots < i_n \qquad i'_1 < \cdots < i'_n\] and \[ j_1 < \cdots < j_n \qquad  j'_1 < \cdots < j'_n\]  in $I$ and any formula $\theta(\tup x_1,\ldots,\tup x_n;\tup y_1,\ldots,\tup y_n)$ with parameters in $A$, we have  \[ (2) \qquad \str M\models \theta(\tup a_{i_1},\ldots,\tup  a_{i_n};\tup b_{j_1},\ldots,\tup b_{j_n}) \leftrightarrow \theta(\tup a_{i'_1},\ldots,\tup a_{i'_n};\tup b_{j'_1},\ldots,\tup b_{j'_n}).\]

If $\Delta$ is a set of formulas with parameters in $A$, we will say that the sequences $(\tup a_i:i\in I)$ and $(\tup b_j:j<\omega)$ are  \emph{mutually $\Delta$-indiscernible} if $(2)$ holds for each $\theta$ in $\Delta$. If $\Delta$, $I$ and $J$ are finite, then this is again expressible by a single first-order formula.

In the following lemma, we use the notation $\Av_F | C$ to mean the restriction of the type $\Av_F$ to $C$. We also use the notation $\tup a_{<i}$ to mean $\{\tup a_j: j<i\}$.

\begin{lemma}
Let $A\subseteq B \subseteq \str M$. Let $F$ be an ultrafilter on $A^{\tup x}$. Let $I$ be a linear order and let $(\tup a_i:i\in I)$ be a sequence of tuples of $\str M$ such that:
\[\tup a_i \models \Av_F | B\tup a_{<i}.\]
Then the sequence $(\tup a_i:i\in I)$ is indiscernible over $B$.
\end{lemma}
\begin{proof}
Write $p=\Av_F$. Note that $p$ is finitely satisfiable in $A$ and a fortiori finitely satisfiable in $B$.

We prove by induction on $n$ that if $n<\omega$ and $i_1 < \cdots < i_n$, $j_1 < \cdots < j_n$ are in $I$, then $\tp(\tup a_{i_1},\ldots, \tup a_{i_n}/A) = \tp(\tup a_{j_1},\ldots,\tup a_{j_n}/A)$. For $n=1$ this follows from the fact that all $\tup a_i$ realize $\Av_F | B$, which is a complete type over $B$. Assume we know it for $n$ and take $i_1 < \cdots < i_n< i_{n+1}$, $j_1 < \cdots < j_n< j_{n+1}$ in $I$. By induction hypothesis, we have \[\tp(\tup a_{i_1},\ldots, \tup a_{i_n}/B) = \tp(\tup a_{j_1},\ldots,\tup a_{j_n}/B).\]
By Lemma \ref{lem:fs invariant}, for any formula $\theta(\tup x;\tup y_1,\ldots, \tup y_n)$ with parameters in $B$, we have:
\[ \theta(\tup x;\tup a_{i_1},\ldots,\tup a_{i_n}) \in p \iff \theta(\tup x;\tup a_{j_1},\ldots,\tup a_{j_n})\in p.\]
Now since $\tup a_{i_{n+1}} \models p| Ba_{i_1}\ldots a_{i_n}$, we have 
\[ \theta(\tup x;\tup a_{i_1},\ldots,\tup a_{i_n}) \in p \iff \str N \models \theta(\tup a_{i_{n+1}};\tup a_{i_1},\ldots,\tup a_{i_n}),\]
and similarly since $\tup a_{j_{n+1}}\models p| Ba_{j_1}\ldots a_{j_n}$, we have
\[ \theta(\tup x;\tup a_{j_1},\ldots,\tup a_{j_n}) \in p \iff \str N \models \theta(\tup a_{j_{n+1}};\tup a_{j_1},\ldots,\tup a_{j_n}).\]
Putting all of this together, we get
\[ \str N \models \theta(\tup a_{i_{n+1}};\tup a_{i_1},\ldots,\tup a_{i_n}) \iff \str N \models \theta(\tup a_{j_{n+1}};\tup a_{j_1},\ldots,\tup a_{j_n}).\]
Since the formula $\theta$ was an arbitrary formula with parameters in $B$, we deduce
\[ \tp(\tup a_{i_1},\ldots,\tup a_{i_{n+1}}/B) = \tp(\tup a_{j_1},\ldots,\tup a_{j_{n+1}}/B)\] as required.
\end{proof}

\begin{definition}
Let the type $p(\tup x)\in \St(\str M)$ be finitely satisfiable in $A\subset \str M$ and let $A\subset B\subset\str M$. A sequence $(\tup a_i :i\in I)$ of tuples in $\str M^{\tup x}$ such that $\tup a_i \models p|B\tup a_{<i}$ is called a \emph{Morley sequence} of $p$ over $B$.
\end{definition}

By the previous lemma, a Morley sequence of $p$ over $B$ is indiscernible over $B$.

\subsection{Building indiscernible sequences}

Indiscernible sequences are easy to find thanks to Ramsey's theorem.

\begin{definition}
Let $(\tup a_i:i<\omega)$ be a sequence of tuples in some structure $\str M$. A family $(\tup b_i:i\in I)$ indexed by a linear order $I$ is \emph{based on $(\tup a_i)_{i<\omega}$} if for any formula $\theta(x_1,\ldots,x_n)\in L$ and $i_1<\ldots <i_n$ in $I$, if $\str M\models \theta(b_{i_1},\ldots,\tup b_{i_n})$ then there are $j_1<\ldots <j_n<\omega$ such that $\str M\models \theta(\tup a_{j_1},\ldots,\tup a_{j_n})$.
\end{definition}

Note that if $(\tup a_i:i<\omega)$ is indiscernible and $(\tup b_i:i\in I)$ is based on it, then it is also indiscernible: indeed for any $i_1<\ldots <i_n$ in $I$ and any $j_1<\ldots <j_n<\omega$, we have \[\tp(\tup b_{i_1},\ldots,\tup b_{i_n}) = \tp(\tup a_{j_1},\ldots,\tup a_{j_n}).\]

\begin{proposition}\label{prop:ramsey}
Let $(\tup a_i:i<\omega)$ be a sequence of tuples in some structure $\str M$ and let $I$ be any linearly ordered set. There is an elementary extension $\str M\prec \str N$ and a sequence $(\tup b_i:i\in I)$ of tuples of $\str N$ that is {based on $(\tup a_i)_{i<\omega}$}.
\end{proposition}
\begin{proof}
Follows from Ramsey and compactness.
\end{proof}

We have analogues for two sequences.

\begin{definition}
Let $(\tup a_i:i<\omega)$ and $(\tup a'_i:i<\omega)$ be two sequences of tuples in $\str M$. Two families $(\tup b_i:i\in I)$, $(\tup b'_j:j\in J)$ indexed by linear orders $I$ and $J$ are \emph{based on $(\tup a_i)_{i<\omega}$ and $(\tup a'_i)_{i<\omega}$} if for any formula $\theta(\tup x_1,\ldots,\tup x_n;\tup y_1,\ldots,\tup y_m)\in L$ and $i_1<\ldots <i_n$ in $I$ and $j_1<\ldots<j_m$ in $J$, if $\str M\models \theta(\tup b_{i_1},\ldots,\tup b_{i_n};\tup b'_{j_1},\ldots,\tup b'_{j_n})$ then there are $k_1<\ldots <k_n<\omega$ and $k'_1<\ldots<k'_m<\omega$ such that $\str M\models \theta(\tup a_{k_1},\ldots,\tup a_{k_n};\tup a_{k'_1},\ldots,\tup a_{k'_m})$.
\end{definition}

Here is a finitary version of Proposition \ref{prop:ramsey} for two sequences.

\begin{lemma}\label{lem:ramsey 2}
Let $\Delta$ be a finite set of formulas and let $m,d<\omega$. Then there is some $m_* <\omega$ such that if $(\tup a_i:i<m_*)$ and $(\tup b_i:i<m_*)$ are two sequences of $d$-tuples of a structure $\str M$, then there are $(\tup a'_i:i<m)$ and $(\tup b'_i:i<m)$ subsequences of $(\tup a_i)_{i<m_*}$ and $(\tup b_i)_{i<m_*}$ respectively such that the sequences $(\tup a'_i:i<m)$ and $(\tup b'_i:i<m)$ are mutually $\Delta$-indiscernible.
\end{lemma}

\begin{proposition}\label{prop:ramsey 2}
Let $(\tup a_i:i<\omega)$ and $(\tup a'_i:i<\omega)$ be two sequences of tuples in $\str M$ and let $I,J$ be two linearly ordered sets. There is an elementary extension $\str M\prec \str N$ and sequences $(\tup b_i:i\in I)$ and $(\tup b'_j:j\in J)$ of tuples of $\str N$ which are {based on $(\tup a_i)_{i<\omega}$ and $(\tup a'_i)_{i<\omega}$}.
\end{proposition}
\begin{proof}
Follows from Lemma \ref{lem:ramsey 2} and compactness.
\end{proof}

\begin{lemma}
    Let $\str N$  be a model and $I=(\tup a_i:i<\omega)$ an indiscernible sequence of tuples of $\str N$. There is an elementary extension $\str N\prec \str N'$, a submodel $\str M\prec \str N'$ and an ultrafilter $F$ on $\str M^{\tup x}$ such that $I$ is a Morley sequence of $\Av_F$ over $\str M$.
\end{lemma}
\begin{proof}
	In an elementary extension of $\str N$, we can increase the sequence to $I + J$, where $J = (\tup b_i :i \in \mathbb Z)$ so that the sequence $I+J$ is indiscernible. Let $F_0$ be an ultrafilter on $\setof{\tup b_i}{i\in \mathbb Z}$ that contains all subsets of the form $\setof{\tup b_i}{i<n}$ for $n\in \mathbb Z$. It follows from indiscernibility that the sequence $I$ is a Morley sequence of $\Av_{F_0}$ over $\setof{\tup b_i}{i\in \mathbb Z}$. Possibly up to passing to a further elementary extension, we can find an elementary submodel $\str M$ such that $I$ is a Morley sequence of $\Av_{F_0}$ over $\str M$. One can see that by compactness, or alternatively, take $\str M_0$ any model containing $\setof{\tup a'_i}{i\in \mathbb Z}$, let $I' = (\tup a'_i:i<\omega)$ be a Morley sequence of $\Av_{F_0}$ over $M_0$. Now $I$ has the same type as $I'$ over $J$, so passing to an elementary extension, there is an automorphism $\sigma$ fixing $\setof{\tup b_i}{i\in \mathbb Z}$ pointwise and sending $I'$ to $I$. Then take $M = \sigma(M_0)$.
	
	Finally, define $F$ to be the unique ultrafilter on $M$ extending $F_0$ (so a set $A$ is in $F$ if and only if it contains a set in $F_0$). Then $I$ is a Morley sequence of $\Av_F$ over $M$.
\end{proof}

\begin{lemma}\label{lem:limit types}
Let $\str N$ be a structure and let $I=(\tup a_i:i<\omega)$ and $J= (\tup b_j:j<\omega)$ two mutually indiscernible sequences of tuples of $\str N$. There is an elementary extension $\str N\prec \str N'$, a submodel $\str M\prec \str N'$ two ultrafilters $F$ and $F'$ on $\str M^{\tup x}$ such that $I$ is a Morley sequence of $\Av_F$ over $\str M \cup \{\tup a_i:i<\omega\}$ and $J$ is a Morley sequence of $\Av_{F'}$ over  $\str M\cup \{\tup b_j:j<\omega\}$.
\end{lemma}
\begin{proof}
	The proof is very similar to the previous one. First, in an elementary extension, construct sequences $I' = (\tup a'_i:i\in \mathbb Z)$ and $J' = (\tup b'_j :j\in \mathbb Z)$ so that the two sequences $I+I'$ and $J+J'$ are mutually indiscernible. This is possible by compactness. Let $F$ be an ultrafilter on $\setof{\tup a'_i}{i\in \mathbb Z}$ containing all initial segments as in the previous proof and similarly for $F'$ on $\setof{\tup b'_j}{j\in \mathbb Z}$. Then $I$ is a Morley sequence of $\Av_F$ over $\setof{\tup a'_i}{i\in \mathbb Z} \cup \setof{\tup b_j}{j<\omega} \cup \setof{\tup b'_j}{j\in \mathbb Z}$ and $J$ is a Morley sequence of $\Av_{F'}$ over $\setof{\tup a_i}{i<\omega} \cup \setof{\tup a'_i}{i\in \mathbb Z} \cup \setof{\tup b'_j}{j\in \mathbb Z}$. One can then construct the model $M$ as above.
\end{proof}

\section{Proof of Lemma~\ref{lem:preparation}}\label{app:preparation-lemma}
A family $(\phi_i(x))_{i\in I}$ 
of formulas with parameters from $\str N$ is \emph{pairwise inconsistent} if for any distinct $i,j\in I$, the formula $\phi_i(x)\land\phi_j(x)$ has no solution in $\str N$.
For a sequence $\setof{\tup a_i}{i<\omega}$ and for $i\le \omega$ 
by $\tup a_{<i}$ denote the set of elements in all the tuples $\tup a_j$ with $j<i$.

We prove a stronger variant of Lemma~\ref{lem:preparation}.
\begin{lemma}%\label{lem:preparation}
    Suppose $\CC$ is not restrained.
    Then there exist:
    \begin{itemize}\item Formulas  
        $\phi(x,\tup y),\psi(x,\tup z)$, $\theta(\tup u,\tup v)$,
        \item  a structure $\str M$ in the elementary closure of $\CC$, 
        \item an elementary extension $\str N$ of $\str M$,
        \item a sequence $(\tup a_i:i<\omega)$ of tuples in $\str N^{\tup y}$ and a sequence $(\tup b_j:j<\omega)$ of tuples in $\str N^{\tup z}$,
    \end{itemize}
    such that the following properties hold:
    \begin{enumerate}
        \item  the tuples $\tup a_0,\tup a_1,\ldots$ have equal types over $\str M$,
		and the tuples $\tup b_0,\tup b_1,\ldots$ have equal types over $\str M$,
        \item\label{it:many-types} for all $0<i,j<\omega$, 
        the set $\Types[\theta](A/B)$ is infinite, where $A=\phi(\str N,\tup a_i)$ and $B=\psi(\str N,\tup b_i)$,
        \item\label{it:ind-base} $\tup a_i \ind[\str M]\tup a_{<i}\tup b_{<\omega}$ for $i<\omega$,
        \item the formulas $\setof{\phi(x;\tup a_i)}{i<\omega}$ are pairwise inconsistent,
        \item the formulas $\setof{\psi(x;\tup b_j)}{j<\omega}$ are pairwise inconsistent.        
    \end{enumerate}
    \end{lemma}
    It is clear that each of the properties (1)-(4)
    implies the corresponding property stated in Lemma~\ref{lem:preparation}.
    Properties (5) and (1) together imply that 
    $\psi(\str M,\tup b_j)=\emptyset$ 
    yielding property (5) in Lemma~\ref{lem:preparation}.
    We thus prove the statement above.

    \begin{proof}
Assume $\CC$ is not restrained.
Then there are formulas 
 $\phi(x,\tup y)$ and $\psi(y,\tup z)$ 
 and a finite set of formulas $\Delta(\tup u,\tup v)$ such that for every $n\in\N$ there 
 is a structure $\str S\in\CC$, two disjoint families $\cal R,\cal L$ of size $k$
 with $\Types[\Delta](L/R)\ge k$ for all $L\in\cal L$ and $R\in\cal R$, where $\cal L$ is $\phi$-definable and $\cal R$ is $\psi$-definable.
We proceed in two steps. 

\medskip
\noindent
\underline{Step 1}. 
 There is a model $\str M$ in the elementary closure of $\CC$,
indiscernible sequences 
$(\tup a_i:i<\omega)$ in $\str M^{\tup y}$ and $(\tup b_j:j<\omega)$ in $\str M^{\tup z}$,
formulas $\phi(x,\tup y)$, $\psi(x,\tup z)$ and $\theta(\tup u,\tup v)\in\Delta$ such that:
\begin{itemize}
    \item the families $\setof{\phi(x,\tup a_i)}{i<\omega}$ and $\setof{\psi(x,\tup b_j)}{j<\omega}$ are both pairwise inconsistent
    \item for each $0\le i,j<\omega$ the set 
    $\Types[\Delta](A/B)$ is infinite, where $A=\phi(\str M,\tup a_i)$ and $B=\psi(\str M,\tup b_j)$.    
\end{itemize}
\medskip

 As $\CC$ is not restrained, for every natural number $m$, we can find a structure $\str M_m\in\CC$ sequences $(\tup a^m_i:i<m)$ and $(\tup b^m_j:j<m)$ of tuples of $\str M_m$ such that:
    
    $\bullet$ the two families $\{\phi(\tup x;\tup a^m_i):i<m\}$ and $\{\psi(x;\tup b^m_j):j<m\}$ are pairwise disjoint;
    
    $\bullet$ for every $i,j<m$, the set $\Types[\Delta](A/B)$ has size at least $m$, where $A = \phi(\str M_m;\tup a^m_i)$ and $B = \psi(\str M_m;\tup b^m_j)$.
    
    Add constants to the signature to name two sequences $(\tup a_i:i<\omega)$ and $(\tup b_j:j<\omega)$. Consider the theory $T'$ in the extended language consisting of the following for every $m<\omega$:
    
    $\bullet_0$ all sentences which hold in all structures in $\CC$;
    
    $\bullet_{1,m}$ for every $i<j<m$, the two sets $\phi(\tup x;\tup a_i)$ and $\phi(\tup x;\tup a_j)$ are disjoint and the two sets $\psi(x;\tup b_i)$ and $\psi(x;\tup b_j)$ are disjoint;
    
    $\bullet_{2,m}$ the two sequences $(\tup a_i:i<m)$ and $(\tup b_j:j<m)$ are  indiscernible;
    
    $\bullet_{3,m}$ for every $i,j<m$ the set $\Types[\Delta](A/B)$ has size at least $m$, where $A = \phi(\str M;\tup a^m)$ and $B = \psi(\str M;\tup b^m)$ and $\str M$ is the considered model.
    
    Note that all those conditions are expressible by first order formulas (infinitely many in the case of $\bullet_0$ and $\bullet_{2,m}$).
    
    We claim that $T'$ is consistent. Let $T_0\subseteq T'$ be finite. Then there is $m<\omega$ such that $T_0$ only contains formulas from $T$ along with formulas $\bullet_{1,m'}$, $\bullet_{2,m'}$ and $\bullet_{3,m'}$ for $m'\leq m$. Furthermore, there is a finite set $\Gamma$ of formulas such that the formulas from $\bullet_2$ appearing in $T_0$ say at most that $(\tup a_i:i<m)$ and $(\tup b_j:j<m)$ are  $\Gamma$-indiscernible.
    
    By Lemma \ref{lem:ramsey 2}, for $m_*<\omega$ is large enough, we can find a subsequences $(\tup a'_i:i<m)$ of $(\tup a^{m_*}_i:i<m)$ and a subsequence $(\tup b'_i:i<m)$ of $(\tup b^{m_*}_i:i<m_*)$ that are  $\Gamma$-indiscernible. But then $M_{m_*}$ where we interpret the constants so as to name the two sequences $(\tup a'_i:i<m)$ of $(\tup a^{m_*}_i:i<m)$ is a model of $T_0$. Hence $T_0$ is consistent. As $T_0$ was an arbitrary finite subset of $T'$, we conclude by compactness that $T'$ is consistent.

    Let $\str M$ be a model of $T'$ and set $I =(\tup a_i:i<\omega)$ and $J=(\tup b_j:j<\omega)$ as interpreted in $\str M$. 
This yields the structure $\str M$ as described in Step 1.

% 
    
%     By Lemma \ref{lem:limit types}, there are an elementary extension $N\prec N'$, a submodel $M\prec N'$, two ultrafilters $F$ and $F'$ on $M^{\tup y}$ and $M^{\tup z}$ respectively such that $I$ is a Morley sequence of $\Av_F$ over $M \cup \{\tup a_i:i<\omega\}$ and $J$ is a Morley sequence of $\Av_{F'}$ over  $M\cup \{\tup b_j:j<\omega\}$. 
% \todo{finish}

\medskip\noindent
\underline{Step 2.} 
Apply Lemma~\ref{lem:limit types} to get an elementary extension $\str N$ of $\str M$, an elementary substructure $\str M'$ of $\str N$, such that 
$(\tup b_j:j<\omega)$ 
is a Morley sequence over  $\str M'\tup a_{<\omega}$ and 
$(\tup a_i:j<\omega)$
is a Morley sequence over $\str M'\tup b_{<\omega}$.
 In particular:
\begin{enumerate}
    \item $(\tup a_i:i>\omega)$ and 
     $(\tup b_j:j>\omega)$ are both indiscernible over $\str M'$,
     \item the families $\setof{\phi(x,\tup a_i)}{i<\omega}$ and $\setof{\psi(x,\tup b_j)}{j<\omega}$ are both pairwise inconsistent,
        \item for each $0\le i,j<\omega$ the set 
        $\Types[\Delta](A/B)$ is infinite, where $A=\phi(\str N,\tup a_i)$ and $B=\psi(\str N,\tup b_j)$,          
    \item $\tup a_i\ind[\str M']\tup a_{<i}\tup b_{<\omega}$.
\end{enumerate}
Let $A=\phi(\str N,\tup a_1)$ and $B=\psi(\str N,\tup b_j)$.
Since $\Types[\Delta](A/B)$ is infinite  and $\Delta$ is finite, there is some $\theta(\tup u,\tup v)\in \Delta$ such that 
$\Types[\theta](A/B)$ is infinite.
This finishes the proof of Lemma~\ref{lem:preparation}.
\end{proof}

\end{document}